\documentclass[a4paper]{article}

\usepackage{bbm}
 
\usepackage[all]{xy}\usepackage[latin1]{inputenc}  
\usepackage[dvips]{graphics,graphicx}
\usepackage{amsfonts,amssymb,amsmath,color,mathrsfs, amstext,bbm,dsfont}
\usepackage{amsbsy, amsopn, amscd, amsxtra, amsthm,authblk, enumerate}
\usepackage{upref}
\usepackage{geometry}
\usepackage[displaymath]{lineno}
\usepackage{bm}
\usepackage{dsfont}
\usepackage{todonotes}
\usepackage{mathtools}
\usepackage{float}

\usepackage[all]{xy}
\usepackage[latin1]{inputenc}        
\usepackage[dvips]{graphics,graphicx}
\usepackage{amsfonts,amssymb,amsmath,xcolor,mathrsfs, amstext}
\usepackage{amsbsy, amsopn, amscd, amsxtra, amsthm,authblk, enumerate}
\usepackage{upref}
\usepackage{geometry}
\usepackage[displaymath]{lineno}
\usepackage{float}

\usepackage[colorlinks,
            linkcolor=blue,
            anchorcolor=green,
            citecolor=blue
            ]{hyperref}

\numberwithin{equation}{section}

\def\e{\varepsilon}
\def\epsilon{\varepsilon}

\newcommand{\ol}{\overline}
\newcommand{\wt}{\widetilde}
\def\alb#1\ale{\begin{align*}#1\end{align*}}
\newcommand{\eqb}{\begin{equation}}
\newcommand{\eqe}{\end{equation}}

\DeclareMathOperator{\dist}{dist}

\newcommand{\bbC}{\mathbb{C}}
\newcommand{\bbD}{\mathbb{D}}
\newcommand{\bbE}{\mathbb{E}}
\newcommand{\bbH}{\mathbb{H}}

\newcommand{\bbR}{\mathbb{R}}
\newcommand{\bbP}{\mathbb{P}}

\newcommand{\QD}{\mathrm{QD}}
\newcommand{\QT}{\mathrm{QT}}
\newcommand{\LF}{\mathrm{LF}}
\newcommand{\SLE}{\mathrm{SLE}}
\newcommand{\IG}{\mathrm{IG}}
\newcommand{\Wd}{\mathrm{Weld}}
\newcommand{\Md}{{\mathcal{M}}^\mathrm{disk}}

\newtheorem{theorem}{Theorem}[section]
\newtheorem{lemma}[theorem]{Lemma}
\newtheorem{proposition}[theorem]{Proposition}
\newtheorem*{proposition*}{Proposition}
\newtheorem{corollary}[theorem]{Corollary}
\newtheorem*{corollary*}{Corollary}
\newtheorem{remark}[theorem]{Remark}
\newtheorem*{thma}{Theorem A}
\newtheorem{definition}[theorem]{Definition}
\newtheorem*{definitions*}{Definitions}

\newtheorem*{example*}{\bf Example}

\numberwithin{equation}{section}
\title{SLE partition functions via conformal welding of random surfaces}
\author{Xin Sun\thanks{Beijing International Center for Mathematical Research, Peking University, \href{mailto:xinsun@bicmr.pku.edu.cn}{xinsun@bicmr.pku.edu.cn}}  \qquad Pu Yu\thanks{Courant Institute of Mathematical Sciences, New York University, \href{mailto:py628@nyu.edu}{py628@nyu.edu}}}

\begin{document}

\maketitle

\begin{abstract}
SLE curves describe the scaling limit of interfaces from many 2D lattice models. 
Heuristically speaking, the SLE partition function is the continuum counterpart of the partition function of the corresponding discrete model. It is well known that conformally welding of Liouville quantum gravity (LQG) surfaces gives SLE curves as the interfaces.  
In this paper, we demonstrate in several settings how the SLE partition function  arises from conformal welding of LQG surfaces. The common theme is that we  conformally weld a collection of canonical LQG surfaces which produces a topological configuration with   {a random conformal structure}.   {Conditioning on the conformal modulus}, the surface after welding is described by Liouville conformal field theory (LCFT), and the density of   {the random modulus} contains the SLE partition function for the interfaces as a multiplicative factor.
The settings we treat includes the multiple SLE for $\kappa\in (0,4)$, the flow lines of imaginary geometry on the disk with boundary marked points, and the boundary Green function. These results demonstrate an alternative approach to construct and study the SLE partition function, which complements the traditional method based on stochastic calculus and differential equation. 
\end{abstract}

\section{Introduction}
Two dimensional conformal random geometry has been an active area of research in probability theory over the past two decades, and three central topics in this area are Schramm-Loewner evolution (  {$\SLE$}), Liouville quantum gravity (LQG) and Liouville conformal field theory (LCFT). $\SLE$  is an important 
family of random non-self-crossing curves   {with a single parameter $\kappa$} introduced by Schramm~\cite{Sch00},
which have been proved or conjectured to describe the scaling limits of a large class of two-dimensional lattice models at criticality, e.g.~\cite{smirnov2001critical,lawler2011conformal,schramm2009contour,chelkak2014convergence}. LQG, as introduced by Polyakov~\cite{polyakov1981quantum}, is a  canonical one-parameter family of random surfaces describing the scaling limits of random planar maps, see e.g.~\cite{LeGall13, BM17,HS19,gwynne2021convergence}. LCFT is the 2D quantum field theory which is made rigorous by~\cite{DKRV16} and follow up works
, and has natural connections with LQG (e.g.~\cite{AHS17,cercle2021unit,AHS21}.) See~\cite{Law08,vargas-dozz,GHS19,BP21,gwynne2020random,She22} for more background on these topics.

Based on~\cite{She16a,DMS14}, a fundamental connection between LQG and SLE is that the SLE curves arise as the interface of conformal welding of LQG surfaces. Conformal welding results from~\cite{She16a,DMS14} are mainly for infinite area surfaces. In~\cite{AHS20}, such results were extended to  a canonical class of finite area LQG surfaces called (two-pointed) quantum disks. Recently, it was realized in~\cite{AHS21,ASY22} that LQG surfaces defined in terms of LCFT include well-studied LQG surfaces such as the  quantum disks, and greatly extends the family of LQG surfaces that are well behaved under conformal welding.

  {As suggested by Lawler~\cite{lawler2009partition}}, a key aspect of SLE is the partition function. The SLE curves describe the scaling limit of interfaces from certain 2D  {statistical physics} models. Heuristically speaking, the SLE partition function is the continuum counterpart of the partition function of the corresponding discrete model. Based on this intuition, it has been defined rigorously   for several variants of SLE. For example, in~\cite{dubedat09partition}, the partition function was introduced for the $\SLE_\kappa(\underline \rho)$, which were later realized as flow lines in imaginary geometry~\cite{MS16a,ig4}. Another well-studied example is  the multiple SLE~\cite{bauer2005multiple,kozdron2006configurational,dubedat2007commutation,graham2007multiple,lawler2009partition,kytola2016pure}, which is a canonical coupling of multiple SLE curves.  
In addition, the SLE Green function~\cite{Lawler2015g,zhan17g,Zhan22b} can  be viewed as the partition function for SLE curves conditioned on hitting some given marked points. 
In many cases, the SLE partition function  can be expressed explicitly and are closely related to  conformal field theory~\cite{bauer2004conformal,peltola2019toward}.

In this paper, we demonstrate in several settings how the SLE partition function  arises from conformal welding of LQG surfaces. The common theme is that when the conformal welding produces a topological configuration with   {a random} conformal structure, the density of   {the random modulus} is given by the LCFT partition functions times the SLE partition function for the interfaces.  Our Theorem~\ref{thm:main} is an example of such  {a result} for the multiple $\SLE$ with  $\kappa\in (0,4)$, where LQG disks are  conformal welded together according to certain topological rules. Theorem~\ref{thm:N=2} is a similar result for the hypergeometric SLE studied in~\cite{Zhan-hSLE,wu2020hypergeometric}\footnote{The name hypergeometric SLE was originally used in~\cite{Qian-hSLE,Qian-hSLE2} to denote a broader class of SLE curves.}. 
Theorem~\ref{thm:IG} is the corresponding result for boundary imaginary geometry flow  {lines} considered in~\cite{MS16a}, where the LQG surfaces welded together are the quantum triangles introduced in~\cite{ASY22}.  
Theorem~\ref{thm:SLE-Green} is for the SLE boundary Green's function in~\cite{Zhan22b}. In Theorem~\ref{thm:green-rho}, we further extend it to the  2-point boundary Green's function for $\SLE_\kappa(\rho)$ considered by~\cite{Zhan21}.
{Our results are consistent with and enrich the picture that SLE coupled with LQG desribes the scaling limit of random planar maps decorated with statistical physics models.}
As our results show, LQG can be used to construct and study the SLE partition function, which complements the traditional method based on stochastic calculus and differential equations. 

The rest of the paper is organized as follows. In Sections~\ref{subsec:intro-sle} and~\ref{subsec:intro-lqg}, we briefly recall backgrounds on SLE, LQG, and conformal welding. In Section~\ref{subsec:intro-result} we state Theorems~\ref{thm:main}-\ref{thm:green-rho}. In Section~\ref{subsec:outlook}, we discuss future directions. We provide a detailed preliminary in Section~\ref{sec:pre} and prove the theorems in Sections~\ref{sec:proof}-\ref{sec:green}.

\subsection{Schramm-Loewner evolution}\label{subsec:intro-sle}
The chordal $\SLE_\kappa$ in the upper half plane $\bbH$ is a probability measure $\mu_\bbH(0,\infty)$ on non-crossing curves from 0 to $\infty$ which is scale invariant and satisfies the domain Markov property. The $\SLE_\kappa$ curves are simple when $\kappa\in(0,4]$, non-simple and non-space-filling for $\kappa\in(4,8)$, and space-filling when $\kappa\ge8$. 
By conformal invariance, for a simply connected domain $D$ and $z,w\in\partial D$ distinct, one can define the $\SLE_\kappa$ probability measure $\mu_D(z,w)^\#$ on $D$ by taking conformal maps $f:\bbH\to D$ where $f(0)=z$, $f(\infty)=w$. For $\rho^-,\rho^+>-2$, $\SLE_\kappa(\rho^-;\rho^+)$  {is} a classical variant of $\SLE_\kappa$, which  {was} introduced in~\cite{LSW03} and studied in e.g.~\cite{Dub05,MS16a}.

{It is natural to extend the notion of SLE to describe multiple random curves, which gives rise to the notion of \emph{multiple SLE}~\cite{bauer2005multiple}. There are two canonical formulations of the multiple SLE. The first approach is to work with the time evolution of several $\SLE_\kappa$ curves via Loewner equations, which lead to the local multiple SLEs~\cite{dubedat2007commutation,graham2007multiple,kytola2016pure}. The second formulation, as considered in~\cite{kozdron2006configurational,lawler2009partition} and known as the global multiple SLE, is to weight the law of independent $\SLE_\kappa$ curves by a term which can be written down in terms of the Brownian loop measure.}

Now we briefly recall the construction of global multiple $\SLE_\kappa$ in~\cite{peltola2019global} when $\kappa\le4$. For $N>0$, consider $N$ disjoint simple curves in $\overline{\bbH}$ connecting $1,2,...,2N\in\partial\bbH$. Topologically, these $N$ curves form a planar pair
partition, which we call a link pattern and denote by  $\alpha=\{\{i_1,j_1\},...,\{i_N,j_N\}\}$. The pairs $\{i,j\}$ in $\alpha$ are called links, and the set of link patterns with $N$ links is denoted by $\mathrm{LP}_N$. Let $D$ be a simply connected domain $D\subset\bbC$ with $2N$ marked points $x_1,...,x_{2N}$ on the boundary in counterclockwise order. 
 Let $X_0^\alpha(D;x_1,...,x_{2N})$  {be the
set of all $N$-tuples $(\eta_1,...,\eta_N)$ of disjoint continuous curves} in $\overline{D}$ which does not intersect $\partial D$ except at the starting and ending points such that for each $1\le k\le N$, $\eta_k$ links $x_{i_k}$ with $x_{j_k}$. Then the \emph{global $N$-$\SLE_\kappa$ associated to $\alpha$}, is  {the} probability measure on $(\eta_1,...,\eta_N)\subset X_0^\alpha(D;x_1,...,x_{2N})$ such that for each $1\le k\le N$, given $\eta_1,...,\eta_{k-1},\eta_{k+1},...,\eta_N$, the conditional law of $\eta_k$ is the chordal $\SLE_\kappa$ in $D\backslash\{\eta_1,...,\eta_{k-1},\eta_{k+1},...,\eta_N\}$ connecting $x_{i_k}$ and $x_{j_k}$. By~\cite{beffara2021uniqueness}, these conditional laws uniquely specify joint law of the $N$ curves, which we denote by $\mathrm{mSLE}_{\kappa,\alpha}(D;x_1,...,x_{2N})^\#$. This resampling uniqueness follows from a Markov chain mixing argument, which originated from~\cite{MS16b}, later extended to  {the case when the $N$ curves are not necessarily disjoint from each other} in~\cite{MSW19} and more recently generalized to possibly infinite measures in~\cite{Yu22,Zhan23} via Markov chain  {irreducibility} from~\cite{Meyn-Tweedie}.  The global $N$-$\SLE_\kappa$ is constructed in~\cite[Section 3]{peltola2019global} via the Brownian loop measure and the conformal restriction property of $\SLE$~\cite{LSW03}, and has been extended to radial and multiply connected domains~\cite{jahangoshahi2018multiple,healey2021n}. 

 For a link pattern $\alpha\in\mathrm{LP}_N$, the corresponding  \emph{pure partition function} $\mathcal{Z}_\alpha(D;x_1,...,x_{2N})$ can be characterized by a PDE, conformal covariance and asymptotic behavior when two marked points merge into one; see~\cite[Theorem 1.1]{peltola2019global}. It is shown in~\cite{dubedat2007commutation} that the \emph{local $N$-$\SLE_\kappa$} can be  {constructed} by the partition function $\mathcal{Z}$ in terms of Loewner driving functions, while~\cite[Theorem 1.3]{peltola2019global} proved that the global $N$-$\SLE_\kappa$ agree with local  $N$-$\SLE_\kappa$ when $\mathcal{Z}=\mathcal{Z}_\alpha$. This leads to the notion of the global $N$-$\SLE_\kappa$ as a non-probability measure, which we write  as
 $$\mathrm{mSLE}_{\kappa,\alpha}(D;x_1,...,x_{2N}) = \mathcal{Z}_\alpha(D;x_1,...,x_{2N})\cdot \mathrm{mSLE}_{\kappa,\alpha}(D;x_1,...,x_{2N})^\#.$$

The information of both  multiple SLE and its pure partition function are encoded in the measure $\mathrm{mSLE}_{\kappa,\alpha}(D;x_1,...,x_{2N})$. 
Our Theorem~\ref{thm:main} is a conformal welding result concerning $\mathrm{mSLE}_{\kappa,\alpha}(D;x_1,...,x_{2N})$. In Section~\ref{subsec:intro-result},  before stating Theorem~\ref{thm:N=2}---\ref{thm:green-rho}, we will introduce the counterpart of $\mathrm{mSLE}_{\kappa,\alpha}(D;x_1,...,x_{2N})$ for other variants SLE, where the total mass of the non-probability measure is the partition function.

\subsection{Liouville quantum gravity surfaces and conformal welding}\label{subsec:intro-lqg}
Let $D\subset\bbC$ be a simply connected domain. The Gaussian Free Field (GFF) on $D$ is the centered Gaussian process on $D$ whose covariance kernel is the Green's function~\cite{She07}. For $\gamma\in(0,2)$ and $\phi$ a variant of the GFF, the $\gamma$-LQG area measure in $D$ and length measure on $\partial D$ is roughly defined by $\mu_\phi(dz)=e^{\gamma\phi(z)}dz$ and $\nu_\phi(dx) = e^{\frac{\gamma}{2}\phi(x)}dx$, and are made rigorous by regularization and renormalization~\cite{DS11}. Two pairs $(D,h)$  and $(D',h')$ represent the same quantum surface if there is a conformal map between $D$ and $D'$ preserving the geometry.  {See Section~\ref{subsec:pre-lqg} for further details.}

For $W>0$, the two-pointed quantum disk of weight $W$, {whose law is denoted by $\Md_2(W)$}, is a quantum surface with two boundary marked points introduced in~\cite{DMS14,AHS20}, which has finite quantum area and length.  The surface is simply connected when $W\ge\frac{\gamma^2}{2}$, and consists of a chain of countably many disks when $W\in(0,\frac{\gamma^2}{2})$. For the special case $W=2$, the two boundary marked points are \emph{quantum typical} with respect to the LQG boundary length measure~\cite[Proposition A.8]{DMS14}. By sampling additional marked points from the boundary quantum length measure, we obtain multiply marked quantum disks. For $m\ge1$, we write $\QD_{m}$ for the law of the LQG disks with $m$ marked points on the boundary sampled from the LQG length measure; see Definition~\ref{def:QD} for a precise description.

As shown in~\cite{cercle2021unit,AHS21}, the quantum disks can be alternatively described in terms of LCFT. The \emph{Liouville field} $\LF_\bbH$ is an infinite measure on the space of generalized functions on $\bbH$ obtained by an additive perturbation of the GFF.  {For $(\beta_i,s_i)\in\bbR\times\partial\bbH$,\ $i=1,...,m$}, we can make sense of the measure $\LF_\bbH^{(\beta_i,s_i)_i}(d\phi) = \prod_i e^{\frac{\beta_i}{2}\phi(s_i)}\LF_\bbH(d\phi)$ via  regularization and renormalization, which leads to the notion of \emph{Liouville fields with boundary insertions}. See Definition~\ref{def-lf-H-bdry} and Lemma~\ref{lm:lf-insertion-bdry}. For $W_1,W_2,W_3>0$, as introduced in~\cite{ASY22}, the \emph{quantum triangle} of weights $W_1,W_2,W_3$ is a finite volume quantum surface with three marked points on the boundary  whose law is denoted by $\QT(W_1,W_2,W_3)$. The definition is based on Liouville fields with three boundary insertions; see Section~\ref{subsec:QT} for a precise definition. 

Now we briefly recap the \emph{conformal welding} of quantum surfaces, as studied in~\cite{She16a,DMS14,AHS20,ASY22}. Let $\mathcal{M}^1, \mathcal{M}^2$ be measures on  {the space of} quantum surfaces with boundary marked points. For $i=1,2$, fix some boundary arcs $e_1, e_2$ such that $e_i$ are boundary arcs on samples from $\mathcal{M}^i$, and define the measure $\mathcal{M}^i(\ell_i)$ via the disintegration 
$$\mathcal{M}^i = \int_0^\infty \mathcal{M}^i(\ell_i)d\ell_i$$
over the quantum lengths of $e_i$. For $\ell>0$, given a pair of surfaces sampled from the product measure $\mathcal{M}^1(\ell)\times \mathcal{M}^2(\ell)$, suppose that  they can a.s.\ be \emph{conformally welded} along arcs $e_1$ and $e_2$ according to the boundary LQG-length,  yielding a {single} surface decorated with an interface from the gluing. We write $\text{Weld}( \mathcal{M}^1(\ell), \mathcal{M}^2(\ell))$ for the law of the resulting curve-decorated surface, and let $$\text{Weld}(\mathcal{M}^1, \mathcal{M}^2):=\int_{\bbR}\, \text{Weld}(  \mathcal{M}^1(\ell), \mathcal{M}^2(\ell))\,d\ell$$ be the conformal welding of $\mathcal{M}^1,  \mathcal{M}^2$ along the boundary arcs $e_1$ and $e_2$. By induction, this definition extends to the conformal welding of multiple quantum surfaces, where we first specify some pairs of boundary arcs on the quantum surfaces, and then identify each pair of arcs according to the LQG length.  {The interfaces under such conformal welding of quantum surfaces are typically SLE-type curves. Throughout this paper, for a measure $\mathcal{M}$ on the space of quantum surfaces and a measure $\mathcal{P}$ on the space of curves with conformal invariance, we write $\mathcal{M}\times\mathcal{P}$ for the measure describing the law of the curve decorated quantum surface $(S,\eta)$ where $(S,\eta)$ is sampled from the usual product measure and $\eta$ is drawn on $S$. See Section~\ref{subsec:pre-lqg} for further details.}

Finally we briefly recall the conformal welding of quantum disks result in~\cite{AHS20}. Roughly speaking, consider the embedding of a sample from $\Md_2(W_1+W_2)$ and embed it as $(D,\phi,a,b)$ with $a,b$ being the two marked boundary points. Let $\kappa=\gamma^2$, and sample an independent $\SLE_\kappa(W_1-2;W_2-2)$ curve from $a$ to $b$ with force points at $a^-;a^+$. Then the curve decorated quantum surface $(D,\phi,\eta,a,b)/\sim_\gamma$ is equal in law to the conformal welding of a pair of weights $W_1$ and $W_2$ two-pointed quantum disk. To be more precise, we have the following result. Also see Figure~\ref{fig:diskwelding} for an illustration.
\begin{thma}[Theorem 2.2 of \cite{AHS20}]\label{thm:disk-welding}
Let $\gamma\in(0,2),\kappa=\gamma^2$ and $W_1,W_2>0$. Then there exists a constant $c:=c_{W_1,W_2}\in(0,\infty)$ such that
$$ {\Md_2(W_1+W_2)\times \SLE_\kappa(W_1-2;W_2-2)} = c\,\Wd(\Md_2(W_1),\Md_2(W_2)).  $$
\end{thma}

\begin{figure}
\centering
\includegraphics[width=0.8\textwidth]{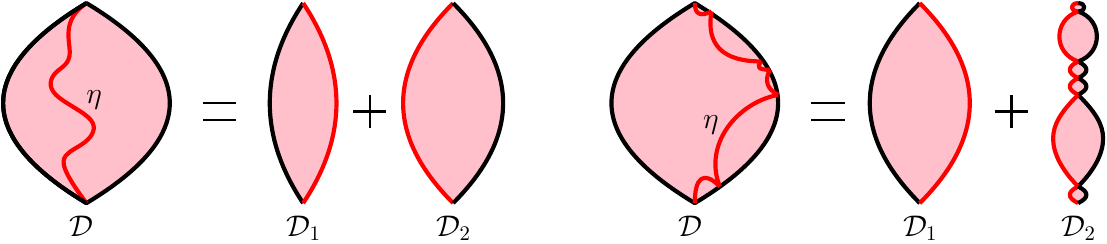}
    \caption{An illustration of Theorem~\hyperref[thm:disk-welding]{A}. In the left regime $W_1,W_2\ge\frac{\gamma^2}{2}$ while in the right picture $W_2<\frac{\gamma^2}{2}$ and $\mathcal{D}_2$ is a thin quantum disk, which is a beaded quantum surface.  }
    \label{fig:diskwelding}
\end{figure}

\subsection{Main results}\label{subsec:intro-result}
\subsubsection{Multiple SLE}
Fix $N>0$ and a link pattern $\alpha\in\mathrm{LP}_N$. Let $D\subset\bbC$ be a simply connected domain with $2N$ marked boundary points. We draw $N$ disjoint simple curves $\eta_1,...,\eta_N$ according to $\alpha$, dividing $D$ into $N+1$ connected components $S_1,...,S_{N+1}$. For $1\le k\le N+1$, let $n_k$ be the number of points  on the boundary of $S_k$, and $\eta_{k,1},...,\eta_{k,m_k}$ be the interfaces which are part of the boundary of $S_k$. Then for each $S_k$, we assign a quantum disk with $n_k$ marked points on the boundary from $\mathrm{QD}_{n_k}$, and consider the disintegration $\mathrm{QD}_{n_k} = \int_{\bbR_+^{n_k}}\mathrm{QD}_{n_k}(\ell_{k,1},...,\ell_{k,m_k})\,d\ell_{k,1}...d\ell_{k,m_k}$ over the quantum length of the boundary arcs  corresponding to the interfaces. For $(\ell_1,...,\ell_N)\in\bbR_+^N$, let $\ell_{k,j} = \ell_i$ if the interface $\eta_{k,j}=\eta_i$. We sample $N+1$ quantum surfaces from $\prod_{k=1}^{N+1}\mathrm{QD}_{n_k}(\ell_{k,1},...,\ell_{k,m_k})$ and conformally weld them together by LQG boundary length according to the link pattern $\alpha$, and write $\Wd_{\alpha}(\mathrm{QD}^{N+1})(\ell_1,...,\ell_N)$ for the law of the resulting quantum surface decorated by $N$ interfaces. Define $\Wd_{\alpha}(\mathrm{QD}^{N+1})$ by
$$ \Wd_{\alpha}(\mathrm{QD}^{N+1}) = \int_{\bbR_+^N}\Wd_{\alpha}(\mathrm{QD}^{N+1})(\ell_1,...,\ell_N)\,d\ell_1...d\ell_N. $$
See Figure~\ref{fig:thm-main} for an illustration.
\begin{figure}[htb]
    \centering
    \begin{tabular}{ccc} 
		\includegraphics[width=0.31\textwidth]{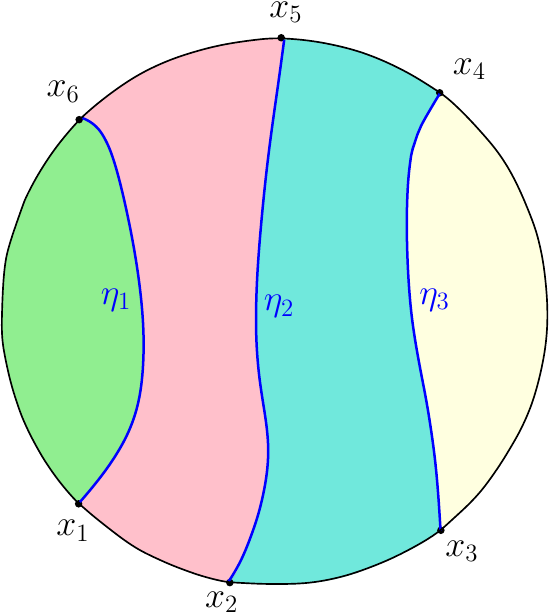}
		& \qquad \qquad\qquad&
		\includegraphics[width=0.32\textwidth]{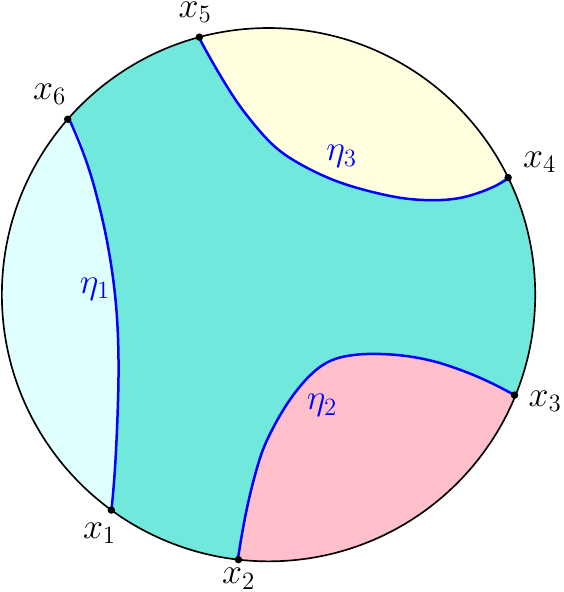}
	\end{tabular}
 \caption{An illustration of the definition of $\Wd_\alpha(\mathrm{QD}^{N+1})$ when $N=3$. \textbf{Left:} Under the link pattern $\alpha=\{\{1,6\},\{2,5\},\{3,4\} \}$, we are welding two samples from $\QD_{2}$ with two samples from $\QD_{4}$, and if we let $\ell_1,\ell_2,\ell_3$ be the interface lengths, then the precise welding equation can be written as $\int_{\bbR_+^3}\Wd(\QD_{2}(\ell_1),\QD_{4}(\ell_1,\ell_2),\QD_{4}(\ell_2,\ell_3),\QD_{2}(\ell_3) )d\ell_1\,d\ell_2\,d\ell_3 $. \textbf{Right:} In the setting where the link pattern $\alpha = \{\{1,6\},\{2,3\},\{4,5\} \}$, we are welding three samples from $\QD_{2}$ with one sample from $\QD_{6}$, and the corresponding welding equation is given by $\int_{\bbR_+^3}\Wd(\QD_{2}(\ell_1),\QD_{2}(\ell_2),\QD_{2}(\ell_3),\QD_{6}(\ell_1,\ell_2,\ell_3))d\ell_1\,d\ell_2\,d\ell_3 $. By Theorem~\ref{thm:main}, the entire disk decorated with the three interfaces $(\eta_1,\eta_2,\eta_3)$ can be realized as a Liouville field with 6 insertions decorated by an independent multiple $\SLE_\kappa$ with the corresponding link pattern, and the random cross ratio of the marked points is encoded by the multiple $\SLE_\kappa$ pure partition function. }\label{fig:thm-main}
 \end{figure}
{Now we are ready to state our  result for the multiple SLE.}
\begin{theorem}\label{thm:main}
Let $\gamma\in(0,2)$, $\kappa=\gamma^2$ and $\beta = \gamma-\frac{2}{\gamma}$. Let $N\ge 2$ and $\alpha\in\mathrm{LP}_N$ be a link pattern.  {Let  $c:=c_{2,2}\in(0,\infty)$ be the constant in Theorem~\hyperref[thm:disk-welding]{A} for $W_1=W_2=2$ depending only on $\kappa$. Let $c_N = \frac{\gamma(Q-\gamma)^{2N-2} c^{-N}}{2(Q-\beta)^{2N}}$. Then}
\begin{equation}\label{eq:thm-main}
\begin{split}
   \Wd_{\alpha}(\mathrm{QD}^{N+1}) &=   {c_N}\int_{0<x_1<...<x_{2N-3}<1}\bigg[\LF_\bbH^{(\beta,0),(\beta,1),(\beta,\infty),(\beta,x_1),...,(\beta,x_{2N-3})}\\&\times \mathrm{mSLE}_{\kappa,\alpha}(\bbH,0,x_1,...,x_{2N-3},1,\infty)\bigg]dx_1...dx_{2N-3}
    \end{split}
\end{equation}
where the  {right} hand side is understood as the law of a curve-decorated quantum surface,  {i.e., the law of $(\bbH,h,0,1,\infty,x_1,...,x_{2N-3},\eta_1,...,\eta_N)/\sim_\gamma$ with 
$(h,\eta_1,...,\eta_N)$  sampled from $\LF_\bbH^{(\beta,0),(\beta,1),(\beta,\infty),(\beta,x_1),...,(\beta,x_{2N-3})}\times \mathrm{mSLE}_{\kappa,\alpha}(\bbH,0,x_1,...,x_{2N-3},1,\infty)$}. 
\end{theorem}

The main idea of the proof is to start with the conformal welding of two quantum disks as proved in~\cite{AHS20}, and then add new boundary marked points as in~\cite{AHS21}. The proof is based on an induction, and we shall start with the following $N=2$ case, where we are welding two samples from $\mathrm{QD}_{2}$ with a sample from $\mathrm{QD}_{4}$ on opposite sides. In fact, for this case, the two two-pointed disks can be replaced by general weight $W$ quantum disks. For $W_1,W_2>0$ and $i=1,2$, let $b_i = \frac{W_i(W_i+4-\kappa)}{4\kappa}$. For $0<x<1$, define a measure on pair of curves $(\eta_1,\eta_2)$ as follows.  {For a simply connected domain $D$ and $x,y\in\partial D$ such that $\partial D$ is smooth near $x,y$, let $H_D(x,y) = \frac{f'(x)f'(y)}{|f(x)-f(y)|^2}$ be the boundary Poisson kernel where $f:D\to\bbH$ is some arbitrary conformal map.} First sample $\eta_1$ as a curve in $\bbH$ from 0 to $\infty$ from the probability measure $\SLE_\kappa(W_1-2)$ with the force point at $0^-$ and weight its law by $H_{D_{\eta_1}}(x,1)^{b_2}$, where $D_{\eta_1}$ is the component of $\bbH\backslash\eta_1$ lying to the right side of $\eta_1$. Then in  $D_{\eta_1}$, sample $\eta_2$ as a curve from $x$ to 1 from the probability measure $\SLE_\kappa(W_2-2)$ with the force point at $x^+$. Let $\mathfrak{m}_x(W_1,W_2)$ ( {which is a non-probability measure}) be the law of $(\eta_1,\eta_2)$.  {Note that for $j=1,2$, the curve $\eta_j$ would be boundary hitting if $W_j<\frac{\gamma^2}{2}$.}

\begin{figure}
    \centering
    \begin{tabular}{ccc} 
		\includegraphics[scale=0.51]{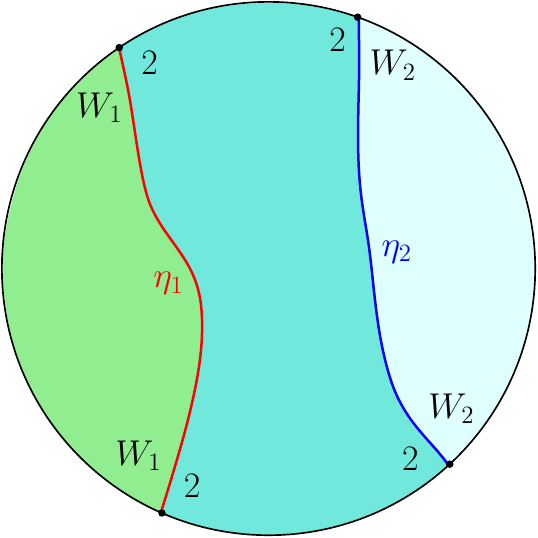}
		& \qquad &
		\includegraphics[scale=0.52]{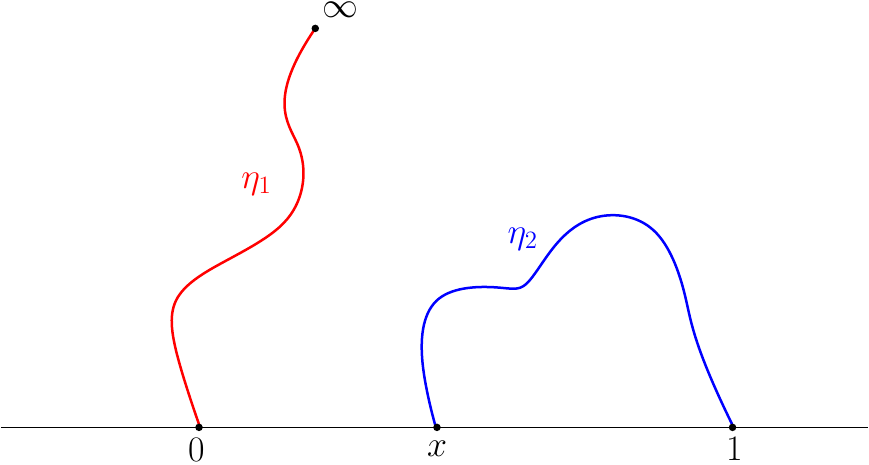}
	\end{tabular}
 \caption{An illustration of Theorem~\ref{thm:N=2}  {with $W_1,W_2\geq\frac{\gamma^2}{2}$}. \textbf{Left:} The conformal welding of three quantum disks as on the right hand side of~\eqref{eq:thm-N=2}. \textbf{Right:} The interface $(\eta_1,\eta_2)$ sampled from $\mathfrak{m}_x(W_1,W_2)$. }\label{fig:thm-N=2}
 \end{figure}

\begin{theorem}\label{thm:N=2}
    Let $\gamma\in(0,2)$, $\kappa=\gamma^2$, $W_1,W_2>0$. For $i=1,2$, let $\beta_i = \gamma-\frac{W_i}{\gamma}$. Then there exists a constant $c\in(0,\infty)$ such that
    \begin{equation}\label{eq:thm-N=2}
\begin{split}
    \int_0^1&\LF_\bbH^{(\beta_1,0),(\beta_2,1),(\beta_1,\infty),(\beta_2,x)}\times \mathfrak{m}_x(W_1,W_2) dx= c\Wd(\Md_2(W_1),\mathrm{QD}_{4},\Md_2(W_2))
    \end{split}
\end{equation}
where the left hand side is understood as the law of a curve-decorated quantum surface  {in the sense of Theorem~\ref{thm:main}}. 
\end{theorem}

See Figure~\ref{fig:thm-N=2} for an illustration. In the case $W_1=2$, the measure $\mathfrak{m}_x(W_1,W_2)$ is finite, and the marginal law of $\eta_1$ under the probability measure proportional to $\mathfrak{m}_x(W_1,W_2)$ agrees with the hypergeometric SLE $\mathrm{hSLE}_\kappa(W_2-2)$ in~\cite{wu2020hypergeometric}. Moreover, one can infer by Theorem~\ref{thm:N=2} and symmetry that up to a multiplicative constant, a sample from $\mathfrak{m}_y(W_1,W_2)$ can also be produced by (i) sample $\eta_2$ as an $\SLE_\kappa(W_2-2)$ curve in $\bbH$ from $x$ to 1 with force point at $x^+$ and weight its law by $H_{D_{\eta_2}}(0,\infty)^{b_1}$ (where $D_{\eta_2}$ is the left component of $\bbH\backslash\eta_2$) and (ii) sample $\eta_1$ as an  $\SLE_\kappa(W_1-2)$ curve in $D_{\eta_2}$ from 0 to $\infty$ with force point at $0^-$. This is the so-called commutation relation. It is also straightforward to see  from Theorem~\ref{thm:N=2} that  the law of the time reversal of  $\mathfrak{m}_x(W_1,W_2)$ is $\mathfrak{m}_x(W_2,W_1)$.   {Commutation relation and reversibility are natural from the conformal welding perspective in other settings as well}.

\subsubsection{Imaginary geometry flow lines}

{Let $\kappa\in(0,4)$, $\chi=\frac{2}{\sqrt{\kappa}}-\frac{\sqrt{\kappa}}{2}$, $\theta\in\bbR$ and $\mathfrak h$ be a Gaussian free field  {on $\bbH$} with piecewise  {constant} boundary conditions. In the framework of imaginary geometry in~\cite{MS16a},   it is possible make sense of the flow lines of the vector field $e^{i(\mathfrak h/\chi+\theta)}$ starting at  fixed boundary points of the domain. Such curves are $\SLE_\kappa(\underline\rho)$ processes, and are  {referred to as the flow lines} of $\mathfrak h$ with angle $\theta$.  One can also emanate flow lines of the GFF from
different boundary points with the same target point. }

For $\underline{x} = (x_1, ..., x_n)\in\bbR^n$ with $x_1<...<x_n$,  $\underline{\theta} = (\theta_1, ..., \theta_n)\in\bbR^n$ and $\underline{\lambda} = (\lambda_0, ..., \lambda_{n})\in\bbR^{n+1}$, let $\mathfrak h$ be the Dirichlet GFF on $\bbH$ with whose boundary value is given by $\lambda_{j-1}$ on $(x_{j-1}, x_j)$ for each $1\le j\le n+1$ (where $x_0 = -\infty$ and $x_{n+1}=+\infty$). For each $1\le i\le n$, let $\eta_i$ be the flow line of $\mathfrak h$ from $x_i$ with angle $\theta_i$, and let $\rho_i = \frac{(\lambda_i-\lambda_{i-1})\sqrt\kappa}{\pi}$. We write $\mathrm{IG}_{\underline{x}, \underline{\lambda}, \underline{\theta}}^\#$ for the joint law of $(\eta_1, ..., \eta_n)$. {Following~\cite{dubedat09partition}, a natural choice of the partition function for $\mathrm{IG}_{\underline{x}, \underline{\lambda}, \underline{\theta}}^\#$ is $\prod_{1\le i<j\le n}(x_j-x_i)^{\frac{\rho_i\rho_j}{2\kappa}}$, which aligns with the $\SLE_\kappa(\underline\rho)$ partition function in~\cite{SW05} of $\eta_j$ for each $1\le j\le n$. Then we set $\mathrm{IG}_{\underline{x}, \underline{\lambda}, \underline{\theta}} = \prod_{1\le i<j\le n}(x_j-x_i)^{\frac{\rho_i\rho_j}{2\kappa}}\mathrm{IG}_{\underline{x}, \underline{\lambda}, \underline{\theta}}^\#$. }

\begin{figure}[ht]
	\centering
 \begin{tabular}{ccc}
	\includegraphics[scale=0.58]{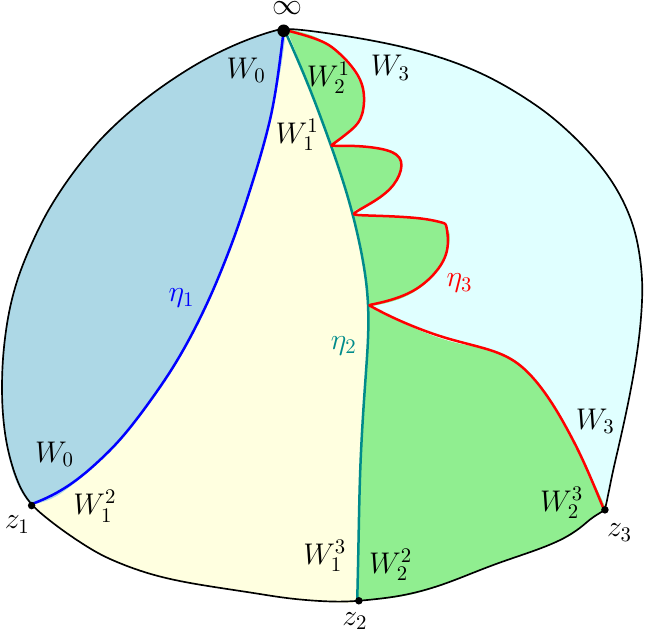} & \qquad & \includegraphics[scale=0.72]{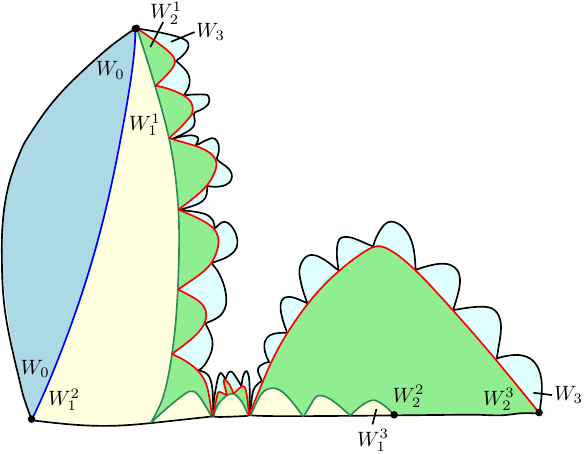}
 \end{tabular}
	\caption{\textbf{Left}: An illustration of Theorem~\ref{thm:IG} where $n=3$. \textbf{Right}: An example where the surface is not simply connected when $W_1^3+W_2^1+W_3<\frac{\gamma^2}{2}$.}\label{fig:thm-IG}
\end{figure}

\begin{theorem}\label{thm:IG}
    Let $n\ge 2$, $\gamma\in(0,2)$, $\kappa = \gamma^2$, $\chi=\frac{2}{\gamma}-\frac{\gamma}{2}$ and $\lambda = \frac{\pi}{\gamma}$. Let $W_0, W_n>0$, and $((W_1^1, W_1^2, W_1^3), ..., (W_{n-1}^1, W_{n-1}^2, W_{n-1}^3))\in (\bbR_+^3)^{n-1}$, such that for each $1\le j\le n-1$, $W_j^1+2 = W_j^2+W_j^3$. Also assume that for every $0\le i<j\le n$, $W_i^3+\sum_{i<k<j}W_k^1+W_j^2>\frac{\gamma^2}{2}$, where $W_0^1=W_0^3=W_0$ and $W_n^1=W_n^2 = W_n$. Let $\lambda_0 = -\lambda(W_0-1)$,  $\lambda_j = \lambda(\sum_{i=1}^{j-1}W_i^1+W_j^2-1)$, $\theta_j = -\sum_{i=1}^{j-1}\frac{\lambda W_i^1}{\chi}$, $\rho_j = W_{j-1}^3+W_j^2-2$ and $\beta_j = \gamma-\frac{\rho_j}{\gamma}$ for each $1\le j\le n$. Let $\beta_\infty = \gamma+\frac{2-\sum_{j=0}^n W_j^1}{\gamma}$. Consider  the conformal welding of samples from 
$\Md_2(W_0), \QT(W^1_1,W_1^2,W_1^3), \cdots,\QT(W_{n-1}^1,W_{n-1}^2,W_{n-1}^3),\Md_2(W_n)$ {in that order}. Then the output curve-decorated quantum surface is simply connected and can be embedded as $(\bbH, \phi, x_1, ..., x_n, \eta_1, ..., \eta_n, \infty)$, and there exists some constant $c\in(0,\infty)$ such that the law of $(\phi, x_1, ..., x_n, \eta_1, ..., \eta_n)$ is
\begin{equation}\label{eq:thm-IG}
   c\cdot \mathds{1}_{0<x_2<...<x_{n-1}<1}\LF_\bbH^{(\beta_j,x_j)_{1\le j\le n},(\beta_\infty, \infty)}(d\phi)\times {\mathrm{IG}_{\underline{x}, \underline{\lambda}, \underline{\theta}}}(d\eta_1...d\eta_n) \, dx_2...dx_{n-1}
\end{equation}
where $x_1=0$ and $x_n=1$.
\end{theorem}
See Figure~\ref{fig:thm-IG} for an illustration. The constraint $W_i^3+\sum_{i<k<j}W_k^1+W_j^2\ge\frac{\gamma^2}{2}$ for every $0\le i<j\le n$ is the necessary and sufficient condition to assure that the surface we obtained by gluing all the quantum disks and triangles together is a.s.\ simply connected.  {The surface would possibly be pinched if this condition is violated; see the right panel of Figure~\ref{fig:thm-IG} as an example.} If any of the sum above equals $\frac{\gamma^2}{2}$, using the same {method}  from~\cite[Section 2.5]{ASY22}, it is possible to define the Liouville field with the so-called $Q^-$ insertions and prove analogous results. We skip it here for simplicity.
The quantum triangles with constraint $W_1+2=W_2+W_3$ are called good quantum triangles in~\cite{ASY22}.
In that paper we proved a special case of Theorem~\ref{thm:IG}, namely ~\cite[Theorem 1.3]{ASY22}. In that setting, $\lambda_0=...=\lambda_n$ hence  the partition function is the constant 1. Moreover, the surface obtained by conformal welding  is $\QD_{n}$.

\subsubsection{Boundary Green's function}

Next we explain the connections with the SLE boundary Green's function established in the recent work~\cite{Zhan22b}. 
Let $N>0$. Parallel to the definition of link patterns, consider a simple curve $\eta$ in $\overline{\bbH}$ from $i_0\in\{0,...,N-1\}$ and ending at $N$ such that $\eta\cap\partial\bbH = \{0, ..., N\}$.  Topologically $\eta$ forms a planar partition of $\bbH$, which we call a curve link pattern and denote by $\alpha = (i_0, ..., i_{N-1})$ where $i_0, ..., i_{N-1}$ is the order of the marked points $(0, ..., N-1)$ visited by $\eta$. Denote the set of  curve link pattern with $N+1$ marked points by $\mathrm{CLP}_N$. For $\alpha\in\mathrm{CLP}_N$, let $\alpha^0\in\mathrm{CLP}_{N-1}$ be the curve link pattern obtained by removing the first segment from $\alpha$. 

For $N\ge 2$, $x_0<...<x_{N-1}\in\bbR$ distinct and $\alpha\in\mathrm{CLP}_N$, following~\cite{zhan17g,Zhan22b} we recursively define a measure $M_\alpha(x_0, ..., x_{N-1})$ on $N$ simple curves in $\bbH$ and a function $G_\alpha(x_0,...,x_{N-1})$ as follows. Let $b_2 = \frac{8}{\kappa}-1$ be the boundary scaling exponent, and $G_1(z) = |z|^{-b_2}$.  Start with $N=2$, and let $\alpha = (i_0,i_1)$.  First sample an $\SLE_\kappa(\kappa-8)$ curve $\eta_1$ from $x_{i_0}$ and aimed at $\infty$ with the force point located at $x_{i_1}$. By~\cite{Dub09, MS16a}, $\eta_1$ a.s.\ terminates at $x_{i_1}$, and given $\eta_1$, we then grow an $\SLE_\kappa$ curve $\eta_2$ from $x_{i_1}$ to $\infty$ in the unbounded component of $\bbH\backslash\eta_1$. We write $M_\alpha(x_0,x_1)^\#$ for the joint law of $(\eta_1,\eta_2)$, and let   $M_\alpha(x_0,x_1)=G_\alpha(x_0,x_1)\cdot M_\alpha(x_0,x_1)^\#$ with $G_\alpha(x_0,x_1)=G_1(x_1-x_0)$. Suppose $M_\alpha(x_0, ..., x_{N-1})$ and $G_\alpha(x_0,...,x_{N-1})$ has been defined for $N$. For $N+1$ case and $\alpha = (i_0, ..., i_{N})\in \mathrm{CLP}_{N+1}$, 
a sample $(\eta_1,...,\eta_N)$ from $M_\alpha(x_0, ..., x_{N})$ is produced as follows:
\begin{enumerate}[(i)]
    \item Sample an $\SLE_\kappa(\kappa-8)$ curve $\eta_1$ from $x_{i_0}$ and aimed at $\infty$ with the force point located at $x_{i_1}$. Let $(f_t)_{t>0}$ be its centered Loewner map and $T_1$ be the capacity time when $\eta_1$ hits $x_{i_1}$;
    \item Weight the law of $\eta_1$ by $$G_1(x_{i_1}-x_{i_0})\prod_{j=2}^Nf_{T_1}'(x_{i_j})^{b_2}\cdot G_{\alpha^0}(f_{T_1}(x_0),...,f_{T_1}(x_{i_0-1}),f_{T_1}(x_{i_0+1}),...,f_{T_1}(x_N)); $$
    \item Sample $(\eta_2^0, ..., \eta_{N+1}^0)$ from  the measure $M_{\alpha^0}(f_{T_1}(x_0),...,f_{T_1}(x_{i_0-1}),f_{T_1}(x_{i_0+1}),...,f_{T_1}(x_N))^\#$ and for $2\le j\le N+1$ let $\eta_j = f_{T_1}^{-1}\circ\eta_j^0$. 
\end{enumerate}
Then $G_\alpha(x_0,...,x_N)$ is defined by $|M_\alpha(x_0, ..., x_N)|$ and $M_\alpha(x_0,...,x_N)^\# = \frac{M_\alpha(x_0,...,x_N)}{G_\alpha(x_0,...,x_N)}$.

Now we comment on the relationship between our measure $M_\alpha(x_0, ..., x_{N-1})$ and the SLE boundary Green's function. For $x_0, ..., x_{n}\in\bbR$ distinct,  the $n$-point SLE boundary Green's function is defined by the limit
\begin{equation}\label{eq:def-SLE-Green}
    G(x_0,...,x_{n}) = \lim_{r_1,...,r_n\to 0^+}r_1^{-b_2}...r_n^{-b_2}\,\bbP\big(\dist(\eta, x_j)<r_j,\ 1\le j\le n \big)
\end{equation}
where $\eta$ is an $\SLE_\kappa$ curve from $x_0$ to $\infty$. The existence of the limit~\eqref{eq:def-SLE-Green} has been shown in~\cite{Lawler2015g} when $n=1$ or $n=2$ with $x_1,x_2>x_0$ and in~\cite{Zhan22b} in full generality. One variant of~\eqref{eq:def-SLE-Green} is the ordered boundary Green's function, which is defined by 
\begin{equation}\label{eq:def-SLE-Green-order}
    \vec{G}(x_0,...,x_{n}) = \lim_{r_1,...,r_n\to 0^+}r_1^{-b_2}...r_n^{-b_2}\,\bbP\big(\tau^{x_1}_{r_1}<...<\tau^{x_n}_{r_n}<\infty \big)
\end{equation}
 where $\tau^{x_j}_{r_j}$ denotes the first time when the $\SLE_\kappa$ curve $\eta$ hits $\{z:|z-x_j|<r_j\}$. By~\cite[Theorem 4.1 and Lemma 3.7]{Zhan22b}, for $\alpha = (i_0,i_1, ..., i_n)$, $x_0<...<x_n$, the identity
 \begin{equation}\label{eq:SLE-Green-order}
     G_\alpha(x_0,...,x_n) =  {\hat{c}^{n-1}}\cdot\vec{G}(x_{i_0}, ..., x_{i_n})
 \end{equation}
holds for some constant  {$\hat c\in(0,\infty)$} at least when (i) $n\le 2$ or (ii)  {$\alpha = (0,1,...,n)$ or $\alpha = (n,n-1,...,0)$. Here $\hat c$ is the constant from one-point boundary Green's function~\cite{Lawler2015g} depending only on $\kappa$.} 
A sample from $M(x_0,x_1)^\#$ is called two-sided chordal $\SLE_\kappa$ from $x_0$ to $\infty$ through $x_1$, while a sample from $M_\alpha(x_0,x_1,...,x_{n})^\#$ can be thought as the chordal $\SLE_\kappa$ from $x_0$ to $\infty$ ``conditioned" on hitting $x_1,...,x_n$ following the order induced by $\alpha$. {Moreover, one can infer from~\cite[Theorem 4.1]{Zhan22b} that $|M_\alpha(x_0, ..., x_{N})|<cG(x_0,...,x_{N})$ for some constant $c$ and thus our definition makes sense.}

\begin{figure}
    \centering
    \begin{tabular}{ccc} 
		\includegraphics[scale=0.51]{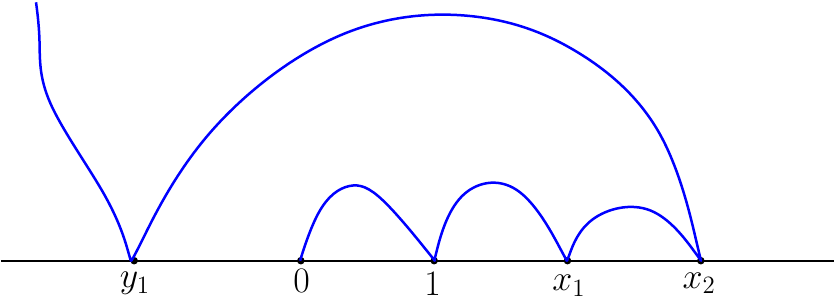}
		& \qquad &
		\includegraphics[scale=0.4]{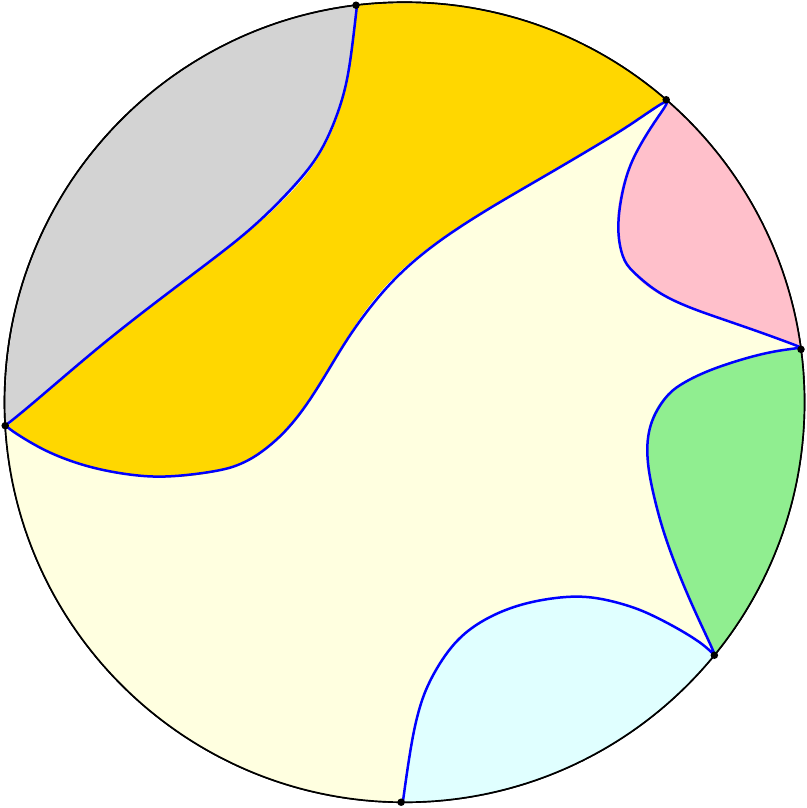}
	\end{tabular}
 \caption{An illustration of Theorem~\ref{thm:SLE-Green}. \textbf{Left:} A curve link pattern $\alpha = (1,2,3,4,0)\in\mathrm{CLP}_5$. \textbf{Right:} The conformal welding of quantum disks $\Wd_\alpha(\QD^{N+1})$ according to $\alpha$, where we are welding 4 samples from $\QD_{2}$, a sample from $\QD_{3}$ and a sample from $\QD_{5}$ together. }\label{thm:Green}
 \end{figure}

We state our result as follows. For $\alpha\in \mathrm{CLP}_N$, let $\Wd_\alpha(\QD^{N+1})$ be the conformal welding of quantum disks defined analogously with Theorem~\ref{thm:main}. See Figure~\ref{thm:Green} for an illustration.
\begin{theorem}\label{thm:SLE-Green}
    Let $\gamma\in(0,2)$, $\kappa=\gamma^2$, $\beta = \gamma-\frac{2}{\gamma}$ and $\beta_2 = \gamma-\frac{4}{\gamma}$. Let $N\ge 2$, $\alpha = (i_0, ..., i_{N-1})\in \mathrm{CLP}_N$ be a curve link pattern. Suppose $i_1>i_0$. Then there exists a constant $c\in(0,\infty)$ such that 
    \begin{equation}\label{eq:thm-Green}
\begin{split}
    &\int_{y_{i_0}<...<y_1<0<1<x_1<...<x_{N-i_0-2}}\bigg[\LF_\bbH^{(\beta,0),(\beta,\infty),(\beta_2,1),(\beta_2,x_1),...,(\beta_2,x_{N-i_0-2}),(\beta_2,y_1),...,(\beta_2,y_{i_0})}\times \\ &M_{\alpha}(y_{i_0},...,y_1,0,1,x_1,...,x_{N-i_0-2})\bigg]dy_1...dy_{i_0}dx_1...dx_{N-i_0-2}= c\,\Wd_{\alpha}(\mathrm{QD}^{N+1})
    \end{split}
\end{equation}
where the left hand side is understood as the law of a curve-decorated quantum surface  {in the sense of Theorem~\ref{thm:main}}. Likewise, if $i_1<i_0$, then for some $c\in(0,\infty)$, 
\begin{equation}\label{eq:thm-Green-1}
\begin{split}
    &\int_{y_{i_0-1}<...<y_1<-1<0<x_1<...<x_{N-i_0-1}}\bigg[\LF_\bbH^{(\beta,0),(\beta,\infty),(\beta_2,1),(\beta_2,x_1),...,(\beta_2,x_{N-i_0-1}),(\beta_2,y_1),...,(\beta_2,y_{i_0-1})}\times \\ &M_{\alpha}(y_{i_0-1},...,y_1,-1,0,x_1,...,x_{N-i_0-1})\bigg]dy_1...dy_{i_0-1}dx_1...dx_{N-i_0-1}= c\,\Wd_{\alpha}(\mathrm{QD}^{N+1}).
    \end{split}
\end{equation}
\end{theorem}

When $N=3$, following~\cite{Zhan21}, the aforementioned result can also be extended to $\SLE_\kappa(\rho)$ curves.  Let $\rho>-2$ and $b_\rho = \frac{(\rho+2)(2\rho+8-\kappa)}{2\kappa}$. For $x>1$, define a measure $M(\rho;x)$ on three simple curves in $\overline{\bbH}$ as follows. First sample an $\SLE_\kappa(\rho,\kappa-8-2\rho)$ curve $\eta_1$ from 0 to $\infty$ with force points $0^+,1$ (such curve a.s. terminates at $1$  {following the criterion from~\cite{Dub09,MS16a} since $\rho+\kappa-8-2\rho<\frac{\kappa}{2}-4$}) and then weight its law by $f_\tau'(x)^{b_\rho}(f_\tau(x)-f_\tau(1))^{-b_\rho}$, where $(f_t)_{t>0}$ is the centered Loewner map for $\eta_1$ and $\tau$ is the time when $\eta_1$ hits 1. Then sample an $\SLE_\kappa(\rho,\kappa-8-2\rho)$ curve $\eta_2$ in the unbounded connected component of $\bbH\backslash\eta_1$ from 1 to $\infty$ with force points at $1^+,x$  {which also a.s.\  terminates at $x$}. Finally sample an $\SLE_\kappa(\rho)$ curve $\eta_3$ in the unbounded connected component of $\bbH\backslash(\eta_1\cup\eta_2)$ from $x$ to $\infty$ with force points $x^+$. Let $M(\rho;x)$ be the joint law of $(\eta_1,\eta_2,\eta_3)$. By~\cite[Remark 5.7]{Zhan21}, $|M(\rho;x)|$ is a constant times the limit
\begin{equation}
    \lim_{r_1\to0,r_2\to0}r_1^{-b_\rho}r_2^{-b_\rho}\bbP(\dist(\eta,1)<r_1,\dist(\eta,x)<r_2)
\end{equation}
where $\eta$ is an $\SLE_\kappa(\rho)$ curve from 0 to $\infty$ with force point at $0^+$.  {Also note that the curves $\eta_1,\eta_2,\eta_3$ does not hit $(0,+\infty)$ except at their starting and ending points if and only if $\rho\geq\frac{\kappa}{2}-2$.}

\begin{figure}
    \centering
    \begin{tabular}{ccc} 
		\includegraphics[scale=0.51]{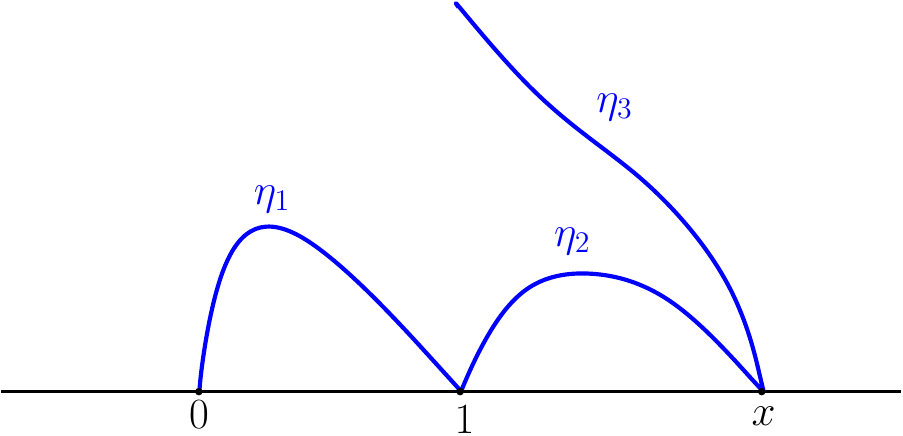}
		& \qquad &
		\includegraphics[scale=0.4]{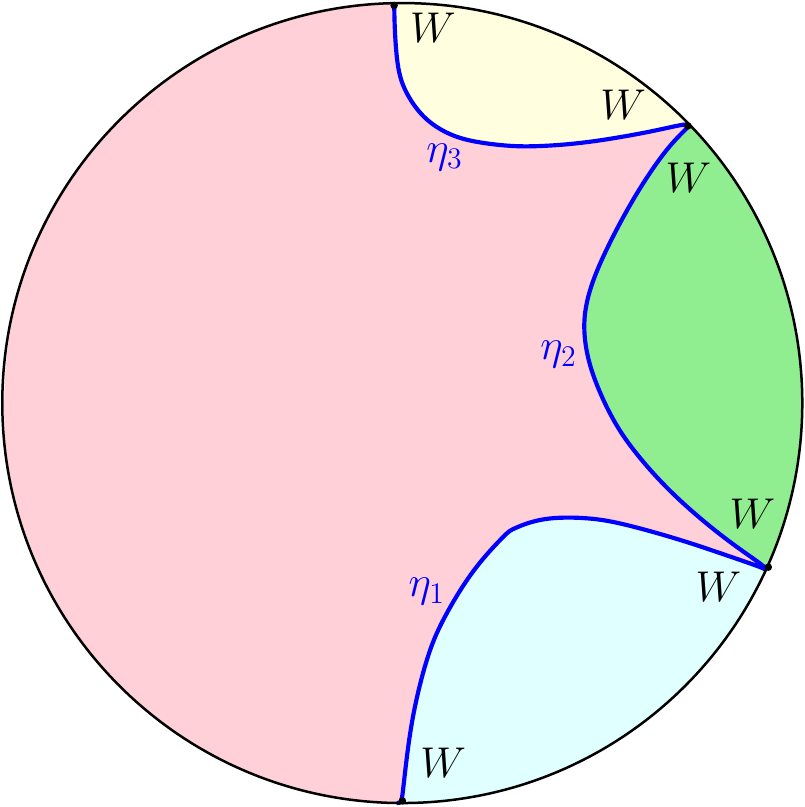}
	\end{tabular}
 \caption{An illustration of Theorem~\ref{thm:green-rho}. \textbf{Left:} An illustration of the measure $M(\rho;x)$. \textbf{Right:} The conformal welding of three quantum disks $\Md_2(W)$ and a sample from $\QD_{4}$ following the link pattern induced by $M(\rho;x)$.}\label{fig:green-rho}
 \end{figure}

Our result is the following. See also Figure~\ref{fig:green-rho} for an illustration.
\begin{theorem}\label{thm:green-rho}
    Let $\gamma\in(0,2)$, $\kappa=\gamma^2$, $\rho>-2$ and $W=\rho+2$. Let $\beta_\rho = \gamma-\frac{2+\rho}{\gamma}$ and $\beta_{2,\rho}=\gamma-\frac{2+2\rho}{\gamma}$. Consider the conformal welding of three samples from $\Md_2(W)$ and a sample from $\QD_{4}$ as in Figure~\ref{fig:green-rho}. Then for some constant $c\in(0,\infty)$, the output curve-decorated quantum surface can be embedded as $(\bbH,\phi,0,1,x,\infty,\eta_1,\eta_2,\eta_3)$ where $(\phi,x,\eta_1,\eta_2,\eta_3)$ has law
    \begin{equation}\label{eq:thm-green-rho}
        c\cdot\mathds{1}_{x\in[1,\infty)} \LF_\bbH^{(\beta_\rho,0),(\beta_{2,\rho},1),(\beta_{2,\rho},x),(\beta_\rho,\infty)}(d\phi)\times M(\rho;x)(d\eta_1d\eta_2d\eta_3)dx.
    \end{equation}
\end{theorem}

\subsubsection{Overview of the proof and consequences}
Different from the existing works on conformal welding of LQG surfaces~\cite{DMS14,AHS20,AHS21,ASY22}, in this work we emphasize on the conformal welding of a multiple number of LQG surfaces,  {where after the welding we obtain LQG surfaces with more than three marked points on the boundary and thus produce} nontrivial conformal structure. To prove Theorems~\ref{thm:main}-\ref{thm:green-rho}, we use induction. We begin with the conformal welding of two quantum disks, or a quantum disk with a quantum triangle. Then we sample additional marked points on the boundary from the LQG length measure, and conformally weld the other surfaces along the new edge. One  {novel technique in our proof} is, we shall constantly shift between different embeddings of the same surface via random conformal maps. In addition to the conformal covariance of Liouville CFT, we also view the SLE measures as non-probability measures with non-trivial parition functions, and perform these conformal maps as well. As we will see in the proof, in all cases in Theorems~\ref{thm:main}-\ref{thm:green-rho} the sum of ``conformal weights" of SLE and LCFT add up to 1, which allows us to switch the random embedding. To identify the interfaces and their partition function after gluing the new surface, we shall also combine existing constructions of the multiple SLE curves in literature. In particular, for Theorem~\ref{thm:main}, we apply the cascade probabilistic construction of multiple SLE in~\cite{peltola2019global}, for Theorem~\ref{thm:IG} we use a refined version of the Schramm-Wilson martingale~\cite{SW05} as in Section~\ref{sec:IG}, and for Theorem~\ref{thm:SLE-Green} and Theorem~\ref{thm:green-rho} we adapt the iterative construction of SLE boundary Green's function in~\cite{Zhan21,Zhan22b}.

In a recent work~\cite{ARS22},  {Ang, Remy and one of us} derived the law of the  moduli of annuli in random conformal geometry, while a main result of  {our current} paper is the  moduli of the surfaces from random conformal welding. In~\cite{ARS22}, the derivation of the random moduli mainly relied on integrability results for annulus LCFT, while  {in this paper}  the derivation is intrinsic on the SLE side and is independent of  LCFT integrability. Nevertheless,  {results of this paper combined with LCFT integrability} are still useful for studying the SLE partition functions. For instance, based on Theorems~\ref{thm:SLE-Green} and~\ref{thm:green-rho} along with the techniques from~\cite{ARS22}, it  {may be} possible to give an  {explicit expression or other fine properties of} the SLE boundary Green's function. One other direction is the finiteness of the SLE partition functions in other settings. In a subsequent work~\cite{AHSY23} with Ang and Holden, we prove the finiteness of the multiple SLE partition function for $\kappa\in(4,8)$ based on the $\kappa\in(4,8)$ analog of Theorem~\ref{thm:main}. Moreover, the finiteness of partition function for SLE in general multiply connected domains is an open question for $\kappa\in(\frac{8}{3},4]$~\cite{lawler2011defining}. Some  {interesting} progress has been made recently by Aru and Bordereau~\cite{aru2024sle}, and we believe that this question can be settled by conformal welding of LQG surfaces as well; the coupling of SLE and CFT in multiply connected domains has also been studied recently in \cite{alberts2024conformal}.

\subsection{Future directions}\label{subsec:outlook}

As we have shown in various settings, conformal welding of LQG surfaces can be used  construct the SLE partition function.  The examples considered in this paper are relatively well understood. The ones in Theorems~\ref{thm:N=2} and ~\ref{thm:IG} are explicit. The one in Theorem~\ref{thm:main} have a PDE characterization. The boundary Green functions have an explicit limit construction. 
In principle, all the information of the SLE partition function is completely contained in the conformal welding of LQG surfaces. It is of great interest to derive the known properties of these partition functions from LQG.  Moreover, we plan to use the LQG method to construct and study SLE partition function in other settings, including the following.
\begin{itemize}
 \item In this paper we focus on $\kappa \in (0,4)$, where the SLE curves are simple. We expect no major difficulty extending the results to $\kappa=4$. For $\kappa\in (4,8)$ where the curves become non-simple, we need to first develop the corresponding conformal welding techniques. We will do this  {in the {following} work~\cite{AHSY23}} with Ang and Holden, based on which we  prove the analog of Theorem~\ref{thm:main} for $\kappa\in (4,8)$. 
  As an application, we prove the finiteness of partition function for multiple $\SLE_\kappa$ for $\kappa\in(4,8)$. Previously this is only known for $N=2$ ~\cite{miller2018connection} or $\kappa\le 6$ when $N\geq3$~\cite{wu2020hypergeometric,peltola2019global}. One might try to settle this problem via the construction discussed in~\cite{miller2018connection} for  {the} $N=2$ case; yet our work~\cite{AHSY23} is the first paper that provides a  {full carried-out} proof.
  
 \item In this paper, we focus on the case where the marked points all lie on the boundary of the domain. 
 Recently  multiple radial SLE were constructed  in~\cite{healey2021n}, where we have one interior marked point in additional to the boundary marked points. Moreover, various interior Green function were considered in~\cite{LR15Green,zhan17g,zhan182green}.  We believe that the analog conformal welding results hold in these case, but except for a few special cases, a rigorous proof would require new ideas.

\item Our Theorem~\ref{thm:IG} shows that imaginary geometry 
 flow lines emanating from the boundary can be obtained by conformally weld good quantum triangles. Similarly, we can conformally weld quantum triangles to form a sphere. It is a natural question to understand the law of the resulting interfaces. In~\cite{ig4}, the imaginary geometry on the sphere were mainly developed for the case of two marked points. We believe that imaginary geometry on the sphere has an interesting extension that allows arbitrarily many marked points. Moreover, under proper assumptions on the quantum triangles, the interfaces can still be interpreted as flow lines in the extended theory.

    \item It is also natural to conformally weld quantum disks and triangles to form non-simply connected surfaces.
With Ang, the second named author  will prove in a future work that conditioning on the moduli, such surfaces are described by LCFT, and the law of SLE interfaces are decoupled with the field. 
It is then natural to define the partition function of these SLE curves using the density of the random moduli. Multiple SLE curves on non-simply connected surfaces is a less explored topic; see  \cite{jahangoshahi2018multiple} on the case of  multiple-connected planar domains. We believe that it is an interesting and fruitful direction. In particular, these SLE partition functions should be related to conformal field theory on Riemann surfaces, which have rich structures. 
\end{itemize}

\medskip
\noindent\textbf{Acknowledgements.} 
We are grateful to Morris Ang for working together with us at the early stage of the project. 
We thank Dapeng Zhan and Baojun Wu for helpful discussions and Hao Wu for pointing out references for Lemma~\ref{lm:martingale}. We thank two anonymous referees for their valuable feedback.  P.Y. were partially supported by NSF grant DMS-1712862.  X.S.\ was partially supported by the NSF Career award 2046514, a start-up grant from the University of Pennsylvania, and a fellowship from the Institute for Advanced Study (IAS) at Princeton.  P.Y. thanks IAS for hosting his visit during Fall 2022.

\section{Preliminaries}\label{sec:pre}
In this paper we work with non-probability measures and extend the terminology of ordinary probability to this setting. For a finite or $\sigma$-finite  measure space $(\Omega, \mathcal{F}, M)$, we say $X$ is a random variable if $X$ is an $\mathcal{F}$-measurable function with its \textit{law} defined via the push-forward measure $M_X=X_*M$. In this case, we say $X$ is \textit{sampled} from $M_X$ and write $M_X[f]$ for $\int f(x)M_X(dx)$. \textit{Weighting} the law of $X$ by $f(X)$ corresponds to working with the measure $d\tilde{M}_X$ with Radon-Nikodym derivative $\frac{d\tilde{M}_X}{dM_X} = f$, and \textit{conditioning} on some event $E\in\mathcal{F}$ (with $0<M[E]<\infty$) refers to the probability measure $\frac{M[E\cap \cdot]}{M[E]} $  over the space $(E, \mathcal{F}_E)$ with $\mathcal{F}_E = \{A\cap E: A\in\mathcal{F}\}$. If $M$ is finite, we write $|M| = M(\Omega)$ and $M^\# = \frac{M}{|M|}$ for its normalization.  We also fix the notation $|z|_+:=\max\{|z|,1\}$ for $z\in\bbC$.

We also extend the terminology to the setting of more than one random variable sampled from non-probability measures. By saying ``we first sample $X_1$ from $M_1$ and then sample $X_2$ from  $M_2$" ( {where the measure $M_2$ possibly depends on $X_1$}), we refer to a sample $(X_1,X_2)$ from  {$ M_2(x_1,dx_2)M_1(dx_1)$}.  In this setting, weighting the law of $X_2$ by $f(X_2)$ corresponds to working with the measure $\wt{M}_X$ with Radon-Nikodym derivative  {$\frac{\wt{M}_X(dx_1,dx_2)}{M_2(x_1,dx_2)M_1(dx_1)}(x_1,x_2) = f(x_2)$}. In the case where $M_2$ is a probability measure, we say that the marginal law of $X_1$ is $M_1$.

For a M\"{o}bi{u}s transform $f:\bbH\to\bbH$ and $s\in\bbR$, if $f(s)=\infty$, then we define $f'(s) = (-\frac{1}{f(w)})'|_{w=s}$. Likewise, if $f(\infty)=s$, then we set $f'(\infty) = ((f^{-1})'(s))^{-1}$. In particular, if $f(z) = a+\frac{\lambda}{x-z}$, then $f'(x) = \lambda^{-1}$ and $f'(\infty) = \lambda$. If $f(z) = a+rz$ with $a\in\bbR,r>0$, then we write $f'(\infty)=r^{-1}$.  These align with the conventions in~\cite{lawler2009partition}.

For a conformal map $\varphi:D\to \tilde D$ and a mesaure $\mu(D;x,y)$ on continuous curves from $x$ to $y$ in $\overline{D}$, we write $\varphi\circ\mu(D,x,y)$ for the law of $\varphi\circ\eta$ when $\eta$ is sampled from $\mu(D;x,y)$.

\subsection{The Gaussian free field and Liouville quantum gravity surfaces}\label{subsec:pre-lqg}
Let $m$ be the uniform measure on the unit  semicircle $\bbH\cap\bbD$. Define the Dirichlet inner product $\langle f,g\rangle_\nabla = (2\pi)^{-1}\int_X \nabla f\cdot\nabla g $ on the space $\{f\in C^\infty(\bbH):\int_\bbH|\nabla f|^2<\infty; \  \int f(z)m(dz)=0\},$ and let $H(\bbH)$ be the closure of this space w.r.t.\ the inner product $\langle f,g\rangle_\nabla$. Let $(f_n)_{n\ge1}$ be an orthonormal basis of $H(\bbH)$, and $(\alpha_n)_{n\ge1}$ be a collection of independent standard Gaussian variables. Then the summation
$$h = \sum_{n=1}^\infty \alpha_nf_n$$
a.s.\ converges in the space of distributions on $\bbH$, and $h$ is the \emph{Gaussian Free Field} on $\bbH$ normalized such that $\int h(z)m(dz) = 0$. {Let $P_\bbH$ be the law of $h$. Following~\cite{She07,dubedat09partition}, it can be shown that $P_\bbH$ is a probability measure on $H^{-1}(\bbH)$, where $H^{-1}(\bbH)$ is the dual space of $H(\bbH)$.}   See~\cite[Section 4.1.4]{DMS14} for more details.

For $z,w\in\overline{\bbH}$, we define
$$
    G_\bbH(z,w)=-\log|z-w|-\log|z-\bar{w}|+2\log|z|_++2\log|w|_+; \ G_\bbH(z,\infty) = 2\log|z|_+.
$$
Then the GFF $h$ is the centered Gaussian field on $\bbH$ with covariance structure $\bbE [h(z)h(w)] = G_\bbH(z,w)$. As pointed out in~\cite[Remark 2.3]{AHS21}, if $\phi = h+f$ where $f$ is a  function continuous everywhere except for finitely many log-singularities, then $\phi$ is a.s.\ in the dual space $H^{-1}(\bbH)$ of $H(\bbH)$.   

Now let $\gamma\in(0,2)$ and $Q=\frac{2}{\gamma}+\frac{\gamma}{2}$. Consider the space of pairs $(D,h)$, where $D\subseteq \mathbb C$ is a planar domain and $h$ is a distribution on $D$ (often some variant of the GFF). For a conformal map $g:D\to\tilde D$ and a generalized function $h$ on $D$, define the generalized function $g\bullet_\gamma h$ on $\tilde{D}$ by setting 
\begin{equation}\label{eq:lqg-changecoord}
    g\bullet_\gamma h:=h\circ g^{-1}+Q\log|(g^{-1})'|.
\end{equation}
 Define the following equivalence relation $\sim_\gamma$, where $(D, h)\sim_\gamma(\wt{D}, \wt{h})$ if there is a conformal map $\varphi:D\to\wt{D}$ such that $\tilde h = \varphi\bullet_\gamma h$.
A \textit{quantum surface} $S$ is an equivalence class of pairs $(D,h)$ under the equivalence relation $\sim_\gamma$, and we say that $(D,h)$ is an \emph{embedding} of $S$ if $S = (D,h)/\mathord\sim_\gamma$. Likewise, a \emph{quantum surface with} $k$ \emph{marked points} is an equivalence class of elements of the form
	$(D, h, x_1,\dots,x_k)$, where $(D,h)$ is a quantum surface, the points  $x_i\in {\overline{D}}$, and with the further
	requirement that marked points (and their ordering) are preserved by the conformal map $\varphi$ in \eqref{eq:lqg-changecoord}. A \emph{curve-decorated quantum surface} is an equivalence class of tuples $(D, h, \eta_1, ..., \eta_k)$,
	where $(D,h)$ is a quantum surface, $\eta_1, ..., \eta_k$ are curves in $\overline D$, and with the further
	requirement that $\eta$ is preserved by the conformal map $\varphi$ in \eqref{eq:lqg-changecoord}. Similarly, we can
	define a curve-decorated quantum surface with $k$ marked points.  {Consider a measure $\mathcal{M}$ on the space of quantum surfaces and a conformally covariant measure $\mathcal{P}$ (which is usually some variant of SLE) on the space of curves. Suppose $D\subseteq\bbC$ is some domain and  $\mathcal{M}_D$ is some measure on the space of generalized functions on $D$ such that for $h_D$ sampled from $D$, the law of $(D,h_D)/\sim_\gamma$ is $\mathcal{M}$. Let $\mathcal{P}_D$ be the measure $\mathcal{P}$ on the domain $D$. Then we write $\mathcal{M}\times\mathcal{P}$ for the measure describing the law of $(D,h_D,\eta_D)/\sim_\gamma$ where $(h_D,\eta_D)$ is a sample from $\mathcal{M}_D\times\mathcal{P}_D$.} Throughout this paper,  {we will mostly work on the case illustrated as above.} 

 For a $\gamma$-quantum surface $(D, h, z_1, ..., z_m)$, its \textit{quantum area measure} $\mu_h$ is defined by taking the weak limit $\epsilon\to 0$ of $\mu_{h_\epsilon}:=\epsilon^{\frac{\gamma^2}{2}}e^{\gamma h_\epsilon(z)}d^2z$, where $d^2z$ is the Lebesgue area measure and $h_\epsilon(z)$ is the circle average of $h$ over $\partial B(z, \epsilon)$. When $D=\mathbb{H}$, we can also define the  \textit{quantum boundary length measure} $\nu_h:=\lim_{\epsilon\to 0}\epsilon^{\frac{\gamma^2}{4}}e^{\frac{\gamma}{2} h_\epsilon(x)}dx$ where $h_\epsilon (x)$ is the average of $h$ over the semicircle $\{x+\epsilon e^{i\theta}:\theta\in(0,\pi)\}$. It has been shown in \cite{DS11, SW16} that all these weak limits are well-defined  for the GFF and its variants we are considering in this paper, and that   $\mu_{h}$ and $\nu_{h}$ can be conformally extended to other domains using the relation $\bullet_\gamma$.

 As argued in \cite[Section 4.1]{DMS14},  we have the decomposition $H(\mathbb{H}) = H_1(\mathbb{H})\oplus H_2(\mathbb{H})$, where $H_1(\mathbb{H})$ is the subspace of radially symmetric functions, and $H_2(\mathbb{H})$ is the subspace of functions having mean 0 about all semicircles $\{|z|=r,\ \text{Im}\ z>0\}$.  {Moreover, $H_1(\mathbb{H})$ and $H_2(\mathbb{H})$ are orthogonal with respect to the Dirichlet inner product.} As a consequence, for the GFF $h$ sampled from $P_\bbH$, we can decompose $h=h_1+h_2$, where $h_1$ and $h_2$ are independent distributions given by the projection of $h$ onto $H_1({\mathbb{H}})$ and $ H_2({\mathbb{H}})$, respectively. 

 We now turn to the definition of \emph{quantum disks}, which is  {split} in two different cases: \emph{thick quantum disks} and \emph{thin quantum disks}. These surfaces can also be equivalently constructed via methods in Liouville conformal field theory (LCFT) as we shall briefly discuss in the next subsection; see e.g.~\cite{DKRV16, HRV-disk} for these constructions and see \cite{AHS17,cercle2021unit,AHS21} for proofs of equivalence with the surfaces defined above. 
	
	\begin{definition}[Thick quantum disk]\label{def-quantum-disk}
		Fix $\gamma\in(0,2)$ and let $(B_s)_{s\ge0}$ and $(\wt{B}_s)_{s\ge0}$ be independent standard one-dimensional Brownian motions.  Fix a weight parameter $W\ge\frac{\gamma^2}{2}$ and let $\beta = \gamma+ \frac{2-W}{\gamma}\le Q$. Let $\mathbf{c}$ be sampled from the infinite measure $\frac{\gamma}{2}e^{(\beta-Q)c}dc$ on $\bbR$ independently from $(B_s)_{s\ge0}$ and $(\wt{B}_s)_{s\ge0}$.
			Let 	
			\begin{equation*}
				Y_t=\left\{ \begin{array}{rcl} 
					B_{2t}+\beta t+\mathbf{c} & \mbox{for} & t\ge 0,\\
					\wt{B}_{-2t} +(2Q-\beta) t+\mathbf{c} & \mbox{for} & t<0,
				\end{array} 
				\right.
			\end{equation*}
			conditioned on  $B_{2t}-(Q-\beta)t<0$ and  $ \wt{B}_{2t} - (Q-\beta)t<0$ for all $t>0$. Let $h$ be a free boundary  GFF on $\mathbb{H}$ independent of $(Y_t)_{t\in\bbR}$ with projection onto $H_2(\mathbb{H})$ given by $h_2$. Consider the random distribution
			\begin{equation*}
				\psi(\cdot)= {Y}_{-\log|\cdot|} + h_2(\cdot) \, .
		\end{equation*}
		Let $\mathcal{M}_2^{\mathrm{disk}}(W)$ be the infinite measure describing the law of $({\mathbb{H}}, \psi,0,\infty)/\mathord\sim_\gamma $. 
		We call a sample from $\mathcal{M}_2^{\textup{disk}}(W)$ a \emph{quantum disk} of weight $W$ with two marked points.
		
		We call $\nu_\psi((-\infty,0))$ and $\nu_\psi((0,\infty))$ the left and right  quantum   {boundary} length of the quantum disk  $(\mathbb{H}, \psi, 0, \infty)$.
	\end{definition}
	
	When $0<W<\frac{\gamma^2}{2}$, we define the \emph{thin quantum disk} as the concatenation of weight $\gamma^2-W$ thick disks with two marked points as in \cite[Section 2]{AHS20}.

	\begin{definition}[Thin quantum disk]\label{def-thin-disk}
		Fix $\gamma\in(0,2)$. For $W\in(0, \frac{\gamma^2}{2})$, the infinite measure $\mathcal{M}_2^{\textup{disk}}(W)$ is defined as follows. First sample a random variable $T$ from the infinite measure $(1-\frac{2}{\gamma^2}W)^{-2}\textup{Leb}_{\mathbb{R}_+}$; then sample a Poisson point process $\{(u, \mathcal{D}_u)\}$ from the intensity measure $\mathds{1}_{t\in [0,T]}dt\times \mathcal{M}_2^{\textup{disk}}(\gamma^2-W)$; and finally consider the ordered (according to the order induced by $u$) collection of doubly-marked thick quantum disks $\{\mathcal{D}_u\}$, called a \emph{thin quantum disk} of weight $W$.
		
		Let $\mathcal{M}_2^{\textup{disk}}(W)$ be the infinite measure describing the law of this ordered collection of doubly-marked quantum disks $\{\mathcal{D}_u\}$.
		The left and right   boundary length of a sample from $\mathcal{M}_2^{\textup{disk}}(W)$ is set to be equal to the sum of the left and right boundary lengths of the quantum disks $\{\mathcal{D}_u\}$.
	\end{definition}

  For $W>0$, one can {disintegrate} the measure $\Md_2(W)$ according to its the quantum length of the left and right boundary arc, i.e., 
 \begin{equation}
     \Md_2(W) = \int_0^\infty\int_0^\infty \Md_2(W;\ell_1,\ell_2)d\ell_1\,d\ell_2,
 \end{equation}
where $\Md_2(W;\ell_1,\ell_2)$ is supported on the set of doubly-marked quantum surfaces with left and right boundary arcs having quantum lengths $\ell_1$ and $\ell_2$, respectively. One can also define  $\Md_2(W;\ell) := \int_0^\infty \Md_2(W;\ell,\ell')d\ell'$, i.e., the disintegration over the quantum length of the left (or right) boundary arc. 

Finally the weight 2 quantum disk is special in the sense that its two marked points are typical with
respect to the quantum boundary length measure~\cite[Proposition A.8]{DMS14}. Based on this we can define the family of quantum disks marked with multiple quantum typical points,  {i.e., marked points sampled from the quantum boundary length measure.}

\begin{definition}\label{def:QD}
    Let $(\mathbb{H},\phi, 0, \infty)$ be  the embedding of a sample from $\Md_2(2)$ as in Definition~\ref{def-quantum-disk}. Let $L=\nu_\phi(\partial\bbH)$, and $\mathrm{QD}$ be the law of $(\bbH,\phi)$ under the reweighted measure $L^{-2}\Md_2(2)$. For $n\ge0$, let  $(\bbH,\phi)$ be a sample  {from $\frac{1}{n!}L^n\mathrm{QD}$ and then sample $s_1,...,s_n$ on $\partial\bbH$ according to the probability measure  {$(n-1)!\cdot1_{A_n}\nu_\phi^\#(ds_1)...\nu_\phi^\#(ds_n)$, where $A_n$ is the event that $s_1,...,s_n$ are ordered counterclockwise along $\partial\bbH$}}. Let $\mathrm{QD}_{n}$ be the law of $(\bbH,\phi,s_1,...,s_n)/\sim_\gamma$, and we call a sample from $\QD_{n}$ \emph{a quantum disk with $n$ boundary marked points}.
\end{definition}

\begin{remark}\label{remark:QD}
   { {In~\cite{ARS21,AHS21,ASY22}, the quantum disks there could have $m$ bulk marked points and $n$ boundary marked points, and their corresponding laws are denoted by  $\QD_{m,n}$. Since in this paper we mostly work on boundary marked points, we drop the $m$ parameter to simplify the notation.} The notion $\QD_{n}$ defined here also differs by a constant  {$(n-1)!$} compared with that defined in~\cite{AHS21,ASY22}. It also follows easily that for $n\ge3$, $\QD_{n}$ can also be defined recursively. Namely, if $(\bbH,\phi,s_1,...,s_n)$ is an embedding of $\QD_{n}$ with $s_1<...<s_n$, then as we weight the law of $(\bbH,\phi,s_1,...,s_n)$ by  {$\nu_\phi(s_n,s_1)$} and sample $s_{n+1}$ from $\nu_\phi|_{(s_n,s_1)}^\#$, then  $(\bbH,\phi,s_1,...,s_{n+1})$ is an embedding of $\QD_{n+1}$. }
\end{remark}

\begin{proposition}[Proposition A.8 of~\cite{DMS14}]
We have $\QD_{2}=\Md_2(2)$.
\end{proposition}

\subsection{Liouville conformal field theory on the upper half plane}
Recall that $P_\bbH$ is the law of the free boundary GFF on  $\bbH$ normalized to have average zero on $\partial\bbD\cap\bbH$,  {and $|z|_+ = \max\{|z|,1\}$.}
\begin{definition}
Let $(h, \mathbf c)$ be sampled from $P_\bbH \times [e^{-Qc} dc]$ and $\phi = h-2Q\log|z|_+ + \mathbf c$. We call $\phi$ the Liouville field on $\bbH$, and we write
$\LF_\bbH$ for the law of $\phi$.
\end{definition}

\begin{definition}[Liouville field with boundary insertions]\label{def-lf-H-bdry}
	Let $\beta_i\in\mathbb{R}$  and $s_i\in \partial\mathbb{H}\cup\{\infty\}$ for $i = 1, ..., m,$ where $m\ge 1$ and all the $s_i$'s are distinct. Also assume $s_i\neq\infty$ for $i\ge 2$. We say $\phi$ is a \textup{Liouville Field on $\mathbb{H}$ with insertions $\{(\beta_i, s_i)\}_{1\le i\le m}$} if $\phi$ can be produced as follows by  first sampling $(h, \mathbf{c})$ from $C_{\mathbb{H}}^{(\beta_i, s_i)_i}P_\mathbb{H}\times [e^{(\frac{1}{2}\sum_{i=1}^m\beta_i - Q)c}dc]$ with
	$$C_{\mathbb{H}}^{(\beta_i, s_i)_i} =
	\left\{ \begin{array}{rcl} 
\prod_{i=1}^m  |s_i|_+^{-\beta_i(Q-\frac{\beta_i}{2})} \exp(\frac{1}{4}\sum_{j=i+1}^{m}\beta_i\beta_j G_\mathbb{H}(s_i, s_j)) & \mbox{if} & s_1\neq \infty\\
	\prod_{i=2}^m  |s_i|_+^{-\beta_i(Q-\frac{\beta_i}{2}-\frac{\beta_1}{2})}\exp(\frac{1}{4}\sum_{j=i+1}^{m}\beta_i\beta_j G_\mathbb{H}(s_i, s_j)) & \mbox{if} & s_1= \infty
	\end{array} 
	\right.
	 $$
	and then taking
	\begin{equation}\label{eqn-def-lf-H}
\phi(z) = h(z) - 2Q\log|z|_++\frac{1}{2}\sum_{i=1}^m\beta_i G_\mathbb{H}(s_i, z)+\mathbf{c}
	\end{equation}
with the convention $G_\mathbb{H}(\infty, z) = 2\log|z|_+$. We write $\textup{LF}_{\mathbb{H}}^{(\beta_i, s_i)_i}$ for the law of $\phi$.
\end{definition}

The following lemma explains that adding a $\beta$-insertion point at $s\in\partial\mathbb{H}$  is equivalent to weighting the law of Liouville field $\phi$ by $e^{\frac{\beta}{2}\phi(s)}$. 

\begin{lemma}\label{lm:lf-insertion-bdry}
	For $\beta, s\in\mathbb{R}$ such that $s\notin\{s_1, ..., s_m\}$, in the sense of vague convergence of measures,
	\begin{equation}
	\lim_{\epsilon\to 0}\epsilon^{\frac{\beta^2}{4}}e^{\frac{\beta}{2}\phi_\epsilon(s)}\textup{LF}_{\mathbb{H}}^{(\beta_i, s_i)_i} = \textup{LF}_{\mathbb{H}}^{(\beta_i, s_i)_i, (\beta, s)}.
	\end{equation}
 Similarly, for $a,s_1,...,s_m\in\bbR$, we have
 \begin{equation}
\lim_{R\to\infty}R^{Q\beta-\frac{\beta^2}{4}}e^{\frac{\beta}{2}\phi_R(a)}    \LF_{\bbH}^{(\beta_i,s_i)_i} =  \LF_{\bbH}^{(\beta_i,s_i)_i,(\beta,\infty)}.
\end{equation}
\end{lemma}
\begin{proof}
The first part is precisely~\cite[Lemma 2.6]{AHS21}. For the second part, set $\psi(z) = \sum_{i=1}^m\frac{\beta_i}{2}G_\bbH(z,s_i)-2Q\log|z|_+$ and $s_0 = \frac{1}{2}\sum_i\beta_i-Q$. Then as $R\to\infty$, $\psi_R(a) = -2Q\log R+2\sum_{i=1}^m\alpha_i\log|z_i|_++\sum_{j=1}^n{\beta_j}\log|s_j|_++o_R(1)$. One also has $\mathrm{Var}(h_R(a)) = 2\log R + o_R(1)$. Then for any non-negative continuous functions $F$ on $H^{-1}(\bbH)$, one has
\begin{equation*}
\begin{split}
    &\lim_{R\to\infty}\int\int R^{Q\beta-\frac{\beta^2}{4}}C_{\bbH}^{(\beta_i,s_i)_i}F(h+\psi+c)e^{s_0c}e^{\frac{\beta}{2}(h_R(0)+\psi_R(0)+c)}P_\bbH(dh)\ dc\\ &= \lim_{R\to\infty}\int\int R^{Q\beta-\frac{\beta^2}{4}}C_{\bbH}^{(\beta_i,s_i)_i} F(h+\psi+c+\frac{\beta}{2}G_R(\cdot,a))e^{(s_0+\frac{\beta}{2})c}\bbE[e^{\frac{\beta}{2}h_R(0)}]e^{\frac{\beta}{2}\psi_R(0)}P_\bbH(dh)\ dc\\
    &=\int\int C_{\bbH}^{(\beta_i,s_i)_i,(\beta,\infty)} F(h+\psi+\beta\log|z|_++c)e^{(s_0+\frac{\beta}{2})c}P_\bbH(dh)\ dc.
\end{split}
\end{equation*}
Here $G_R(z,a) = \bbE[h(z)h_R(a)]$ and  we have applied the Girsanov's theorem. Therefore the lemma follows.
\end{proof}
In general, the Liouville fields has nice compatibility with the notion of quantum surfaces. To be more precise, for a measure $M$ on the space of distributions on a domain $D$ and a conformal map $\psi:D\to \tilde{D}$, let $\psi_* M$ be the push-forward of $M$ under the mapping $\phi\mapsto\psi\bullet_\gamma\phi$. Then we have the following conformal covariance of the Liouville field. For $\beta\in\bbR$, we use the shorthand 
$$\Delta_\beta:=\frac{\beta}{2}(Q-\frac{\beta}{2}).$$
 {Also recall that as defined above Section~\ref{subsec:pre-lqg}, for a M\"{o}bius map $f:\bbH\to\bbH$, if $f(z) = a+\frac{\lambda}{x-z}$, then $f'(x) = \lambda^{-1}$ and $f'(\infty) = \lambda$. If $f(z) = a+rz$ with $a\in\bbR,r>0$, then we write $f'(\infty)=r^{-1}$. }

\begin{lemma}\label{lm:lcft-H-conf}
	Fix $(\beta_i, s_i)\in\mathbb{R}\times\partial \bbH $ for $i=1, ..., m$ with $s_i$'s being distinct. Here $\partial \bbH = \bbR\cup\{\infty\}$.  Suppose $\psi:\mathbb{H}\to\mathbb{H}$ is conformal map. Then $\textup{LF}_{\mathbb{H}} = \psi_*\textup{LF}_{\mathbb{H}}$, and 
	\begin{equation}
	\textup{LF}_{\mathbb{H}}^{(\beta_i, \psi(s_i))_i} = \prod_{i=1}^m|\psi'(s_i)|^{-\Delta_{\beta_i}}\psi_*\textup{LF}_{\mathbb{H}}^{(\beta_i, s_i)_i}.
	\end{equation} 
\end{lemma}
\begin{proof}
    When none of $s_i,\psi(s_i)$ is $\infty$, the lemma is proved in \cite[Proposition 2.7]{AHS21}. For the remaining part, we work with the case  $m=2$, $\psi(z) = a+\frac{\lambda}{x-z}$, $(s_1,s_2) = (x,\infty)$ where $a,x\in\bbR$ and $\lambda>0$, while the general case follows similarly. 

    Let $\phi$ be a sample from $\LF_\bbH$ and we weight its law by $R^{\beta_2 Q-\frac{\beta_2^2}{4}}e^{\frac{\beta_2}{2}\phi_R(x)}\e^{\frac{\beta_1^2}{4}}e^{\frac{\beta_1}{2}\phi_\e(x)}$. On one hand, if we send $\e\to0,R\to\infty$, by Lemma~\ref{lm:lf-insertion-bdry}, the law of $\phi$ converges in vague topology to $\LF_\bbH^{(\beta_1,x),(\beta_2,\infty)}$. On the other hand, the law of $\bar{\phi}:=\psi\bullet_\gamma\phi$ is weighted by $\lambda^{\frac{(\beta_1-\beta_2)Q}{2}}\e^{\frac{\beta_1^2}{4}-\beta_1Q}R^{-\frac{\beta_2^2}{4}}e^{\frac{\beta_1}{2}\bar{\phi}_{\frac{\lambda}{\e}}(a)+ \frac{\beta_2}{2}\bar{\phi}_{\frac{\lambda}{R}}(a)}$ and thus converges in vague topology to $\lambda^{\Delta_{\beta_2}-\Delta_{\beta_1}}\LF_\bbH^{(\beta_1,\infty),(\beta_2,a)}$ by Lemma~\ref{lm:lf-insertion-bdry}. Therefore the claim follows by definition that $\psi'(x)=\lambda^{-1}$ and $\psi'(\infty)=\lambda$.
\end{proof}

Next we recall the relations between marked quantum disks and Liouville fields. The statements in~\cite{AHS21} are involving Liouville fields on the strip $\mathcal{S}:=\bbR\times(0,\pi)$, yet we can use the map $z\mapsto e^z$ to transfer to the upper half plane.

\begin{definition}\label{def:three-pointed-disk}
	Fix $W>0$. First sample  $(D,\phi, x,y)$ from $\mathcal{M}_2^{\textup{disk}}(W)$ and weight its law by the quantum length of its left boundary.  Then sample a point $s$ on the left boundary arc according to the probability measure proportional to $\nu_{\phi}$. We denote the law of the surface  $(D,\phi, x,y,s){/\sim_\gamma}$ by $\mathcal{M}_{2, \bullet}^{\textup{disk}}(W)$. 
\end{definition}

\begin{proposition}[Proposition 2.18 of \cite{AHS21}]\label{prop:m2dot}
	For $W>\frac{\gamma^2}{2}$ and $\beta = \gamma+\frac{2-W}{\gamma}$, let $\phi$ be sampled from $\frac{\gamma}{2(Q-\beta)^2}\textup{LF}_{\mathbb{H}}^{(\beta,\infty), (\beta,0), (\gamma, 1)}$. Then  {the law of $(\bbH, \phi, 0, \infty, 1)/{\sim_\gamma}$ is $\mathcal{M}_{2, \bullet}^{\textup{disk}}(W)$}.
\end{proposition}

The proposition above gives rise to the quantum disks with general third insertion points, which could be defined via three-pointed Liouville fields.

\begin{definition}\label{def:m2dot-alpha}
	Fix $W>\frac{\gamma^2}{2},\beta=\gamma+\frac{2-W}{\gamma}$ and let $\beta_3\in\mathbb{R}$. Set $\mathcal{M}_{2, \bullet}^{\textup{disk}}(W;\beta_3)$ to be the law of $(\bbH, \phi, 0,\infty, 1)/\sim_\gamma$ with $\phi$ sampled from $\frac{\gamma}{2(Q-\beta)^2}\textup{LF}_{\bbH}^{(\beta,0),(\beta,\infty), (\beta_3, 1)}$. We call the boundary arc between the two $\beta$-singularities with (resp. not containing) the $\alpha$-singularity the marked (resp. unmarked) boundary arc.
\end{definition}

In general, for a Liouville field $\LF_\bbH^{(\beta_i,s_i)_i}$, sampling a marked point from the boundary quantum length measure corresponds to adding a $\gamma$-insertion to the field.
\begin{lemma}\label{lm:gamma-insertion}
    Let $m\ge 2,n\ge1$ and $(\beta_i,s_i)\in \bbR\times\partial \bbH$ with  {$ s_1<s_2<...<s_m\leq +\infty$}. Let $1\le k\le m-1$ and $g$ be a non-negative measurable function supported on  {$[s_{k},s_{k+1}]^n$}. Then as measures we have the identity
\begin{equation}\label{eq:gamma-insertion}
\begin{split}
    & {\mathds{1}_{x_1,...,x_n\in[{s_{k}},{s_{k+1}}]} g(x_1,...,x_n)\nu_\phi(dx_1)...\nu_\phi(dx_n) \LF_\bbH^{(\beta_i,s_i)_i}(d\phi)} \\&= \mathds{1}_{x_1,...,x_n\in[{s_{k}},{s_{k+1}}]} \LF_\bbH^{(\beta_i,s_i)_i,(\gamma,x_1),...,(\gamma,x_n)}(d\phi)\, g(x_1,...,x_n)dx_1...dx_n.
    \end{split}
\end{equation}
\end{lemma}
\begin{proof}
    First assume that $g$ is bounded with compact support. Then for any non-negative measurable function $F$ on $H^{-1}(\bbH)$, 
    \begin{equation}
        \begin{split}
            &\int\int_{x_1,...,x_n\in[{s_{k+1}},{s_k}]}\, F(\phi)g(x_1,...,x_n)\nu_\phi(dx_1)...\nu_\phi(dx_n) \LF_\bbH^{(\beta_i,s_i)_i}(d\phi)\\& = \lim_{\e\to0}\int_{x_1,...,x_n\in[{s_{k+1}},{s_k}]}\int\, \prod_{i=1}^n(\e^{\frac{\gamma^2}{4}}e^{\frac{\gamma}{2}\phi_\e(x_i)})F(\phi)\LF_\bbH^{(\beta_i,s_i)_i}(d\phi)\, g(x_1,...,x_n)dx_1...dx_n\\
            &=\int_{x_1,...,x_n\in[{s_{k+1}},{s_k}]}\int\,F(\phi) \LF_\bbH^{(\beta_i,s_i)_i,(\gamma,x_1),...,(\gamma,x_n)}(d\phi)\, g(x_1,...,x_n)dx_1...dx_n
        \end{split}
    \end{equation}
    where we have applied Lemma~\ref{lm:lf-insertion-bdry}. For general non-negative measurable function $g$, we may take $g_n\uparrow g$ where $g_n$ is bounded with compact support.
\end{proof}

As a consequence, we have the following lemma on sampling two quantum typical points on a weight $W$ quantum disk.
\begin{lemma}\label{lm:qd-4-pted}
   Fix $W>\frac{\gamma^2}{2}$ and $\beta = \gamma+\frac{2-W}{\gamma}$. First sample  {a quantum disk from  $\frac{1}{2}L_R^2\mathcal{M}_2^{\textup{disk}}(W)$ where $L_R$ is its right boundary length, and embed it as $(\mathbb{H},\phi, 0, \infty)$.} Then sample $x,y\in\mathbb{R}_+$ according to the probability measure proportional to $1_{x<y}\nu_{\phi}|_{\mathbb{R}_+}(dx)\nu_{\phi}|_{\mathbb{R}_+}(dy)$. Then the law of the quantum surface $(\mathbb{H},\phi, 0, x,y,\infty)/\sim_\gamma$, denoted by $\Md_{2,\bullet,\bullet}(W)$, agrees with that of $(\mathbb{H},\phi_0, 0, x,1,\infty)/\sim_\gamma$, where $(\phi_0,x)$ is sampled from
\begin{equation}\label{eq:lm-qt-4-pted}
   {  \frac{\gamma}{2(Q-\beta)^2}\cdot\mathds{1}_{x\in(0,1)} \LF_\bbH^{(\beta,0),(\beta,\infty),(\gamma,x),(\gamma,1)}(d\phi_0)\,dx.}
\end{equation}   
\end{lemma}
\begin{proof}
    By Proposition~\ref{prop:m2dot} and Lemma~\ref{lm:gamma-insertion}, the quantum surface $(\mathbb{H},\phi, 0, x,y,\infty)/\sim_\gamma$ agrees in law with $(\bbH,\phi_1,0,x,1,\infty)/\sim_\gamma$ where $(\phi_1,x)$ is sampled from
    \begin{equation}\label{eq:lm-qt-4-pted-1}
    \begin{split}
    &\frac{\gamma}{2(Q-\beta)^2} {\mathds{1}_{x\in(0,1)}}\nu_{\phi_1}(dx)\LF_\bbH^{(\beta,0),(\beta,\infty),(\gamma,1)}(d\phi_1) = \frac{\gamma}{2(Q-\beta)^2} {\mathds{1}_{x\in(0,1)}}\LF_\bbH^{(\beta,0),(\beta,\infty),(\gamma,1),(\gamma,x)}(d\phi_1)dx.\\
    \end{split}
\end{equation} 
\end{proof}

\subsection{SLE and Multiple SLE}\label{subsec:pre-sle}
In this section, we review the background and some basic properties of the chordal SLE and multiple SLE as established in e.g.~\cite{lawler2009partition,peltola2019global}.

Fix $\kappa>0$. We start with the $\SLE_\kappa$ process on the upper half plane $\bbH$. Let $(B_t)_{t\ge0}$ be the standard Brownian motion. The $\SLE_\kappa$ is the probability measure on continuously growing  {curves $\eta$ in $\overline{\bbH}$},  {whose mapping out function $(g_t)_{t\ge0}$ (i.e., the unique conformal transformation from the unbounded component of $\mathbb{H}\backslash \eta([0,t])$ to $\mathbb{H}$ such that $\lim_{|z|\to\infty}|g_t(z)-z|=0$) can be described by}
\begin{equation}\label{eq:def-sle}
g_t(z) = z+\int_0^t \frac{2}{g_s(z)-W_s}ds, \ z\in\mathbb{H},
\end{equation} 
where $W_t=\sqrt{\kappa}B_t$ is the Loewner driving function.


The $\SLE_\kappa$, as a probability measure, can be defined on other domains by conformal maps. To be more precise, let $\mu_\bbH(0,\infty)$ be the $\SLE_\kappa$ on $\bbH$ from 0 to $\infty$, $D$ be a simply connected domain, and $f:\bbH\to D$ be a conformal map with $f(0)=x,f(\infty)=y$. Then we can define a probability measure $\mu_D({x,y})^\# = f\circ\mu_\bbH(0,\infty)$. Let $$b=\frac{6-\kappa}{2\kappa}$$ be the \emph{boundary scaling exponent}, and recall that for $x,y\in\partial D$ such that $\partial D$ is smooth near $x,y$, the boundary Poisson kernel is defined by $H_D(x,y)=\varphi'(x)\varphi'(y)(\varphi(x)-\varphi(y))^{-2}$ where $\varphi:D\to\bbH$ is a conformal map. Then as in~\cite{lawler2009partition}, one can define the $\SLE_\kappa$ in $(D,x,y)$ as a non-probability measure by setting $\mu_D({x,y}) = H_D(x,y)^b\cdot\mu_D({x,y})^\#$, which satisfies the conformal covariance
$$f\circ\mu_{D}(x,y) = |f'(x)|^b|f'(y)|^b\mu_{f(D)}(f(x),f(y))  $$
for any conformal map $f:D\to f(D)$.

Let $\kappa\in(0,4]$, $\alpha\in\mathrm{LP}_N$, $D\subset\bbC$ be a simply connected domain with $x_1,...,x_{2N}\in\partial D$ be $2N$ distinct marked points in counterclockwise order. 
Let $\bbP_{\alpha}^{\otimes N}(D;x_1,...,x_{2N}):=\otimes_{k=1}^N\mu_D(x_{i_k},x_{j_k})^\#$ be the law of $N$ independent chordal $\SLE_\kappa$ in $D$ following the link pattern $\alpha$. Then $\mathrm{mSLE}_{\kappa,\alpha}(D;x_1,...,x_{2N})^\#$, the {global $N$-$\SLE_\kappa$ associated with $\alpha$}, is absolutely continuous with respect to $\bbP_\alpha^{\otimes N}(D;x_1,...,x_{2N})$ with the Radon-Nikodym derivative written in terms of the Brownian loop measure~\cite{kozdron2006configurational,lawler2009partition,peltola2019global}. 
Moreover, the measure $\mathrm{mSLE}_{\kappa,\alpha}$ is conformally invariant, in the sense that
\begin{equation}\label{eq:msle-prob-conf}
    f\circ\mathrm{mSLE}_{\kappa,\alpha}(D;x_1,...,x_{2N})^\# = \mathrm{mSLE}_{\kappa,\alpha}(f(D);f(x_1),...,f(x_{2N}))^\#
\end{equation}
for any conformal map $f:D\to f(D)$.

The measure $\mathrm{mSLE}_{\kappa,\alpha}(D;x_1,...,x_{2N})$ is defined by multiplying the  measure $\mathrm{mSLE}_{\kappa,\alpha}(D;x_1,...,x_{2N})^\#$ by the {pure partition function} $\mathcal{Z}_\alpha(D;x_1,...,x_{2N})$. As proved in~\cite[Theorem 1.1]{peltola2019global}, when $D=\bbH$ and $x_1<...<x_{2N}$, $\mathcal{Z}_\alpha(D;x_1,...,x_{2N})$ is the solution to the PDE
\begin{equation}\label{eq:msle-pde}
    \bigg[\frac{\kappa}{2}\partial_i^2+\sum_{j\neq i}\big(\frac{2}{x_j-x_i}\partial_j-\frac{2b}{(x_j-x_i)^2} \big)   \bigg]\mathcal{Z}_\alpha(\bbH;x_1,...,x_{2N}) = 0\ \ \ \text{for\ } i=1,...,2N,
\end{equation}
satisfies the conformal covariance
\begin{equation}\label{eq:msle-conformal-conf}
    \mathcal{Z}_\alpha(\bbH;x_1,...,x_{2N}) = \prod_{i=1}^{2N}f'(x_i)^b\times \mathcal{Z}_\alpha(\bbH;f(x_1),...,f(x_{2N}))
\end{equation}
for any conformal map $f:\bbH\to\bbH$ with $f(x_1)<...<f(x_{2N})$, and has certain asymptotic behavior when $x_k,x_{k+1}\to\xi$ for fixed $\xi\in(x_{k-1},x_{k+2})$. By~\eqref{eq:msle-prob-conf} and~\eqref{eq:msle-conformal-conf}, this allows us to define the measure $\mathrm{mSLE}_{\kappa,\alpha}(D;x_1,...,x_{2N})$ via
\begin{equation}\label{eq:msle-conf}
    f\circ\mathrm{mSLE}_{\kappa,\alpha}(D;x_1,...,x_{2N}) =\prod_{i=1}^{2N}f'(x_i)^b \mathrm{mSLE}_{\kappa,\alpha}(f(D);f(x_1),...,f(x_{2N}))
\end{equation}
for any conformal map $f:D\to f(D)$ such that $\partial D$ (resp.\ $\partial f(D)$) is smooth near every $x_i$ (resp.\ $f(x_i)$).

Finally,  the curves from $\mathrm{mSLE}_{\kappa,\alpha}(D;x_1,...,x_{2N})$ can be sampled in a recursive way.  Let $\{i,j\}\in\alpha$ be a link dividing $\alpha$ into two sub-link patterns $\alpha^L,\alpha^R$, and $\eta$ be the curve connecting $x_i$ and $x_j$ a sample from $\mathrm{mSLE}_{\kappa,\alpha}(D;x_1,...,x_{2N})$. Let $D_\eta^L$ (resp.\ $D_\eta^R$) be the left (resp.\ right) connected component of $D\backslash\eta$, and $\{s_1,...,s_{2k}\}$ (resp. $\{w_1,...,w_{2N-2-2k}\}$) be the subset of  $\{x_1,...,x_{2N}\}$ lying to the left (resp.\ right) of $\eta$. Then the following immediately follows from~\cite[Proposition 3.5]{peltola2019global}.
\begin{proposition}\label{prop:msle-margin}
    The marginal law of $\eta$ under $\mathrm{mSLE}_{\kappa,\alpha}(D;x_1,...,x_{2N})^\#$ is absolutely continuous with respect to the law $\mu_D(x_i,x_j)^\#$ of $\SLE_\kappa$ in $D$ connecting $x_i$ and $x_j$ with Radon-Nikodym derivative
    \begin{equation}
        \frac{H_D(x_i,x_j)^b}{\mathcal{Z}_\alpha(D;x_1,...,x_{2N})}\times \mathcal{Z}_{\alpha^L}(D_\eta^L;s_1,...,s_{2k})\times \mathcal{Z}_{\alpha^R}(D_\eta^R;y_1,...,y_{2N-2k-2}).
    \end{equation}
    In particular, a sample from $\mathrm{mSLE}_{\kappa,\alpha}(D;x_1,...,x_{2N})$ can be produced by 
    \begin{enumerate}[(i)]
        \item Sample $\eta$ from the measure 
    $$ \mathcal{Z}_{\alpha^L}(D_\eta^L;s_1,...,s_{2k})\times \mathcal{Z}_{\alpha^R}(D_\eta^R;y_1,...,y_{2N-2k-2})\, \mu_D(x_i,x_j); $$
    \item Sample $(\eta_1^L,...,\eta_k^L),(\eta_1^R,...,\eta_{N-1-k}^R)$ from the probability measure $$\mathrm{mSLE}_{\kappa,\alpha^L}(D_\eta^L;s_1,...,s_{2k})^\#\times \mathrm{mSLE}_{\kappa,\alpha^R}(D_\eta^R;y_1,...,y_{2N-2k-2})^\#; $$  
    \item Output $(\eta;\eta_1^L,...,\eta_k^L;\eta_1^R,...,\eta_{N-1-k}^R)$.
    \end{enumerate}
    
\end{proposition}

\section{Conformal welding and multiple SLE}\label{sec:proof}
In this section,  {we first prove} Theorem~\ref{thm:N=2}. The proof is based on Theorem~\hyperref[thm:disk-welding]{A} and adding additional marked points as in~\cite{AHS21}. 
Then by an induction, we prove Theorem~\ref{thm:main} via the conformal covariance of the Liouville field and multiple $\SLE_\kappa$.

For $\beta\in\bbR$, $\rho^-,\rho^+>-2$, define the measure $\wt{\SLE}_\kappa(\rho^-;\rho^+;\beta)$ on curves $\eta$ from 0 to $\infty$ on $\mathbb{H}$ as follows. Let $D_\eta$ be the component of $\mathbb{H}\backslash \eta$ containing $1$, and $\psi_\eta$ the unique conformal map from $D_\eta$ to $\mathbb{H}$ fixing 1 and sending the first (resp. last) point on $\partial D_\eta$ hit by $\eta$ to 0 (resp. $\infty$). Then our $\widetilde{\SLE}_\kappa(\rho_-;\rho_+;\beta)$ on $\bbH$ is defined by
\begin{equation}\label{eqn-sle-CR-1}
\frac{d \widetilde{\SLE}_\kappa(\rho_-;\rho_+;\beta)}{d {\SLE}_\kappa(\rho_-;\rho_+)}(\eta) = |\psi_{\eta}(1)'|^\beta.
\end{equation} 
This definition can be  extended to other domains via conformal transforms.

We begin with the following variant of Theorem~\hyperref[thm:disk-welding]{A} proved in~\cite{AHS21}.  See the left panel of Figure~\ref{fig:disk34pt}.

\begin{proposition}[Proposition 4.5 of~\cite{AHS21}]\label{prop:3-pt-disk}
	Suppose $W_1>0,W_2>\frac{\gamma^2}{2}$,  {and $c_{W_1, W_2}\in (0,\infty)$ is the constant in  Theorem~\hyperref[thm:disk-welding]{A}.}  Then for all $\beta\in\mathbb{R}$, 
	\begin{equation}\label{eqn-3-pt-disk}
	\begin{split}
	\mathcal{M}_{2,\bullet}^{\textup{disk}}(W_1+W_2;\beta)& {\times} \widetilde{\SLE}_\kappa(W_1-2;W_2-2;1-\Delta_{\beta})= c_{W_1, W_2}\Wd(\mathcal{M}_2^{\textup{disk}}(W_1), \mathcal{M}_{2,\bullet}^{\textup{disk}}(W_2;\beta)),
	\end{split}
	\end{equation} 
	where we are welding along the unmarked boundary arc of  $\mathcal{M}_{2,\bullet}^{\textup{disk}}(W_2;\beta)$ and  $\Delta_\beta = \frac{\beta}{2}(Q-\frac{\beta}{2})$.
\end{proposition}

\begin{figure}
    \centering
    \begin{tabular}{ccc} 
		\includegraphics[scale=0.6]{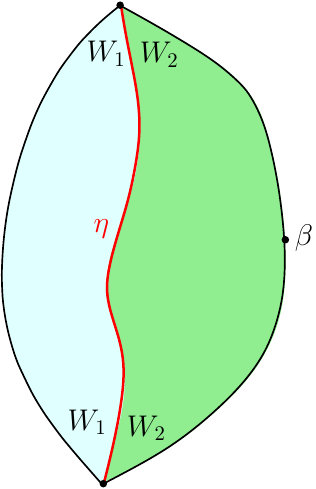}
		& \qquad \qquad \qquad &
		\includegraphics[scale=0.6]{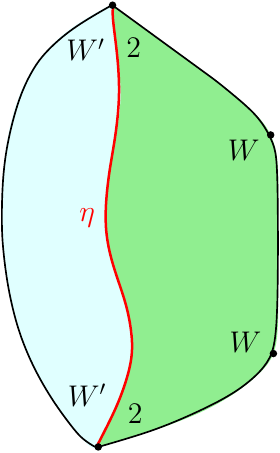}
	\end{tabular}
 \caption{\textbf{Left:} The conformal welding of a weight $W_1$ quantum disk with a sample from $\mathcal{M}_{2,\bullet}^{\textup{disk}}(W_2;\beta)$ along the unmarked boundary arc as in Proposition~\ref{prop:3-pt-disk}. \textbf{Right:} The setup of Proposition~\ref{prop:disk+4-pt-disk}, where we first sample two boundary typical points from LQG length measure of a weight $W$ quantum disk, and then conformally weld a weight $W'$ quantum disk along the two marked points.}\label{fig:disk34pt}
 \end{figure}

We also need  the argument of changing weight of insertions in the conformal welding. We begin with the following disintegration of Liouville fields according to quantum lengths.
\begin{lemma}\label{lm:lf-length}
    Let $m\ge 2$ and $0=s_1<s_2<...<s_m=+\infty$. Fix $\beta_1,...,\beta_m<Q$.  Let $C_\bbH^{(\beta_i,s_i)_i}$ and $P_\bbH$ be as in Definition~\ref{def-lf-H-bdry}, and  
    $\tilde{h} = h - 2Q\log|\cdot|_++\frac{1}{2}\sum_{i=1}^m\beta_iG_\bbH(s_i,\cdot)$, and $L=\nu_{\tilde{h}}((-\infty,0))$. For $\ell>0$, let $\LF_{\bbH,\ell}^{(\beta_i,s_i)_i}$ be the law of  {$\tilde{h}+\frac{2}{\gamma}\log\frac{\ell}{L}$} under the reweighted measure $\frac{2}{\gamma}\frac{\ell^{\frac{1}{\gamma}(\sum_j\beta_j -2Q)-1}}{L^{\frac{1}{\gamma}(\sum_j\beta_j -2Q)}}C_\bbH^{(\beta_i,s_i)_i}P_\bbH(dh)$. Then $\LF_{\bbH,\ell}^{(\beta_i,s_i)_i}$ is supported on $\{\phi:\nu_\phi((-\infty,0))=\ell\}$, and we have 
    \begin{equation}
       \LF_{\bbH}^{(\beta_i,s_i)_i} = \int_0^\infty\LF_{\bbH,\ell}^{(\beta_i,s_i)_i}d\ell.
    \end{equation}
\end{lemma}
\begin{proof}
    The proof is identical to that of~\cite[Lemma 2.27]{ASY22}.
\end{proof}
\begin{lemma}\label{lm:lf-change-weight}
   In the setting of Lemma~\ref{lm:lf-length}, {let  {$\e\in(0,1)$}. For fixed $2<j<m$, and any non-negative function $f$ on $H^{-1}(\bbH)$ for which $\phi\mapsto f(\phi)$ only depends on $\phi|_{\bbH\backslash B_\e(s_j)}$, we have
   \begin{equation}
       \int f(\phi)\times\e^{\frac{\beta_j'^2-\beta_j^2}{4}}e^{\frac{\beta_j'-\beta_j}{2}\phi_\e(s_j)}\LF_{\bbH,\ell}^{(\beta_i,s_i)_i}(d\phi) = \int f(\phi)\LF_{\bbH,\ell}^{(\beta_i,s_i)_{i\neq j}, (\beta_j',s_j)}(d\phi).
   \end{equation}
   In particular, }we have the vague convergence of measures
    \begin{equation*}
        \lim_{\e\to 0}\e^{\frac{\beta_j'^2-\beta_j^2}{4}}e^{\frac{\beta_j'-\beta_j}{2}\phi_\e(s_j)}\LF_{\bbH,\ell}^{(\beta_i,s_i)_i}(d\phi) = \LF_{\bbH,\ell}^{(\beta_i,s_i)_{i\neq j}, (\beta_j',s_j)}(d\phi).
    \end{equation*}
\end{lemma}
 \begin{proof}
        The proof is identical to that of \cite[Lemma 4.7]{ARS21}, which is a direct application of the Girsanov theorem. We omit the details. 
    \end{proof}
    
Let $W>\frac{\gamma^2}{2}$ and recall the notion $\Md_{2,\bullet,\bullet}(W)$ from Lemma~\ref{lm:qd-4-pted}. We are going to show that, for a sample from $\Md_{2,\bullet,\bullet}(W)$, if we weld a weight $W'$ quantum disk along its boundary arc between the two quantum typical points, the interface we obtain is the $\SLE_\kappa(W'-2)$ process weighted by a power of the boundary Poisson kernel. See the right panel of Figure~\ref{fig:disk34pt} for an illustration.

\begin{proposition}\label{prop:disk+4-pt-disk}
    Let $W>\frac{\gamma^2}{2}$, $W'>0$, $\beta=\gamma+\frac{2-W}{\gamma}$ and $\tilde\beta = \gamma-\frac{W'}{\gamma}$. Let $b_W=\frac{(W-2)(W+2-\kappa)}{4\kappa}$. Then there is a constant $c\in(0,\infty)$ such that, in the sense of curve-decorated quantum surfaces,
    \begin{equation}\label{eq:disk+4-pt-disk}
        c\,\bigg[\int_0^1 \LF_\bbH^{(\tilde\beta,0),(\beta,x),(\beta,1),(\tilde\beta,\infty)}\times H_{D_\eta}{(x,1)}^{b_W}dx\bigg]\cdot\big(\SLE_\kappa(W'-2;0)\big)(d\eta)= \Wd(\Md_2(W'),\Md_{2,\bullet,\bullet}(W))
    \end{equation}
    where $D_\eta$ is the component of $\bbH\backslash\eta$ to the right of $\eta$, and the welding is along its boundary arc between the two quantum typical points of the surface from $\Md_{2,\bullet,\bullet}(W)$.
\end{proposition}
\begin{proof}
\emph{Step 1: Embedding of $\Md_{2,\bullet,\bullet}(W)$.} By Lemma~\ref{lm:qd-4-pted}, a sample from $\Md_{2,\bullet,\bullet}(W)$ can be embedded as $$\frac{\gamma}{2(Q-\beta)^2}\cdot\mathds{1}_{y'\in(0,1)} \LF_\bbH^{(\beta,0),(\beta,\infty),(\gamma,y'),(\gamma,1)}(d\phi_0)\,dy'.$$ On the other hand, if we perform the coordinate change $z\mapsto f_{y'}(z):=\frac{{y'}(z-1)}{z-{y'}}$, then by Lemma~\ref{lm:lcft-H-conf} ( {also recall the definitions of $f_{y'}'({y'})$ and $f_{y'}'(\infty)$ there and at the beginning of Section~\ref{sec:pre}}),  {the  quantum surface $(\bbH,\phi_0,0,y',1,\infty)/\sim_\gamma$ described as above has the same law as  $(\bbH,\phi,0,{y'},1,\infty)/\sim_\gamma$ where $(\phi,{y'})$ has law} 
\begin{equation}\label{eq:lm-qt-4-pted-2}
\begin{split}
    &\ \ \frac{\gamma}{2(Q-\beta)^2} \mathds{1}_{{y'}\in(0,1)} f_{y'}'(0)^{\Delta_\beta}f_{y'}'(\infty)^{\Delta_\beta}f_{y'}'({y'})^{\Delta_\gamma}f_{y'}'(1)^{\Delta_\gamma}\LF_\bbH^{(\beta,f_{y'}(0)),(\beta,f_{y'}(\infty)),(\gamma,f_{y'}({y'})),(\gamma,f_{y'}(1))}(d\phi)\,d{y'} \\&=\frac{\gamma}{2(Q-\beta)^2}\mathds{1}_{{y'}\in(0,1)} (1-{y'})^{2\Delta_\beta-2}\, \LF_\bbH^{(\gamma,0),(\beta,{y'}),(\beta,1),(\gamma,\infty)}(d\phi)\,d{y'}.
    \end{split}
\end{equation}

\emph{Step 2: Add a typical point to the welding of $\Md_{2}(W')$ and $\Md_{2,\bullet}(2;\beta)$.} Consider the welding in Proposition~\ref{prop:3-pt-disk} with $W_1=W'$, $W_2=2$ {and the constant $c':=c_{W',2}$}, and the surface $S$ on the left hand side of~\eqref{eqn-3-pt-disk} is embedded as  {$(\bbH,X,\eta,0,\infty,1)$} where  {$X\sim \frac{\gamma}{2(Q-\tilde\beta)^2}\LF_\bbH^{(\tilde\beta,0),(\tilde\beta,\infty),(\beta,1)}$}, and $\eta$ is the $\SLE_\kappa(W-2;0)$ process in $\bbH$ weighted by $\psi_\eta'(1)^{1-\Delta_\beta}$. Let 
\begin{equation}\label{eq:lm-qt-4-pted-4}
  {Y=X\circ\psi_\eta^{-1}+Q\log|(\psi_\eta^{-1})'|,} 
\end{equation}
$D_\eta^1$ be the union of the components of $\bbH\backslash\eta$ whose boundaries contain a segment of $(-\infty,0)$. Let $S_1=(D_\eta^1,X)/\sim_\gamma$, and $S_2 = (\bbH,Y,0,\infty,1)/\sim_\gamma$. We sample a marked point $y$ on $S_2$ from the measure $1_{y\in(0,1)}(1-y)^{2\Delta_\beta-2}\nu_Y(dy)$. Then by Lemma~\ref{lm:gamma-insertion}, the surface $S$ is now the conformal welding of a weight $W'$ quantum disk with a 4-pointed quantum disk with embedding
\begin{equation}
    \frac{c'\gamma}{2(Q-\gamma)^2}\int_0^1 (1-y)^{2\Delta_\beta-2}\, \LF_\bbH^{(\gamma,0),(\gamma,y),(\beta,1),(\gamma,\infty)}(d\phi)\,dy.
\end{equation}
On the other hand, for $x=\psi_\eta^{-1}(y)$, the law of $(X,\eta,x)$ is given by
$$ \frac{\gamma}{2(Q-\tilde\beta)^2}\bigg[\mathds{1}_{x\in(0,1)}(1-\psi_\eta(x))^{2\Delta_\beta-2}\,\nu_Y(dx)\,\LF_\bbH^{(\tilde\beta,0),(\tilde\beta,\infty),(\beta,1)}(dX)\bigg]\cdot\psi_\eta'(1)^{1-\Delta_\beta}\big(\SLE_\kappa(W'-2;0)\big)(d\eta) $$
which, by Lemma~\ref{lm:gamma-insertion}, is equal to 
\begin{equation}\label{eq:lm-qt-4-pted-3}
    \frac{\gamma}{2(Q-\tilde\beta)^2}\bigg[\mathds{1}_{x\in(0,1)} (1-\psi_\eta(x))^{2\Delta_\beta-2}\,\LF_\bbH^{(\tilde\beta,0),(\tilde\beta,\infty),(\beta,1),(\gamma,x)}(dX)\,dx\bigg]\cdot\psi_\eta'(1)^{1-\Delta_\beta}\big(\SLE_\kappa(W'-2;0)\big)(d\eta).
\end{equation}

\emph{Step 3: Change the insertion from $\gamma$ to $\beta$.} Weight the law of $(x,X,\eta)$ in~\eqref{eq:lm-qt-4-pted-3} by  {$\frac{(Q-\gamma)^2}{c'(Q-\beta)^2}\e^{\frac{\beta^2-\gamma^2}{4}} e^{\frac{\beta-\gamma}{2}Y_\e(y)}$}, where $Y$ is given by~\eqref{eq:lm-qt-4-pted-4} and $y=\psi_\eta(x)$. Then following from the same argument as in~\cite[Proposition 4.5]{AHS21}, we have:
\begin{enumerate}[(i)]
    \item  {Once we disintegrate over} the interface length $\ell$,  {the joint law of $(S_1,S_2)$   is given by $\Md_2(W';\ell)\times \mu_{2,\e}^\ell$, where $\mu_{2,\e}^\ell$ is the law of $(\bbH,Y,0,\infty,1,y)/\sim_\gamma$ with  $(Y,y)$ sampled from
    $$ {\frac{\gamma}{2(Q-\beta)^2}} \mathds{1}_{y\in(0,1)} \e^{\frac{\beta^2-\gamma^2}{4}} e^{\frac{\beta-\gamma}{2}Y_\e(y)}(1-y)^{2\Delta_\beta-2}\, \LF_{\bbH,\ell}^{(\gamma,0),(\gamma,y),(\beta,1),(\gamma,\infty)}(dY)\,dy. $$
    } By Lemma~\ref{lm:lf-change-weight}, after sending $\e\to0$, the law of $(Y,y)$ converges in vague topology to
    $$ {\frac{\gamma}{2(Q-\beta)^2}}\mathds{1}_{y\in(0,1)} (1-y)^{2\Delta_\beta-2}\, \LF_{\bbH,\ell}^{(\gamma,0),(\beta,y),(\beta,1),(\gamma,\infty)}(dY)\,dy.$$
    In particular, following~\eqref{eq:lm-qt-4-pted-2},  {when disintegrated over the interface length $\ell$, the joint law of $(S_1,S_2)$ equals $\Md_2(W';\ell)\times \Md_{2,\bullet,\bullet}(W;\ell)$, and   $(S_1,S_2)$ has the same law as the conformal welding of $\Md_{2,\bullet,\bullet}(W)$ with $\Md_2(W')$}.
    \item The law of $(X,x,\eta)$ is weighted by 
    \begin{equation}\label{eq:lm-qt-4-pted-5}
        \frac{(Q-\gamma)^2}{c'(Q-\beta)^2}\e^{\frac{\beta^2-\gamma^2}{4}} e^{\frac{\beta-\gamma}{2}\big((X,\theta_\e(x))+Q\log|(\psi_\eta^{-1})'(y)| \big)} = \frac{(Q-\gamma)^2}{c'(Q-\beta)^2}\bigg(\frac{\e}{\psi_\eta'(x)}\bigg)^{\frac{\beta^2-\gamma^2}{4}}e^{\frac{\beta-\gamma}{2}(X,\theta^\eta_\e(x))}\big|\psi_\eta'(x)\big|^{1-\Delta_\beta} 
    \end{equation}
    where $\theta_\e^\eta$ is push-forward of the uniform probability measure on $B_\e(y)\cap\bbH$ under $\psi_\eta^{-1}$ and we used the fact that $\log|(\psi_\eta^{-1})(z)|$ is a harmonic function along with Schwartz reflection. 
    As argued after~\cite[Eq. (4.12)]{AHS21}, by Girsanov's theorem, under the weighting~\eqref{eq:lm-qt-4-pted-5}, as $\e \to0$, the law of $(X,x,\eta)$ converges in vague topology  to 
    \begin{equation}\label{eq:prop-constant-3-4}
        c\,\bigg[\mathds{1}_{x\in(0,1)} \big|\psi_\eta'(x)\big|^{1-\Delta_\beta} (1-\psi_\eta(x))^{2\Delta_\beta-2}\,\LF_\bbH^{(\tilde\beta,0),(\tilde\beta,\infty),(\beta,1),(\beta,x)}(dX)\,dx\bigg]\cdot\psi_\eta'(1)^{1-\Delta_\beta}\big(\SLE_\kappa(W'-2;0)\big)(d\eta).
    \end{equation}
    where $c = \frac{\gamma(Q-\gamma)^2}{2c'(Q-\beta)^2(Q-\tilde\beta)^2}$.  {Intuitively, this is because when $\e\to0$,  $\theta_\e^\eta$ is roughly the uniform measure on $B_{\frac{\e}{\psi_\eta'(y)}}\cap\bbH$ and the conclusion follows by applying Lemma~\ref{lm:lf-change-weight} with $\e$ replaced by $\frac{\e}{\psi_\eta'(y)}$.}
\end{enumerate}
    Therefore we conclude the proof by observing that $H_{D_\eta}(x,1) = \psi_\eta'(x)\psi_\eta'(1)|\psi_\eta(x)-1|^{-2}$ and $b_W=1-\Delta_\beta$.
\end{proof}

\begin{remark}
In the case $W'=2$, one may check by~\cite[Proposition 3.5]{wu2020hypergeometric} that given the marked point $x$,  {the law of the interface $\eta$ above is $H_{D_\eta}(x,1)^{b_W}\SLE_\kappa$, which agrees with $Z(x)\mathrm{hSLE}_\kappa(\nu)$.} Here $\mathrm{hSLE}_\kappa(\nu)$ is the version of hypergeometric $\SLE_\kappa$ considered in~\cite{wu2020hypergeometric} with $\nu=W-4$ and marked points $x,1$,  {and the finite constant $Z(x)$ can be viewed as its partition function}. In this regime $\eta$ is non-boundary hitting, and it would be interesting to understand the analog where $W<\frac{\gamma^2}{2}$ and $\eta$ hits the boundary.
\end{remark}

\begin{proof}[Proof of Theorem~\ref{thm:N=2}]
    Consider the welding of a sample from $\QD_{4}$ and a quantum disk from $\Md_2(W_2)$. By Definition~\ref{def:QD} one can deduce that $\QD_{4}=\Md_{2,\bullet,\bullet}(2)$. Therefore if we start with the welding of a weight 2 quantum disk with a weight $W_2$ quantum disk as in Theorem~\hyperref[thm:disk-welding]{A}, and sample two boundary typical points from the quantum length measure on the weight 2 quantum disk, then it follows that we obtain a sample from $c\,\Md_{2,\bullet,\bullet}(W_2+2)$ decorated with an independent $\SLE_\kappa(0;W_2-2)$ curve (where $c\in(0,\infty)$ is some constant). Therefore the theorem follows immediately by applying Proposition~\ref{prop:disk+4-pt-disk} with $W'=W_1$ and $W=W_2+2$.
\end{proof}

Before proving Theorem~\ref{thm:main}, we first show that the left hand side of~\eqref{eq:thm-main} has the following rotational invariance. Given two link patterns $\alpha = \{\{i_1,j_1\},...,\{i_N,j_N\}\}$ and $\alpha'=\{\{i_1',j_1'\},...,\{i_N',j_N'\}\}$ in $\mathrm{LP}_N$, we say $\alpha$ and $\alpha'$ are rotationally equivalent if there exists some integer $0\le m\le 2N-1$ such that for every $1\le k\le N$, $i_k'=i_k+m$ and $j_k'=j_k+m$ ($\mathrm{mod}$ $2N$), and we write $\alpha'=\alpha+m$.
\begin{lemma}\label{lm:rotation}
    Let $\beta=\gamma-\frac{2}{\gamma}$. For any $0\le m\le 2N-1$, the  {right} hand side of~\eqref{eq:thm-main} when viewed as curve-decorated quantum surfaces is equal to 
    \begin{equation}
        {c_N} \int_{0<y_1<...<y_{2N-3}<1}\LF_\bbH^{(\beta,0),(\beta,1),(\beta,\infty),(\beta,y_1),...,(\beta,y_{2N-3})}\times \mathrm{mSLE}_{\kappa,\alpha+m}(\bbH,0,y_1,...,y_{2N-3},1,\infty)dy_1...dy_{2N-3}
    \end{equation}
\end{lemma}
\begin{proof}
    We write $y_{2N-2}=1$, $y_{2N-1}=\infty$ and $y_{2N}=0$. Without loss of generality we may work on the case where $m=1$. Let $f_{y_{2N-3}}(z):=\frac{1-y_{2N-3}}{1-z}$ be the conformal map $\bbH\to\bbH$ sending $(y_{2N-3},1,\infty,0)$ to $(1,\infty,0,x_1)$, and for $2\le k\le 2N$, let $x_k =f_{y_{2N-3}}(y_{k-1})$. Then by Lemma~\ref{lm:lcft-H-conf} and~\eqref{eq:msle-conf}, the  {right}  hand side of~\eqref{eq:thm-main}, when viewed as the law of curve-decorated quantum surfaces, is the same as
\begin{equation}\label{eq:lm-rotation-1}
    \begin{split}
        {c_N} \int_{0<y_1<...<y_{2N-3}<1}&\bigg[\prod_{k=1}^{2N}f_{y_{2N-3}}'(y_k)^{b+\Delta_\beta} \LF_\bbH^{(\beta,f_{y_{2N-3}}(y_1)),...,(\beta,f_{y_{2N-3}}(y_{2N}))}\\&\times \mathrm{mSLE}_{\kappa,\alpha+m}(\bbH,f_{y_{2N-3}}(y_1),\,...,f_{y_{2N-3}}(y_{2N}))\bigg]\,dy_1...dy_{2N-3}
    \end{split}
\end{equation}
Since $\Delta_\beta+b=1$, by a change of variables $x_k =f_{y_{2N-3}}(y_{k-1})$ for $2\le k\le 2N-3$, ~\eqref{eq:lm-rotation-1} is equal to 
\begin{equation}\label{eq:lm-rotation-2}
    \begin{split}
         {c_N}\int_{0<f_{y_{2N-3}}(0)<x_2<...<x_{2N-3}<1}&\bigg[f_{y_{2N-3}}'(y_{2N-3})f_{y_{2N-3}}'(1)f_{y_{2N-3}}'(\infty)f_{y_{2N-3}}'(0)\cdot \LF_\bbH^{(\beta,f_{y_{2N-3}}(0)),(\beta,x_2)...,(\beta,x_{2N})}\\&\times \mathrm{mSLE}_{\kappa,\alpha+m}(\bbH,f_{y_{2N-3}}(0),x_2,\,...,x_{2N}))\bigg]\,dy_{2N-3}dx_2...dx_{2N-3}
    \end{split}
\end{equation}    
where we used the convention $(x_{2N-2},x_{2N-1},x_{2N})=(1,\infty,0)$. Note that
$$f_{y_{2N-3}}'(y_{2N-3})f_{y_{2N-3}}'(1)f_{y_{2N-3}}'(\infty)f_{y_{2N-3}}'(0)=1.$$
Therefore we conclude the proof by the change of variable $x_1=1-y_{2N-3}$ in~\eqref{eq:lm-rotation-2}.
\end{proof}

\begin{proof}[Proof of Theorem~\ref{thm:main}]
    First we notice that by Proposition~\ref{prop:msle-margin}, when $W_1=W_2=2$ the measure $\mathfrak{m}_x(W_1,W_2)$ agrees with $\mathrm{mSLE}_\kappa(\bbH;0,x,1,\infty)$ and for the $N=2$ case,  Theorem~\ref{thm:main} follows immediately from Theorem~\ref{thm:N=2} when $\alpha=\{\{1,4\},\{2,3\}\}$ and Lemma~\ref{lm:rotation} when $\alpha=\{\{1,2\},\{3,4\}\}$. The constant can be traced by the one in~\eqref{eq:prop-constant-3-4}.

    Now assume we have proved Theorem~\ref{thm:main} for $N$ and we work on the $N+1$ case. By Lemma~\ref{lm:rotation}, without loss of generality we may assume $\{1,2N+2\}\in\alpha$, and let $\alpha_0\in\mathrm{LP}_N$ be the link pattern induced by $\alpha\backslash\{1,2N+2\}$.  Let $(\bbH,\phi,s_1,...,s_{2N-3},\eta_1,...,\eta_N)$ be a sample from the right hand side of~\eqref{eq:thm-main}  {(with $x$ replaced by $s$)} where the multiple $\SLE$ is sampled from $\mathrm{mSLE}_{\kappa,\alpha_0}(\bbH,0,s_1,...,s_{2N-3},1,\infty)$. As in the proof of Proposition~\ref{prop:disk+4-pt-disk}, we finish the induction by the following 3 steps.

    \emph{Step 1: Sample two extra boundary typical points and change coordinates.}  {Let $\bbR_-=(-\infty,0)$.} We weight the law of $(\bbH,\phi,s_1,...,s_{2N-3},\eta_1,...,\eta_N)$  {by $\frac{1}{2}\nu_\phi(\bbR_-)^2$ and sample two points $\xi_1,\xi_2$ from the probability measure proportional to $ 1_{\xi_1<\xi_2<0}\nu_\phi|_{\bbR_-}(d\xi_1)\nu_\phi|_{\bbR_-}(d\xi_2)$}. Then by Definition~\ref{def:QD}, if we conformally weld a quantum disk from $\QD_{2}$ to the arc $(\xi_1,\xi_2)$, the curve-decorated quantum surface we obtain has law {$\Wd_\alpha(\QD^{N+2})$}. On the other hand, by Lemma~\ref{lm:gamma-insertion}, the law of $(\phi,\eta_1,...,\eta_N,s_1,...,s_{2N-3},\xi_1,\xi_2)$ is now 
    \begin{equation}\label{eq:thm-main-1}
    \begin{split}
         {c_N}\cdot\mathds{1}_{\xi_1<\xi_2<0<s_1<...<s_{2N-3}<1}&\bigg[\LF_\bbH^{(\beta,s_1),...,(\beta,s_{2N}),(\gamma,\xi_1),(\gamma,\xi_2)}\times \mathrm{mSLE}_{\kappa,\alpha_0}(\bbH,s_1,...,s_{2N})\bigg]\,ds_1...ds_{2N-3}d\xi_1d\xi_2
    \end{split}
\end{equation}
    where we used the convention $(s_{2N-2},s_{2N-1},s_{2N})=(1,\infty,0)$. For $\xi_1<\xi_2<0$, consider the conformal map $f_{\xi_1,\xi_2}(z)=\frac{z-\xi_2}{z-\xi_1}$ from $\bbH$ to $\bbH$ sending $(\xi_1,\xi_2,\infty)$ to $(\infty,0,1)$. Let $y_1=f_{\xi_1,\xi_2}(0)$, $y_{2N-1}=f_{\xi_1,\xi_2}(1)$, $y_{2N}=1$ and $y_k = f_{\xi_1,\xi_2}(y_{k-1})$ for $2\le k\le 2N-2$.  Then by Lemma~\ref{lm:lcft-H-conf} and~\eqref{eq:msle-conf}, when viewed as the law of curve-decorated quantum surfaces, ~\eqref{eq:thm-main-1} is equal to
    \begin{equation}\label{eq:thm-main-2}
    \begin{split}
         {c_N}\cdot\mathds{1}_{\xi_1<\xi_2<0<s_1<...<s_{2N-3}<1}&\bigg[f_{{\xi_1,\xi_2}}'(\xi_1)f_{{\xi_1,\xi_2}}'(\xi_2)\prod_{k=1}^{2N}f_{{\xi_1,\xi_2}}'(s_k) \cdot\LF_\bbH^{(\beta,f_{{\xi_1,\xi_2}}(s_1)),...,(\beta,f_{{\xi_1,\xi_2}}(s_{2N})),(\gamma,0),(\gamma,\infty)}\\&\times \mathrm{mSLE}_{\kappa,\alpha_0}(\bbH,f_{{\xi_1,\xi_2}}(s_1),\,...,f_{{\xi_1,\xi_2}}(s_{2N}))\bigg]\,ds_1...ds_{2N-3}d\xi_1d\xi_2
    \end{split}
\end{equation}
where we used $\Delta_\beta+b=\Delta_\gamma=1$. Then by a change of variables $y'_k = f_{\xi_1,\xi_2}(s_{k-1})$ for $2\le k\le 2N-2$, ~\eqref{eq:thm-main-2} is equal to
\begin{equation}\label{eq:thm-main-3}
\begin{split}
     &{c_N}\cdot\mathds{1}_{0<f_{{\xi_1,\xi_2}}(0)<y'_2<...<y'_{2N-2}<f_{{\xi_1,\xi_2}}(1) <1}\bigg[f_{{\xi_1,\xi_2}}'(\xi_1)f_{{\xi_1,\xi_2}}'(\xi_2)f_{{\xi_1,\xi_2}}'(0) f_{{\xi_1,\xi_2}}'(1)f_{{\xi_1,\xi_2}}'(\infty)\cdot\\&\LF_\bbH^{(\beta,f_{{\xi_1,\xi_2}}(0)),(\beta,y'_2),...,(\beta,y'_{2N-2}), (\beta,f_{{\xi_1,\xi_2}}(1)),(\beta,1),(\gamma,0),(\gamma,\infty)}\times \mathrm{mSLE}_{\kappa,\alpha_0}(\bbH,y'_2,...,y'_{2N-2})\bigg]dy'_2...dy'_{2N-2}d\xi_1d\xi_2.
     \end{split}
\end{equation}
Since $y'_1=f_{{\xi_1,\xi_2}}(0)=\frac{\xi_2}{\xi_1}$, $y'_{2N-1}=f_{{\xi_1,\xi_2}}(1)=\frac{1-\xi_2}{1-\xi_1}$, it is straightforward to check that
$$\frac{\partial(y'_1,y'_{2N-1})}{\partial(\xi_1,\xi_2) } = \frac{\xi_2-\xi_1}{\xi_1^2(1-\xi_1)^2}. $$
On the other hand, we may compute
$$f_{{\xi_1,\xi_2}}'(\xi_1)f_{{\xi_1,\xi_2}}'(\xi_2)f_{{\xi_1,\xi_2}}'(0) f_{{\xi_1,\xi_2}}'(1)f_{{\xi_1,\xi_2}}'(\infty) =  \frac{\xi_2-\xi_1}{\xi_1^2(1-\xi_1)^2}.$$
Therefore by a change of variables to~\eqref{eq:thm-main-3}, the law of $(\bbH,\phi,\eta_1,...,\eta_N,0,s_1,...,s_{2N-3},\xi_1,\xi_2,1,\infty)/\sim_\gamma$ agrees with that of $(\bbH,\tilde\phi,\eta_1,...,\eta_N,y'_1,...,y'_{2N-1},0,1,\infty)/\sim_\gamma$, where $(\tilde\phi,\tilde\eta_1,...,\tilde\eta_N,y'_1,...,y'_{2N-1})$ is sampled from
\begin{equation}\label{eq:thm-main-4}
    \begin{split}
         {c_N}\cdot\mathds{1}_{0<y'_1<...<y'_{2N-1}<1}&\bigg[\LF_\bbH^{(\beta,y'_1),...,(\beta,y'_{2N-1}),(\beta,1),(\gamma,0),(\gamma,\infty)}\times \mathrm{mSLE}_{\kappa,\alpha_0}(\bbH,y'_1,...,y'_{2N-1},1)\bigg]\,dy'_1...dy'_{2N-1}.
    \end{split}
\end{equation}

\emph{Step 2: Adding boundary typical points to the welding of $\Md_2(2)$ and $\Md_{2,\bullet}(2;\beta)$.} Consider the welding in Proposition~\ref{prop:3-pt-disk} with $W_1=W_2=2$, and the surface $S$ on the left hand side of~\eqref{eqn-3-pt-disk} is embedded as $(\bbH,X,\eta,0,\infty,1)$ where $X\sim {\frac{\gamma}{2c(Q-\beta)^2}}\LF_\bbH^{(\beta,0),(\beta,\infty),(\beta,1)}$, and $\eta$ is the $\SLE_\kappa$ process in $\bbH$ weighted by $\psi_\eta'(1)^{1-\Delta_\beta}$. Let 
\begin{equation}\label{eq:lm-qt-4-pted-40}
   Y=X\circ\psi_\eta^{-1}+Q\log|(\psi_\eta^{-1})'|, 
\end{equation}
$D_\eta^1$ be the  component of $\bbH\backslash\eta$ to the left of $\eta$. Let $S_1=(D_\eta^1,X)/\sim_\gamma$, and $S_2 = (\bbH,Y,0,\infty,1)/\sim_\gamma$. Recall that $$\mathrm{mSLE}_{\kappa,\alpha_0}(\bbH,y_1,...,y_{2N-1},1) = \mathcal{Z}_{\alpha_0}(\bbH,y_1,...,y_{2N-1},1)\cdot\mathrm{mSLE}_{\kappa,\alpha_0}(\bbH,y_1,...,y_{2N-1},1)^\#.$$
By Lemma~\ref{lm:gamma-insertion}, as we sample $2N-1$ marked points on $S_2$ from the measure
$$\mathds{1}_{0<y_1<...<y_{2N-1}<1}\cdot\mathcal{Z}_{\alpha_0}(\bbH,y_1,...,y_{2N-1},1)\nu_Y(dy_1)...\nu_Y(dy_{2N-1}), $$
the surface $S$ is now the conformal welding of a weight $2$ quantum disk with a $2N$-pointed quantum disk with embedding
\begin{equation}\label{eq:thm-main-5}
    \begin{split}
         {\frac{\gamma}{2(Q-\gamma)^2}}\int_{0<y_1<...<y_{2N-1}<1}&\bigg[\LF_\bbH^{(\gamma,y_1),...,(\gamma,y_{2N-1}),(\beta,1),(\gamma,0),(\gamma,\infty)}\times \mathcal{Z}_{\alpha_0}(\bbH,y_1,...,y_{2N-1},1)\bigg]\,dy_1...dy_{2N-1}.
    \end{split}
\end{equation}
On the other hand, for $x_k = \psi_\eta^{-1}(y_k)$ and $k=1,...,2N-1$, the law of $(X,\eta,x_1,...,x_{2N-1})$ is given by
\begin{equation}\label{eq:thm-main-6}
    \begin{split}
         {\frac{\gamma}{2c(Q-\beta)^2}}\mathds{1}_{0<x_1<...<x_{2N-1}<1}&\bigg[\mathcal{Z}_{\alpha_0}\big(\bbH,\psi_\eta(x_1),...,\psi_\eta(x_{2N-1}),1\big)\nu_X(dx_1)...\nu_X(dx_{2N-1}) \bigg]\cdot\\&\LF_\bbH^{(\beta,1),(\beta,0),(\beta,\infty)}(dX)\cdot\psi_\eta'(1)^{1-\Delta_\beta}\,\SLE_\kappa(d\eta)  
    \end{split}
\end{equation}
which, by Lemma~\ref{lm:gamma-insertion}, is equal to 
\begin{equation}\label{eq:thm-main-7}
    \begin{split}
         {\frac{\gamma}{2c(Q-\beta)^2}}\mathds{1}_{0<x_1<...<x_{2N-1}<1}&\bigg[\mathcal{Z}_{\alpha_0}\big(\bbH,\psi_\eta(x_1),...,\psi_\eta(x_{2N-1}),1\big) \times\\&\LF_\bbH^{(\gamma,x_1),...,(\gamma,x_{2N-1}),(\beta,1),(\beta,0),(\beta,\infty)}(dX)dx_1...dx_{2N-1}\bigg]\cdot\psi_\eta'(1)^{1-\Delta_\beta}\,\SLE_\kappa(d\eta).  
    \end{split}
\end{equation}

\emph{Step 3: Change the insertion from $\gamma$ to $\beta$.} We weight the law of $(x_1,...,x_{2N-1},X,\eta)$ from~\eqref{eq:thm-main-7} by $$ {\frac{2c_N(Q-\gamma)^2}{\gamma}  }\prod_{k=1}^{2N-1}\big(\e^{\frac{\beta^2-\gamma^2}{4}} e^{\frac{\beta-\gamma}{2}Y_\e(y_k)}\big),$$ where $X$ is given by~\eqref{eq:lm-qt-4-pted-40} and $y_k=\psi_\eta(x_k)$. Then as in Step 3 of the proof of Proposition~\ref{prop:disk+4-pt-disk}, as we send $\e\to0$, the law of $(x_1,...,x_{2N-1},X,\eta)$ converges in vague topology to 
\begin{equation}\label{eq:thm-main-8}
    \begin{split}
         {\frac{c_N(Q-\gamma)^2}{(Q-\beta)^2}\cdot c^{-1}  }\mathds{1}_{0<x_1<...<x_{2N-1}<1}&\bigg[\mathcal{Z}_{\alpha_0}\big(\bbH,\psi_\eta(x_1),...,\psi_\eta(x_{2N-1}),1\big)\cdot \prod_{k=1}^{2N-1}\psi_\eta'(x_k)^{1-\Delta_\beta}\cdot \psi_\eta'(1)^{1-\Delta_\beta}\cdot\\&\LF_\bbH^{(\beta,x_1),...,(\beta,x_{2N-1}),(\beta,1),(\beta,0),(\beta,\infty)}(dX)\,dx_1...dx_{2N-1}\bigg]\SLE_\kappa(d\eta).  
    \end{split}
\end{equation}
Since $\Delta_\beta+b=1$, by~\eqref{eq:msle-conf} we have 
\begin{equation*}
    \mathcal{Z}_{\alpha_0}\big(\bbH,\psi_\eta(x_1),...,\psi_\eta(x_{2N-1}),1\big)\cdot \prod_{k=1}^{2N-1}\psi_\eta'(x_k)^{1-\Delta_\beta}\cdot \psi_\eta'(1)^{1-\Delta_\beta} = \mathcal{Z}_{\alpha_0}\big(D_\eta,x_1,...,x_{2N-1},1\big).
\end{equation*}
Therefore it follows from Proposition~\ref{prop:msle-margin} that if we draw the interfaces $(\eta_1^R,...,\eta_N^R)$ in $D_\eta$ sampled from $\mathrm{mSLE}_{\kappa,\alpha_0}(D_\eta,x_1,...,x_{2N-1},1)^\#$, then the law of $(x_1,...,x_{2N-1},X,\eta,\eta_1^R,...,\eta_N^R)$ is precisely given by the left hand side of~\eqref{eq:thm-main}. On the other hand,  {when disintegrated over} the interface length $\ell$ of $\eta$,  {by Lemma~\ref{lm:lf-change-weight}, the weighted law~\eqref{eq:thm-main-5} for $(Y,\psi_\eta\circ\eta_1^R,..., \psi_\eta\circ\eta_N^R,y_1,...,y_{2N-1})$  converges in vague topology  to ~\eqref{eq:thm-main-4} (disintegrated over the length $\ell$ of $\nu_{\tilde\phi}(\bbR_-)$). Then joint law of $(S_1,S_2)$ converges in vague topology to $\Md_2(2;\ell)\times \mu_{2,\ell}$, where $\mu_{2,\ell}$ is the law of the quantum surface with embedding~\eqref{eq:thm-main-4} with $\nu_{\tilde\phi}(\bbR_-)=\ell$.}  {In other words,  {once integrating over $\ell$}, we obtain the conformal welding of $\Wd_{\alpha_0}(\QD^{N+1})$ which can be embedded as~\eqref{eq:thm-main-4} with a sample from $\Md_2(2)$ according to the link pattern $\alpha$. } Therefore by Step 1 and our induction hypothesis, the conformal welding of $S_1$ and $S_2$ under the reweighted measure is a constant times $\Wd_\alpha(\QD^{N+2})$, which finishes the proof of Theorem~\ref{thm:main}. 
\end{proof}

\section{Conformal welding and Imaginary Geometry}\label{sec:IG}
In this section,  {we first provide background on  imaginary geometry in Section~\ref{subsec:pre-ig} and quantum triangle in Section~\ref{subsec:QT}. Then we prove Theorem~\ref{thm:IG} which gives the partition function of imaginary geometry flow lines from conformal welding of quantum triangles. Similar to the proof of Theorem~\ref{thm:main}, we use an inductive argument by adding marked points. We will need a variant of the $\SLE_\kappa(\underline{\rho})$ martingale by~\cite{SW05} which we recall in Section~\ref{subsec:martingale}. The proof of Theorem~\ref{thm:IG} is done in Section~\ref{subsec:pf-thm-IG}.
}

\subsection{$\SLE_\kappa(\underline{\rho})$ process and the imaginary geometry}\label{subsec:pre-ig}
 For the force points $x^{k,L}<...<x^{1,L}<x^{0,L}= 0^-<x^{0,R}= 0^+< x^{1,R}<...<x^{\ell, R}$ and the weights $\rho^{i,q}\in\bbR$, the $\SLE_\kappa(\underline{\rho})$ process is the probability measure on  {curves  $\eta$  in $\overline{\bbH}$} growing the same as ordinary SLE$_\kappa$ (i.e, satisfies \eqref{eq:def-sle}) except that the Loewner driving function $(W_t)_{t\ge 0}$ is now characterized by 
\begin{equation}\label{eq:def-sle-rho}
\begin{split}
&W_t = \sqrt{\kappa}B_t+\sum_{q\in\{L,R\}}\sum_i \int_0^t \frac{\rho^{i,q}}{W_s-V_s^{i,q}}ds; \\
& V_t^{i,q} = x^{i,q}+\int_0^t \frac{2}{V_s^{i,q}-W_s}ds, \ q\in\{L,R\}.
\end{split}
\end{equation}
It has been proved in \cite{MS16a} that the SLE$_\kappa(\underline{\rho})$ process a.s. exists, is unique and generates a continuous curve until the \textit{continuation threshold}, the first time $t$ such that $W_t = V_t^{j,q}$ with $\sum_{i=0}^j\rho^{i,q}\le -2$ for some $j$ and $q\in\{L,R\}$. {Moreover, the curve is disjoint from $(0,\infty)$ (resp.\ $(-\infty,0)$) if  $\sum_{i=0}^j\rho^{i,q}\ge\frac{\kappa}{2} -2$ for $q=R$ and every $0\le j\le \ell$ (resp.\ for $q=L$ and every $0\le j\le k$). } 

Now we recap the definition of the Dirichlet Gaussian Free Field. Let $D\subsetneq \mathbb{C}$ be a domain. We construct the GFF on $D$ with \textit{Dirichlet} \textit{boundary conditions}  as follows. Consider the space of smooth functions on $D$ with finite Dirichlet energy and zero value near $\partial D$, and let $H_0(D)$ be its closure with respect to the inner product $(f,g)_\nabla=\int_D(\nabla f\cdot\nabla g)\ dx dy$. Then the (zero boundary) GFF on $D$ is defined by 
\begin{equation}\label{eqn-def-gff}
h = \sum_{n=1}^\infty \xi_nf_n
\end{equation}
where $(\xi_n)_{n\ge 1}$ is a collection of i.i.d. standard Gaussians and $(f_n)_{n\ge 1}$ is an orthonormal basis of $H_0(D)$. The sum \eqref{eqn-def-gff} a.s.\ converges to a random distribution independent of the choice of the basis $(f_n)_{n\ge 1}$. For a function $g$ defined on $\partial D$ with harmonic extension $f$ in $D$ and a zero boundary GFF $h$, we say that $h+f$ is a GFF on $D$ with boundary condition specified by $g$. 
 See \cite[Section 4.1.4]{DMS14} for more details. 

Next we introduce the notion of \emph{GFF flow lines}. We restrict ourselves to the range $\kappa\in(0,4)$.  Heuristically, given a GFF $h$, $\eta(t)$ is a flow line of angle $\theta$ if
\begin{equation}
\eta'(t) = e^{i(\frac{h(\eta(t))}{\chi}+\theta)}\ \text{for}\ t>0, \ \text{where}\ \chi = \frac{2}{\sqrt{\kappa}}-\frac{\sqrt{\kappa}}{2}.
\end{equation} 
 {To make rigorous sense of this}, let $(K_t)_{t\ge 0}$ be the hull at time $t$ of the SLE$_\kappa(\underline{\rho})$ process described by the Loewner flow \eqref{eq:def-sle} with $(W_t, V_t^{i,q})$ solving \eqref{eq:def-sle-rho}, and let $\mathcal{F}_t$ be the filtration generated by  $(W_t, V_t^{i,q})$. 
Let $\mathfrak{h}_t^0$ be the bounded harmonic function on $\mathbb{H}$ with boundary values 
\begin{equation}\label{eq:flowlinecouple}
    	-\lambda(1+\sum_{i=0}^j \rho^{i,L})\  \text{on} \ (V_t^{j+1, L}, V_t^{j,L}),\ \ \ \text{and}\ \ \lambda(1+\sum_{i=0}^j \rho^{i,R})\ \text{on}\ \ (V_t^{j, R}, V_t^{j+1,R})
\end{equation}
	and $-\lambda$ on $(V_t^{0,L},W_t)$,  $\lambda$ on $(W_t,V_t^{0,R})$ where $\lambda = \frac{\pi}{\sqrt{\kappa}}$, $x^{k+1, L} = -\infty, x^{\ell+1, R} = +\infty$. Set $\mathfrak{h}_t(z) = \mathfrak{h}_t^0(g_t(z))-\chi\arg g_t'(z)$. Let  $\tilde{h}$ be a zero boundary GFF on $\mathbb{H}$ and 
	\begin{equation}\label{eq:ig-gff}
	  h = \tilde{h}+ \mathfrak{h}_0.  
	\end{equation}
	Then as proved in \cite[Theorem 1.1]{MS16a}, there exists a coupling between $h$ and the  SLE$_\kappa(\underline{\rho})$ process $(K_t)$, such that for any $\mathcal{F}_t$-stopping time $\tau$ before the continuation threshold, $K_\tau$ is a local set  {(introduced in~\cite{SS13})} for $h$  and the conditional law of $h|_{\mathbb{H}\backslash K_\tau}$ given $\mathcal{F}_\tau$ is the same as the law of  {$\mathfrak{h}_\tau+\tilde{h}\circ g_\tau$}.
	
	For $\kappa<4$, the SLE$_\kappa(\underline{\rho})$ coupled with the GFF $h$ as above is referred as a \emph{flow line} of $h$ from 0 to $\infty$, and we say an SLE$_\kappa(\underline{\rho})$ curve is a flow line of angle $\theta$ if it can be coupled with $h+\theta\chi$ in the above sense. See Figure~\ref{fig:igflowline} for an example.

 \begin{figure}
     \centering
     \includegraphics[width=1\textwidth]{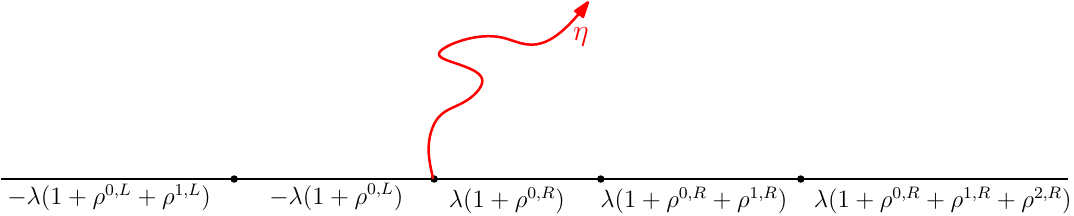}
     \caption{An $\SLE_\kappa(\rho^{0,L},\rho^{1,L};\rho^{0,R},\rho^{1,R},\rho^{2,R})$ process coupled with the GFF $h$ with illustrated boundary conditions as the (zero angle) flow line of $h$. The $\theta$ angle flow line of $h$ then has the law as $\SLE_\kappa(\rho^{0,L}-\frac{\theta\chi}{\lambda},\rho^{1,L};\rho^{0,R}+\frac{\theta\chi}{\lambda},\rho^{1,R},\rho^{2,R})$ process.}
     \label{fig:igflowline}
 \end{figure}

 We shall use the following result from~\cite[Theorem 1.1, Theorem 1.2]{MS16a},  {where the first theorem proved the existence of the flow line coupling, and the second showed that flow lines are a.s.\ determined by the free field}.
 \begin{theorem}\label{thm:IGflowline}
     Let $k,\ell\ge0$, $\rho^{0,L},...,\rho^{k,L},\rho^{0,R},...,\rho^{\ell,R},\theta\in\bbR$, $x^{k,L}<...<x^{0,L}=0^-$ {,} $0^+=x^{0,R}<...<x^{\ell,R}$ and $h$ be a Dirichlet GFF on $\bbH$ with boundary conditions specified by~\eqref{eq:ig-gff}. Suppose $h+\chi\theta>\lambda(\frac{\kappa}{2}-1)$ for $x>0$ and $h+\chi\theta<\lambda$ for $x<0$. Let $\eta$ be a flow line of $h$ of angle $\theta$ from 0 to $\infty$. Then $\eta\cap(0,\infty)=\emptyset$ a.s.. Moreover, given $\eta$, if we let $D_\eta$ be  {the  connected component of $\bbH\backslash\eta$ to the right of $\eta$},  $\psi_\eta:D_\eta\to\bbH$ be the conformal map fixing 0, 1, $\infty$, then $h_\eta:=h\circ \psi_\eta^{-1}-\chi\arg (\psi_\eta^{-1})'$ is the Dirichlet GFF on $\bbH$ with boundary conditions $\lambda-\theta\chi$ on $(-\infty,0)$ and $\psi_\eta\circ h$ on $(0,\infty)$. If  {$\theta_1<\theta$ and} $\eta_1$ is  {the $\theta_1$ angle flow line of $h$} from $x_1\ge0$ to $\infty$, then $\psi_\eta\circ\eta_1$ is the angle $\theta_1$ flow line of $h_\eta$  from $\psi_\eta(x_1)$ to $\infty$.
 \end{theorem}

 \subsection{Quantum triangles}\label{subsec:QT}
Next we recap the notion of quantum triangle as in \cite{ASY22}. It is a quantum surface parameterized by weights $W_1,W_2,W_3>0$ and defined based on Liouville fields with three insertions and the thick-thin duality.

\begin{definition}[Thick quantum triangles]\label{def-qt-thick}
	Fix {$W_1, W_2, W_3>\frac{\gamma^2}{2}$}. Set $\beta_i = \gamma+\frac{2-W_i}{\gamma}<Q$ for $i=1,2,3$, and {let $\phi$ be sampled from $\frac{1}{(Q-\beta_1)(Q-\beta_2)(Q-\beta_3)}\textup{LF}_{\mathbb{H}}^{(\beta_1, \infty), (\beta_2, 0), (\beta_3, 1)}$}. Then we define the infinite measure $\textup{QT}(W_1, W_2, W_3)$ to be the law of $(\bbH, \phi, \infty, 0,1)/\sim_\gamma$.
\end{definition}

\begin{definition}[Quantum triangles with thin vertices]\label{def-qt-thin}
	Fix {$W_1, W_2, W_3\in (0,\frac{\gamma^2}{2})\cup(\frac{\gamma^2}{2}, \infty)$}. Let $I:=\{i\in\{1,2,3\}:W_i<\frac{\gamma^2}{2}\}$ and define the $\textup{QT}(W_1, W_2, W_3)$ to be the law of the quantum surface $S$ constructed as follows. First sample the surface $S_0 := (D, \phi, a_1, a_2, a_3)$ from $\textup{QT}(\tilde{W}_1, \tilde{W}_2, \tilde{W}_3)$ where $\tilde{W}_i = W_i$ if $i\notin I$, and $\tilde{W}_i = \gamma^2-W_i$ if $i\in I$ so that each point $a_i$ corresponds to the insertion as given by $\tilde{\beta}_i = \gamma+ \frac{2-\tilde{W}_i}{\gamma}$. For each $i\in I$, {independently sample a surface $S_i$ from $(1-\frac{2W_i}{\gamma^2})\mathcal{M}_2^{\textup{disk}}(W_i)$} and concatenate it with $S_0$ at point $a_i$, and let $S$ be the output.
\end{definition} 

\begin{figure}[ht]
	\centering
	\includegraphics[scale=0.43]{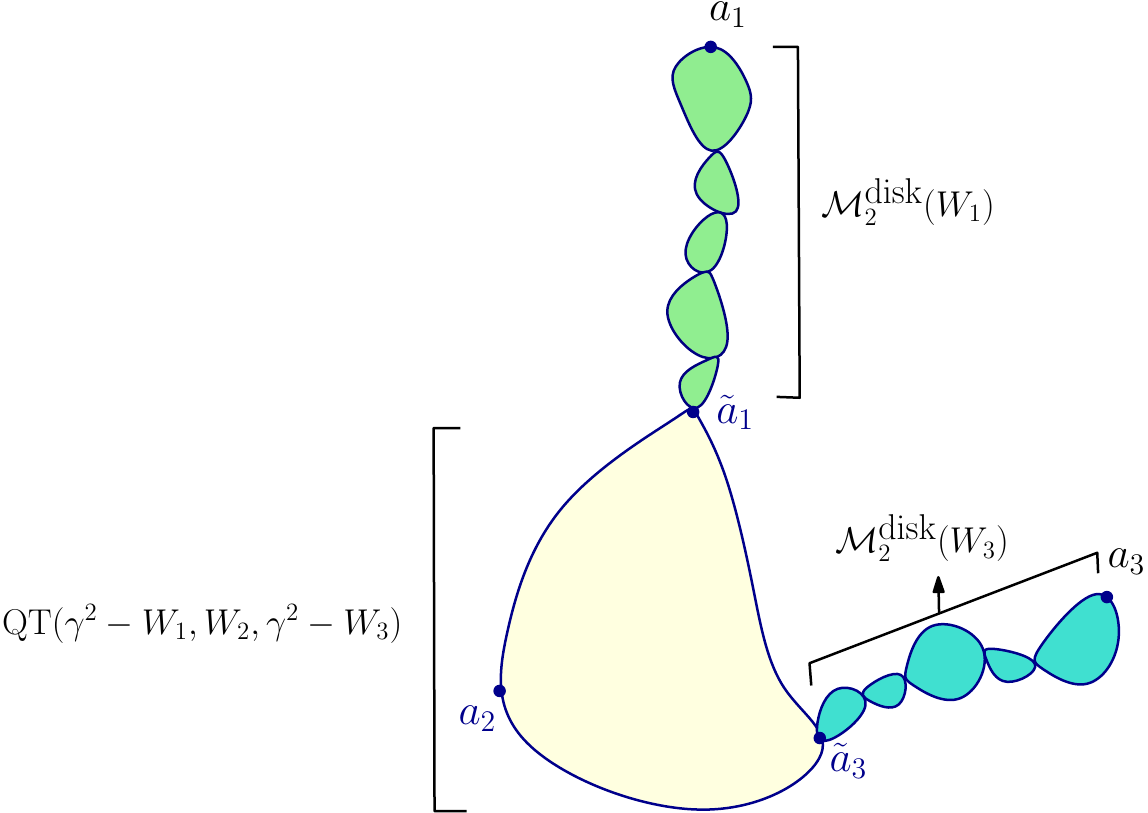}
	\caption{A quantum triangle with $W_2>\frac{\gamma^2}{2}$ and $W_1,W_3<\frac{\gamma^2}{2}$ embedded as $(D,\phi,a_1,a_2,a_3).$  The two thin disks (colored green) are concatenated with the thick triangle (colored yellow) at points $\tilde{a}_1$ and $\tilde{a}_3$.}\label{fig-qt}
\end{figure}

See Figure~\ref{fig-qt} for an example when $W_1,W_3<\frac{\gamma^2}{2}$ and  $W_2>\frac{\gamma^2}{2}.$ We remark that if any of $W_1,W_2,W_3$ is equal to $\frac{\gamma^2}{2}$, then the measure $\QT(W_1,W_2,W_3)$ is defined by a limiting procedure,  {which we omit here as this is of less interest in this paper. A detailed definition is in \cite[Section 2.5]{ASY22}}. 

Recall the notion of $\Md_{2,\bullet}(W)$ from Definition~\ref{def:three-pointed-disk}.
\begin{lemma}[Lemma 6.12 of~\cite{ASY22}]\label{lem:QT(W,W,2)}
   Let $W>0$. For some constant $C$ depending only on $W$, we have $\Md_{2,\bullet}(W) = C\QT(W,W,2)$.
\end{lemma}

The  {disintegration} of the measure  $\Md_{2}(W)$  over the boundary arc lengths  can be similarly extended to quantum triangles by setting
\begin{equation}
    \QT(W_1,W_2,W_3) = \iiint_{\bbR_+^3}\QT(W_1,W_2,W_3;\ell_1,\ell_2,\ell_3)\ d\ell_1d\ell_2d\ell_3
\end{equation}
where $\QT(W_1,W_2,W_3;\ell_1,\ell_2,\ell_3)$ is the measure supported on quantum surfaces $(D,\phi,a_1,a_2,a_3)/\sim_\gamma$ such that the boundary arcs between $a_1a_2$, $a_1a_3$ and $a_2a_3$ have quantum lengths $\ell_1,\ell_2,\ell_3$. We can also disintegrate over one or two boundary arc lengths of quantum triangles. For instance, we can define 
$$\QT(W_1,W_2,W_3;\ell_1,\ell_2) = \int_0^\infty \QT(W_1,W_2,W_3;\ell_1,\ell_2,\ell_3)\ d\ell_3$$
and $$\QT(W_1,W_2,W_3;\ell_1) = \iint_{\bbR^2_+}\QT(W_1,W_2,W_3;\ell_1,\ell_2,\ell_3)\ d\ell_2 d\ell_3.$$

Given a quantum triangle of weights $W+W_1,W+W_2,W_3$ with $W_2+W_3 = W_1+2$ embedded as $(D,\phi,a_1,a_2,a_3)$, we start by making sense of the $\SLE_\kappa(W-2;W_1-2,W_2-W_1)$ curve $\eta$ from $a_2$ to $a_1$. If the domain $D$ is simply connected (which corresponds to the case where $W+W_1,W+W_2, W_3\ge\frac{\gamma^2}{2}$), $\eta$ is just the ordinary $\SLE_\kappa(W-2;W_1-2,W_2-W_1)$  with force points at $a_2^-,a_2^+$ and $a_3$. Otherwise, let $(\tilde{D}, \phi,\tilde{a}_1,\tilde{a}_2, \tilde{a}_3)$ be the thick quantum triangle component, and sample an $\SLE_\kappa(W-2;W_1-2,W_2-W_1)$ curve $\tilde{\eta}$ in $\tilde{D}$ from $\tilde{a}_2$ to $\tilde{a}_1$. Then our curve $\eta$ is the concatenation of $\tilde{\eta}$ with $\SLE_\kappa(W-2;W_1-2)$ curves in each bead of the weight $W+W_1$ (thin) quantum disk  and $\SLE_\kappa(W-2;W_2-2)$ curves in each bead of the weight $W+W_2$ (thin) quantum disk.

With this notation, we state the welding of quantum disks with quantum triangles below.
\begin{theorem}[Theorem 1.1 of \cite{ASY22}]\label{thm:disk+QT}
Fix $W,W_1,W_2,W_3>0$ such that $W_2+W_3=W_1+2$. There exists some constant $c:=c_{W,W_1,W_2,W_3}\in (0,
\infty)$ such that
\begin{equation}\label{eqn:disk+QT}
    \QT(W+W_1,W+W_2,W_3) {\times} \SLE_\kappa(W-2;W_2-2,W_1-W_2) = c\Wd (\Md_{2}(W),\QT(W_1,W_2,W_3)).
\end{equation}
That is, if we draw an independent $\SLE_\kappa(W-2;W_1-2,W_2-W_1)$ curve $\eta$ on a sample from $\QT(W+W_1,W+W_2,W_3)$ embedded as $(D,\phi,a_1,a_2,a_3)$, then the quantum surfaces to the left and right of $\eta$ are  independent quantum disks and triangles conditioned on having the same interface quantum length.
\end{theorem}

\subsection{The Schramm-Wilson $\SLE_\kappa(\underline{\rho})$ martingale}\label{subsec:martingale}
We start with the following analog of~\cite[Theorem 6 and Remark 7]{SW05}. Recall that by~\cite{MS16a}, an $\SLE_\kappa(\underline{\rho})$ process in $\bbH$ generates a continuous curve from 0 to $\infty$ as long as 
\begin{equation}\label{eq:continuation-threshold}
\sum_{i=0}^j\rho^{i,L}>-2\ \text{for all}\ 0\le j\le k\ \ \text{and} \ \ \sum_{i=0}^j\rho^{i,R}>-2  \ \text{for all}\ 0\le j\le \ell.  
\end{equation}
\begin{proposition}\label{prop:martingale}
    Fix $k,\ell\ge0$ and force points $x^{k,L}<...<x^{1,L}<x^{0,L}=0^-<x^{0,R}=0^+<x^{1,R}<...<x^{\ell,R}$. For $i\ge 0$ and $q\in\{L,R\}$, let $\rho^{i,q},\tilde{\rho}^{i,q}\in\bbR$ such that $\rho^{0,q}=\tilde\rho^{0,q}$ and~\eqref{eq:continuation-threshold} hold for both $\rho^{i,q}$ and $\tilde{\rho}^{i,q}$. Let $\eta$ be an $\SLE_\kappa(\underline{\rho})$ process and $(f_t)_{t\ge0}$ be its associated centered Loewner flow. Let
\begin{equation}\label{eq:swmartingale}
    M_t:=\prod_{i,q}|f_t'(x^{i,q})|^{\frac{(\tilde{\rho}^{i,q}-\rho^{i,q})(\tilde{\rho}^{i,q}+\rho^{i,q}+4-\kappa)}{4\kappa}}|f_t(x^{i,q})|^{\frac{\tilde{\rho}^{i,q}-\rho^{i,q}}{\kappa}}\cdot\prod |f_t(x^{i,q})-f_t(x^{i',q'})|^{\frac{\tilde{\rho}^{i,q}\tilde{\rho}^{i',q'} - {\rho}^{i,q}{\rho}^{i',q'} }{2\kappa}}
\end{equation}
where the second product is taken over all distinct pairs of $(i,q)$ and $(i',q')$.  {Let $\tau_0$ be the first time when any of $x^{i,q}$ with $\rho^{i,q}\neq\tilde{\rho}^{i,q}$ is separated from $\infty$.} Then $M_t$ is a local martingale, and  {for $t<\tau_0$}, when weighted by $M_t$, the law of $\eta$ is an $\SLE_\kappa(\underline{\tilde{\rho}})$ process.
\end{proposition}
 {Indeed $M_t$ is well-defined for $t<\tau_0$, and $\tau_0>0$ a.s.\ since by our assumption $\rho^{0,q}-\tilde\rho^{0,q}=0$.}

\begin{proof}
    We directly apply the It\^{o}'s formula from~\eqref{eq:def-sle-rho} (note that $f_t(x^{i,q}) = g_t(x^{i,q})-W_t = V_t^{i,q}$) and get
    \begin{equation}
        dM_t = \frac{1}{\sqrt{\kappa}}\cdot M_t\left( \sum_{i,q}\frac{\rho^{i,q}-\tilde{\rho}^{i,q}}{f_t(x^{i,q})}dB_t  \right)
    \end{equation}
    and thus $M_t$ is a local martingale. Therefore by Girsanov's theorem, once we weight the law of $W_t$ by $M_t$, the driving process $W_t$ has the law as the solution to 
    $$dW_t = \sqrt{\kappa}dB_t+\sum_{q\in\{L,R\}}\sum_i \frac{\tilde{\rho}^{i,q}}{W_t-V_t^{i,q}}dt $$
    which justifies the claim.
\end{proof}

We will analyze the $t\to\infty$ limit in~\eqref{eq:swmartingale}. We begin with the following lemma. 
\begin{lemma}\label{lm:martingale}
    Let  {$0<x_1<x_2<\infty$}. Let $\eta$ be a continuous {non-crossing} curve in $\bar{\bbH}$ from 0 to $\infty$ parameterized by the half-plane capacity such that $\eta(t)\to\infty$ as $t\to\infty$. Assume {$\eta\cap (x_1,x_2)=\emptyset$}. 
    Let $(f_t)_{t\ge0}$ be the centered Loewner map generated by $\eta$. {Let $D$ be the connected component of $\bbH\backslash\eta$ with $x_1,x_2$ on its boundary, and $\sigma,\xi$ be the leftmost and rightmost point on $\partial D\cap\bbR$. Let $\tau$ be the time when $\eta$ hits $\xi$ (where $\tau=\infty$ if $\xi=\infty$). Let $\psi:D\to\bbH$ be the conformal map sending $(\sigma,x_2,\xi)$ to $(0,1,\infty)$. Then we have\footnote{If $\sigma\neq 0^+$, then following the convention of defining $\SLE_\kappa(\underline\rho)$ processes, $f_t(0^+) = f_t(\sigma)$ after $\eta$ hits $\sigma$.}
    $$\lim_{t\uparrow\tau}\frac{f_t(x_1)}{f_t(x_2)} = 1, \ \lim_{t\uparrow\tau}\frac{f_t'(x_1)}{f_t'(x_2)} = \frac{\psi'(x_1)}{\psi'(x_2)}, \ \text{and}\ \lim_{t\uparrow\tau}\frac{f_t(x_1)-f_t(0^+)}{f_t(x_2)-f_t(0^+)} = \frac{\psi(x_1)}{\psi(x_2)}.  $$
    Moreover, on the event where $\dist(\eta, [x_1,x_2])>\delta$, there exists some  {positive constants $C_1,C_2$} depending only on $\delta, x_1$ and $x_2$ such that for every $t<\tau$}, 
    \begin{equation}\label{eq:loewner-bound}
        C_1<\frac{f_t(x_1)-f_t(0^+)}{f_t(x_2)-f_t(0^+)}<C_2 .
    \end{equation}
\end{lemma}
\begin{proof}
 {First assume that $\tau = \infty$. The first statement is proved in Eq.\ (A.2) in~\cite{liu2020scaling}, while~\eqref{eq:loewner-bound} follows from the proof of Eq.\ (A.1) in the same paper. See also the proof of ~\cite[Proposition 2.11]{Zhan19a}.}
    Let $\psi_t(z) = \frac{f_t(z) - f_t(0^+)}{f_t(x_1)-f_t(0^+)}$  {and $D_t$ be the unbounded connected component of $\bbH\backslash\eta([0,t])$. Then $\psi_t:D_t\to\bbH$} is a conformal map  {sending $(0^+,x_1,\infty)$ to $(0,1,\infty)$}. By the Schwartz reflection principle, we can extend $\psi_t$ to  {the closure $\bar{D}_t$ of $D_t\cup\{z:\bar{z}\in D_t\})$}, and by the  Caratheodory kernel theorem, as $t\to\infty$, $\psi_t$ converges to $\psi$ uniformly on compact subsets of $\bar{D}:=D\cup\{z:\bar{z}\in D\}$. This in particular implies that $\psi_t(x_i)\to\psi(x_i)$ and $\psi_t'(x_i)\to \psi'(x_i)$ for $i=1,2$ and thus the second and the third claim follow.

    {Now for $\tau<\infty$, let  {$\varphi(z) = \frac{z}{\xi-z}$}. Then $\tilde{\eta}:=(\varphi\circ\eta)_{0\le t\le \tau}$ is a continuous simple curve from 0 to $\infty$. Then $\tilde{f}_t:=\frac{f_t\circ\varphi^{-1}}{f_t(\xi)-f_t\circ\varphi^{-1}}$ is a conformal map from $\bbH\backslash\tilde{\eta}([0,t])$ sending the tip of $\tilde{\eta}$ to 0 and fixes $\infty$, and we may choose some constant $c_t$ such that $\lim_{|z|\to\infty}|c_t\tilde{f}_t(z)-z|$ is a finite constant. In other words, $c_t\tilde{f}_t$ is the Lowener map of $\tilde{\eta}$. Now if we reparameterize $\tilde{\eta}$ by its half-plane capacity, by applying the results for $\tau = \infty$ case, $\lim_{t\uparrow \tau}\frac{\tilde{f}_t(\varphi(x_1))}{\tilde{f}_t(\varphi(x_2))}=1$, i.e.,
    \begin{equation}\label{eq:lowenerratio1}
        \lim_{t\uparrow \tau}\frac{f_t(x_1)}{f_t(x_2)}\cdot\frac{f_t(\xi)-f_t(x_2)}{f_t(\xi)-f_t(x_1)}=1.
    \end{equation}
    Then~\eqref{eq:lowenerratio1} implies that 
    \begin{equation}\label{eq:lowenerratio2}
        \lim_{t\uparrow \tau}\frac{f_t(x_2)-f_t(x_1)}{f_t(\xi)-f_t(x_1)}\cdot \frac{f_t(\xi)}{f_t(x_2)} = 0.
    \end{equation}
    Since $f_t(x_2)<f_t(\xi)$, by~\eqref{eq:lowenerratio2} we must have $\lim_{t\uparrow \tau}\frac{f_t(x_2)-f_t(x_1)}{f_t(\xi)-f_t(x_1)}=0$, which implies $\lim_{t\uparrow \tau}\frac{f_t(\xi)-f_t(x_1)}{f_t(\xi)-f_t(x_2)}=1$. Therefore it further follows from~\eqref{eq:lowenerratio1} that $\lim_{t\uparrow \tau}\frac{f_t(x_1)}{f_t(x_2)}=1$. The rest of the statements follows analogously.}
\end{proof}

\begin{proposition}\label{prop:martingale-1}
    {Let $\rho_->-2$, {$\kappa>0$}, $\ell>0$, $0^+=x_0<x_1<...<x_\ell<\infty$, and $\rho_0,...,\rho_\ell,\tilde\rho_0,...,\tilde\rho_\ell\in\bbR$ such that for each $0\le k\le\ell$, {$\sum_{i=0}^k\rho_i>\max\{-2,\frac{\kappa}{2}-4\},\ \sum_{i=0}^k\tilde\rho_i>\max\{-2,\frac{\kappa}{2}-4\}$},  {i.e., the curve  a.s.\ does not hit the continuation threshold, does not fill any boundary segment.} Also assume $\rho_0=\tilde\rho_0$ and  {$\bar\rho = \sum_{i=1}^\ell\rho_i=\sum_{i=1}^\ell\tilde\rho_i$}. Let $\eta$ be an $\SLE_\kappa(\rho_-;\rho_0,...,\rho_\ell)$ process with force points $(0^-;0^+,x_1,...,x_\ell)$ with centered Loewner map $(f_t)_{t\ge0}$. {Restrict to the event that $\eta\cap [x_1,x_\ell]=\emptyset$.} Let $D$ be the connected component of $\bbH\backslash\eta$ with $x_\ell$ on its boundary, {and $\sigma,\xi$ be the leftmost and rightmost point of $\bbR\cap \partial D$. Let $\psi:D\to\bbH$ be the conformal map sending $(\sigma,x_\ell,\xi)$ to $(0,1,\infty)$.}} Then if we weight the law of $\eta$ by
    \begin{equation}\label{eq:martingale-1}
        {\frac{1}{Z}\cdot\mathds{1}_{\eta\cap [x_1,x_\ell]=\emptyset}}\cdot\prod_{i=1}^\ell |\psi'(x_i)|^{\frac{(\tilde\rho_i-\rho_i)(\tilde\rho_i+\rho_i+4-\kappa)}{4\kappa}}\cdot  \prod_{i=1}^\ell |\psi(x_i)|^{\frac{(\tilde\rho_i-\rho_i)\rho_0}{2\kappa}}\cdot\prod_{1\le i<j\le \ell}|\psi(x_i)-\psi(x_j)|^{\frac{\tilde\rho_i\tilde\rho_j-\rho_i\rho_j}{2\kappa}}
    \end{equation}
    where $Z = \prod_{i=1}^\ell x_i^{\frac{(\tilde\rho_i-\rho_i)(2+\rho_-+\rho_0)}{2\kappa}}\cdot\prod_{1\le i<j\le \ell}|x_i-x_j|^{\frac{\tilde\rho_i\tilde\rho_j-\rho_i\rho_j}{2\kappa}}$, then $\eta$ evolves as an $\SLE_\kappa(\rho_-;\tilde\rho_0,...,\tilde\rho_n)$ process {restricted to the event that ${\eta\cap [x_1,x_\ell]=\emptyset}$}.
\end{proposition}
\begin{proof}
    Let $D_t$ be the connected component of  $\bbH\backslash\eta([0,t])$ with $x_1,x_2$ on its boundary and recall that $H_{D_t}(x_i,x_j) = \frac{f_t'(x_i)f_t'(x_j)}{(f_t(x_i)-f_t(x_j))^2}$ for $1\le i<j\le n$. Then the local martingale~\eqref{eq:swmartingale} can be rewritten as
    \begin{equation}\label{eq:martingale-2}
    \begin{split}
        M_t = \prod_{1\le i<j\le \ell} H_{D_t}(x_i,x_j)^{\frac{\rho_i\rho_j-\tilde\rho_i\tilde\rho_j}{4\kappa}}&\cdot\prod_{i=1}^{\ell-1} \bigg|\frac{f_t'(x_i)}{f_t'(x_\ell)}\bigg|^{\frac{(\tilde\rho_i-\rho_i)(\bar\rho+4-\kappa)}{4\kappa}}\cdot \prod_{i=1}^{\ell-1}\bigg|\frac{f_t(x_i)-f_t(0^+)}{f_t(x_\ell)-f_t(0^+)}  \bigg|^{\frac{(\tilde\rho_i-\rho_i)\rho_0}{2\kappa} }\\&\cdot  \prod_{i=1}^{\ell-1}\bigg|\frac{f_t(x_i)-f_t(0^-)}{f_t(x_\ell)-f_t(0^-)}  \bigg|^{\frac{(\tilde\rho_i-\rho_i)\rho_-}{2\kappa} }\cdot  \prod_{i=1}^{\ell-1}\bigg|\frac{f_t(x_i)}{f_t(x_\ell)}  \bigg|^{\frac{\tilde\rho_i-\rho_i}{\kappa} }.
        \end{split}
    \end{equation}
    Let $\tau$ be the first time (possibly infinity) when $\eta$ separates $x_\ell$ from infinity.  By Lemma~\ref{lm:martingale}, {under the event $\eta\cap [x_1,x_{\ell}] = \emptyset$, as $t\uparrow\tau$, $M_t$ converges to $Z$ times the expression~\eqref{eq:martingale-1}, and $Z = M_0$}. {Let $\tau_\delta$ be the first time when $\dist(\eta([0,t]), [x_1,x_\ell])<\delta$   {and we set $\tau_\delta=\infty$ if there is no such time}. We first show that $(M_{t\wedge\tau_\delta})_{t>0}$ is bounded. Fix $t<\tau_\delta$. For $1\le j\le \ell-1$, we have
    \begin{equation}
        \bigg|\frac{f_t'(x_j)}{f_t'(x_\ell)}\bigg| = H_{D_t}(x_j,x_\ell)\cdot\frac{(f_t(x_\ell) - f_t(x_j))^2}{f_t'(x_\ell)^2}.
    \end{equation}
    Let $\Omega_j(t)$ be the reflection of the domain $\ol\bbH\backslash (\eta([0,t])\cup (-\infty, x_j))$ over the real line. Then the map $\tilde f_t^j(z):=\frac{\sqrt{f_t(x_j)-f_t(z)}}{\sqrt{f_t(x_\ell)-f_t(x_j)}}$ maps $\Omega_j(t)$ to $\bbH$ sending $x_\ell$ to $i$. By the Koebe-1/4 theorem, $|(\tilde f_t^j)'(x_\ell)| = \frac{f_t'(x_\ell)}{2(f_t(x_\ell)-f_t(x_j))}$ is bounded from above and below by positive constants depending only on $x_j,x_\ell$ and $\delta$. Moreover, using the monotonicity property of the boundary Poisson kernel, if we write $\Omega_\delta$ for the $\delta$-neighborhood of $[x_1,x_\ell]$, then $H_{\Omega_\delta} \le H_{D_t}(x_j,x_\ell)\le H_\bbH(x_j,x_\ell) =  |x_j-x_{\ell}|^{-2}$. Therefore by combining~\eqref{eq:loewner-bound}, it follows that~\eqref{eq:martingale-2} is uniformly bounded for $t<\tau_\delta$.
    
    Now by the Girsanov theorem and Proposition~\ref{prop:martingale}, as we weight the law of ${\eta}$ by $M_{\tau_\delta}$,  {${\eta}$} evolves as an $\SLE_\kappa(\rho_-;\tilde\rho_0,...,\tilde\rho_n)$ process   until time $\tau_\delta$.  In particular, by sending $\delta\to0$, the proposition is true up to the time when $\eta$ first hits $[x_\ell,\infty)$. We also have $\eta\cap[x_1,x_\ell]=\emptyset$ under the weighted measure. Indeed, this can be seen by restricting on the event where $\dist(\eta,[x_1,x_\ell])>\delta$ and then letting $\delta\to0$. Since by Domain Markov property, the law of the both processes are the same $\SLE_\kappa(\rho_-;\sum_{j=0}^\ell \rho_j)$ afterwards, we conclude the proof.}
\end{proof}

\subsection{Proof of Theorem~\ref{thm:IG}}\label{subsec:pf-thm-IG}
In this section we prove Theorem~\ref{thm:IG}. The $n=2$ case is a quick consequence of Theorem~\ref{thm:disk+QT}, while to finish the induction step, we need the following result.

\begin{lemma}\label{lm:weld-IG}
    Let $\gamma\in(0,2)$ and $\kappa=\gamma^2$. Let $n\ge 1$ and {$\rho_0,...,\rho_n\in\bbR$}, such that for each {$0\le k\le n$}, $\sum_{i=0}^k\rho_i>\frac{\kappa}{2}-2$. Let $\rho_{n+1}=\sum_{i=0}^n\rho_i$. Let $W_->0$ and $\rho_-=W_--2$. For $0\le k\le n+1$, let $\beta_k = {\min\{\gamma-\frac{\rho_k}{\gamma}, 2Q-\gamma+\frac{\rho_k}{\gamma}\}}$. Let $(\phi,y_1,...,y_{n-1})$ be sampled from
    \begin{equation}\label{eq:lm-Weld-IG}
        \mathds{1}_{0<y_1<...<y_{n-1}<1} \prod_{0\le i<j\le n}(y_j-y_i)^{\frac{\rho_i\rho_j}{2\kappa} }\LF_\bbH^{(\beta_k,y_k)_{0\le k\le n}, (\beta_{n+1},\infty)}(d\phi)\, dy_1...dy_{n-1}
    \end{equation}
    with the convention $y_0=0$ and $y_n=1$. Let $S=(\bbH, \phi, 0, y_1, ..., y_{n-1}, 1, \infty)/\sim_\gamma$. Then if we conformally weld a weight $W_-$ quantum disk along the edge $(-\infty,0)$ of $S$, then for some constant $c\in(0,\infty)$, the output curve-decorated quantum surface $(S',\eta)$ can be described by $(\bbH,\phi,\eta,0,x_1,...,x_{n-1},1,\infty)/\sim_\gamma$ whose law is a constant multiple of
    \begin{equation}\label{eq:lm-Weld-IG-1}
        \mathds{1}_{0<x_1<...<x_{n-1}<1} \prod_{0\le i<j\le n}(x_j-x_i)^{\frac{\hat\rho_i\hat\rho_j}{2\kappa} }\LF_\bbH^{(\hat\beta_k,x_k)_{0\le k\le n}, (\hat\beta_{n+1},\infty)}(d\phi)\times\SLE_\kappa(\rho_-;\rho_0,...,\rho_n)(d\eta)  \, dx_1...dx_{n-1}
    \end{equation}
    where $x_0=0$, $x_n=1$, $\hat\rho_0 = \rho_-+\rho_0+2$, $\hat\rho_i=\rho_i$ for $1\le i\le n$, $\hat\rho_{n+1} = \rho_{n+1}+\rho_-+2$, and $\hat\beta_k = \gamma-\frac{\hat\rho_k}{\gamma}$ for $0\le k\le n+1$.
\end{lemma}
\begin{proof}
    The proof is very close to Step 2 and Step 3 of the proof of Theorem~\ref{thm:main} and we only list the key steps. Consider the conformal welding of a weight $W_-$ quantum disk and a quantum triangle of weight $(\rho_{n+1}+2, \rho_0+2, \rho_{n+1}-\rho_0+2)$ (which could be embedded as $(\bbH,\phi,0,1,\infty)$ with $\phi\sim\LF_\bbH^{(\beta_0,0),(\beta',1), (\beta_{n+1},\infty)}$ with $\beta'=\gamma-\frac{\rho_{n+1}-\rho_0}{\gamma}$). By Theorem~\ref{thm:disk+QT},  the output surface $S'$ can be embedded as  {$(\bbH,X,\eta,0,1,\infty)$} with 
    \begin{equation}\label{eq:lm-weld-CCC}
        (X,\eta)\sim c\LF_\bbH^{(\hat\beta_0,0),(\beta',1), (\hat\beta_{n+1},\infty)}\times \SLE_\kappa(\rho_-;\rho_0,\rho')
    \end{equation}
  where $\rho'=\rho_{n+1}-\rho_0$ and the force points of $\eta$ is located at $0^-;0^+,1$. Let $D_\eta$ be the component of $\mathbb{H}\backslash \eta$ containing $1$, and $\psi_\eta$ the unique conformal map from $D_\eta$ to $\mathbb{H}$ fixing 0, 1 and $\infty$. Let $Y = X\circ\psi_\eta^{-1}+Q\log|(\psi_\eta^{-1})'|$ and $S_2 = (\bbH,Y,0,1,\infty)/\sim_\gamma$. By Lemma~\ref{lm:gamma-insertion}, as we sample $n-1$ marked points on $S_2$ from the measure (where $y_0=0$, $y_n=1$)
  \begin{equation}\label{eq:lm-weld-DDD}
           \mathds{1}_{0<y_1<...<y_{n-1}<1} \prod_{0\le i<j\le n}(y_j-y_i)^{\frac{\rho_i\rho_j}{2\kappa}}\nu_Y(dy_1)...\nu_Y(dy_{n-1}), 
  \end{equation}
    and let $x_i = \psi_\eta^{-1}(y_i)$, then the law of $(X,\eta, x_1,..., x_{n-1})$ is given by
    \begin{equation}
        \begin{split}
              c\int_{0<x_1<...<x_{n-1}<1} \prod_{0\le i<j\le n}(\psi_\eta(x_j)-\psi_\eta(x_i))^{\frac{\rho_i\rho_j}{2\kappa} }&\LF_\bbH^{(\hat\beta_0,0),(\gamma,x_k)_{1\le k\le n-1},(\beta',1), (\hat\beta_{n+1},\infty)}(dY)\\&\times\SLE_\kappa(\rho_-;\rho_0,\rho')(d\eta)  \, dx_1...dx_{n-1}.  
        \end{split}
    \end{equation}
    Meanwhile, the surface $S$ is now the conformal welding of a weight $W_-$ quantum disk with a surface $S_2' = (\bbH,Y,0,y_1,...,y_{n-1},1,\infty)\sim_\gamma$ whose law can be written as
    \begin{equation}
         \int_{0<y_1<...<y_{n-1}<1} \prod_{0\le i<j\le n}(y_j-y_i)^{\frac{\rho_i\rho_j}{2\kappa} }\LF_\bbH^{(\beta_0,0),(\gamma,y_k)_{1\le k\le n-1},(\beta',1), (\beta_{n+1},\infty)}(dY)\, dy_1...dy_{n-1}.
    \end{equation}

    Now we weight the law of $(X,\eta, x_1,...,x_{n-1})$ by 
    \begin{equation}\label{eq:lm-weld-EEEE}
        \prod_{i=1}^{\ell-1}\big(\e^{\frac{\beta_i^2-\gamma^2}{4}}e^{\frac{\beta_i-\gamma}{2}Y_\e(y_i)} \big)\cdot \e^{\frac{\beta_n^2-(\beta')^2}{4}}e^{\frac{\beta_n-\beta'}{2}Y_\e(1)},
    \end{equation}
    applying Lemma~\ref{lm:lf-change-weight} to the surface $S_2'$, the law of $(X,\eta, x_1,...,x_{n-1})$ is the desired output surface in the conformal welding picture in the Lemma. On the other hand, following the same proof of Step 3 of Proposition~\ref{prop:disk+4-pt-disk}, as $\e\to0$ the law of $(X,\eta, x_1,...,x_{n-1})$ converges vaguely to 
    \begin{equation}
        \begin{split}
              c\int_{0<x_1<...<x_{n-1}<1}& \prod_{0\le i<j\le n}(\psi_\eta(x_j)-\psi_\eta(x_i))^{\frac{\rho_i\rho_j}{2\kappa} }\cdot \prod _{i=1}^{n-1}|\psi_\eta'(x_i)|^{1-\Delta_{\beta_i}}\cdot |\psi_\eta'(1)|^{\Delta_{\beta'}-\Delta_{\beta_n}}\\& \LF_\bbH^{(\hat\beta_k,x_k)_{0\le k\le n-1},(\hat\beta_{n+1},\infty)}(dX)\times\SLE_\kappa(\rho_-;\rho_0,\rho')(d\eta)  \, dx_1...dx_{n-1}.  
        \end{split}
    \end{equation}
    Since $1-\Delta_{\beta_i} = \frac{\rho_i(\rho_i+4-\kappa)}{4\kappa}$, and $\Delta_{\beta'}-\Delta_{\beta_n} = \frac{(\rho_n-\rho')(\rho_n+\rho'+4-\kappa)}{4\kappa}$, it follows from Proposition~\ref{prop:martingale-1} that when weighted by (note $\psi_\eta(0)=0$ and $\psi_\eta(y_n)=1$) $$\prod_{0\le i<j\le n}(\psi_\eta(x_j)-\psi_\eta(x_i))^{\frac{\rho_i\rho_j}{2\kappa} }\cdot \prod _{i=1}^{n-1}|\psi_\eta'(x_i)|^{1-\Delta_{\beta_i}}\cdot |\psi_\eta'(1)|^{\Delta_{\beta'}-\Delta_{\beta_n}},$$ the law of an $\SLE_\kappa(\rho_-;\rho_0,\rho')$ process with force points $0^-;0^+,1$ becomes $\prod_{0\le i<j\le n}(x_j-x_i)^{\frac{\hat\rho_i\hat\rho_j}{2\kappa} }$ times the   $\SLE_\kappa(\rho_-;\rho_0,...,\rho_n)$   with force points $0^-;0^+,x_1,...,x_{n-1},1$ and the claim thus follows.
\end{proof}

\begin{figure}
    \centering
    \begin{tabular}{ccc} 
		\includegraphics[scale=0.62]{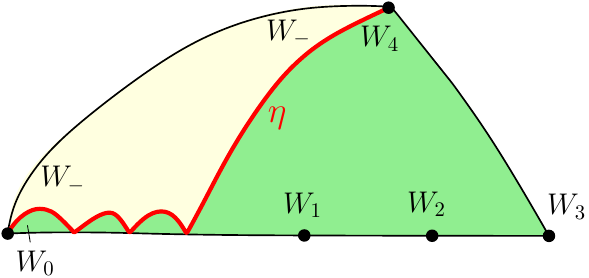}
		& \qquad   &
		\includegraphics[scale=0.7]{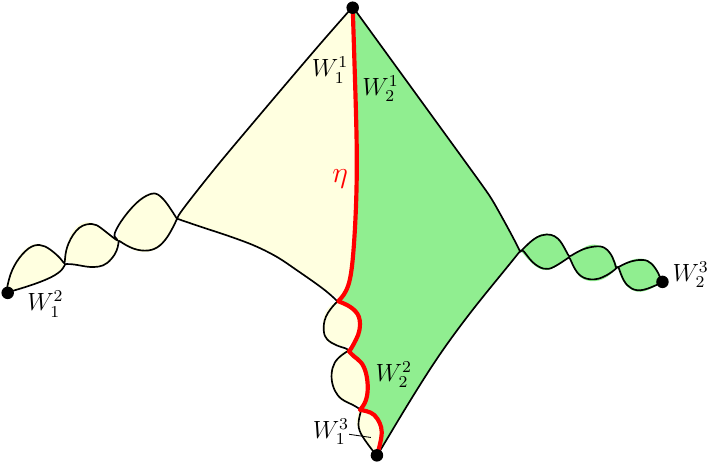}
	\end{tabular}
    \caption{\textbf{Left}: The conformal welding picture in Lemma~\ref{lm:weld-IG-A} with $n = 3$. Here $W_j=\rho_j+2$ for $j=0,1,...,4$. \textbf{Right}: The conformal welding picture in Corollary~\ref{corollary:IG} with $W_1^2,W_1^3,W_2^3<\frac{\gamma^2}{2}$.}
    \label{fig:weld-polygon}
\end{figure}

{We also have the following extension of Lemma~\ref{lm:weld-IG}  {to the case where $\rho_0$ is   less than $\frac{\kappa}{2}-2$.}
\begin{lemma}\label{lm:weld-IG-A}
    Let $\gamma\in(0,2)$, $\kappa=\gamma^2$, $n\ge1$, and $\rho_0\in(-2,\frac{\kappa}{2}-2)$. Set $W_0 = \rho_0+2$.  Let $W_->\frac{\gamma^2}{2}-W_0$ and $\rho_1,...,\rho_{n}\in\bbR$, such that for each {$0\le k\le n$}, $\sum_{i=0}^k\rho_i>\frac{\kappa}{2}-2$. Define $\rho_{n+1}, W_-,(\beta_k)_{0\le k\le n+1}$ in the same way as Lemma~\ref{lm:weld-IG}.  {Let $\mu_0$ be the law ~\eqref{eq:lm-Weld-IG}, $((\phi,y_1,...,y_{n-1}),S_0)$ be sampled from $\mu_0\times\Md_2(W_0)$, and let $S$ be the concatenation of}  $S_0$ with $(\bbH,\phi,0,y_1,...,y_{n-1},1,\infty)/\sim_\gamma$ at the  {marked point corresponding to} $0$. Let $\hat{y}$ be the other endpoint of $S_0$.   Then if we conformally weld a weight $W_-$ quantum disk along the edge between $-\infty$ and $\hat{y}$ of $S$, then an embedding of the output curve-decorated quantum surface $(S',\eta)$ is given by { some positive deterministic constant times}~\eqref{eq:lm-Weld-IG-1} as in Lemma~\ref{lm:weld-IG}.
\end{lemma}}

\begin{proof}
    The proof is almost identical to that of Lemma~\ref{lm:weld-IG}. Consider the conformal welding of a weight $W_-$ quantum disk and a quantum triangle of weight $(\rho_{n+1}+2, \rho_0+2, \rho_{n+1}-\rho_0+2)$ as in Lemma~\ref{lm:weld-IG}, and embed the output surface $S'$   as $(\bbH,X,\eta,0,1,\infty)$ following~\eqref{eq:lm-weld-CCC}.  Let $D_\eta$ be the component of $\mathbb{H}\backslash \eta$ containing $1$, $\sigma$ and $\xi$  be the leftmost and rightmost point of $\partial D\cap\bbR$, and $\psi_\eta$ the unique conformal map from $D_\eta$ to $\mathbb{H}$ sending $(\sigma,1,\xi)$ to $(0,1,\infty)$. Let $Y$ and $S_2$ be the same as in the proof of Lemma~\ref{lm:weld-IG}, and sample $n-1$ marked points $(y_1,...,y_{n-1})$ on $S_2$ from the measure~\eqref{eq:lm-weld-DDD}. Let $x_i = \psi_\eta^{-1}(y_i)$, then the law of $(X,\eta, x_1,...,x_{n-1})$ is given by
    \begin{equation}
        \begin{split}
              c\int_{0<x_1<...<x_{n-1}<1} \mathds{1}_{x_1,...,x_{n-1}\in \partial D_\eta}\prod_{0\le i<j\le n}(\psi_\eta(x_j)-\psi_\eta(x_i))^{\frac{\rho_i\rho_j}{2\kappa} }&\LF_\bbH^{(\hat\beta_0,0),(\gamma,y_k)_{1\le k\le n-1},(\beta',1), (\hat\beta_{n+1},\infty)}(dX)\\&\times\SLE_\kappa(\rho_-;\rho_0,\rho')(d\eta)  \, dx_1...dx_{n-1}.  
        \end{split}
    \end{equation}
    Now we weight the law of $(X,\eta, x_1,...,x_{n-1})$ by~\eqref{eq:lm-weld-EEEE}, and the rest of the proof is the same as in Lemma~\ref{lm:weld-IG} where we apply Proposition~\ref{prop:martingale-1} (and note that by assumption the $\SLE_\kappa(\rho_-;\rho_0,...,\rho_n)$ curve will a.s.\ not hit $[x_1,x_n]$).
\end{proof}

\begin{proof}[Proof of Theorem~\ref{thm:IG}]
    When $n=2$, the theorem follows directly by applying Theorem~\ref{thm:disk+QT} twice. In this setting, the marginal law of $\eta_1$ is $\SLE_\kappa(W_0-2;W_1^2-2,W_1^3+W_2-2)$ with force points $0^-;0^+,1$, while given $\eta_1$, $\eta_2$ is the $\SLE_\kappa(W_1^3-2,W_1^2-2;W_2-2)$ curve in the right component of $\bbH\backslash\eta_1$ with force points $1^-,0;1^+$, which agrees with the $\mathrm{IG}$ description from~\cite[Theorem 1.1]{MS16a}. 

    Now suppose the claim has been proved for $n$. For $n+1$, we work on the cases where $W_0+W_1^2$ is less, equal or larger than 2. As we shall see from the proof, when   $W_0+W_1^2=2$,  our theorem may be achieved by adding a marked point in the case for $n$ since the weight 2 is typical with respect to the quantum length measure. When $W_0+W_1^2$ is less or greater than 2, then we shall apply Lemma~\ref{lm:weld-IG} to cut or glue a quantum disk to match the corresponding picture.  {Throughout the proof, $c$ is some positive constant that may vary line by line.}

    \emph{Case 1. $W_0+W_1^2 = 2$.} Then $\rho_1=0$. Consider the conformal welding of samples from $\Md_2(W_1^3),\\\QT(W_2^1,W_2^2,W_2^3),...,\QT(W_n^1,W_n^2,W_n^3),\Md_2(W_{n+1})$ in that order. Then these weights satisfy the constraint in the theorem statement and therefore by our induction hypothesis, the output curve-decorated quantum surface can be embedded as $(\bbH, X', x_2', ..., x_{n+1}', \eta_2', ..., \eta_{n+1}')$ whose law is  
    \begin{equation}
        c\mathds{1}_{0<x_3'<...<x_n'<1}\prod_{2\le i<j\le n+1 }(x_j'-x_i')^{\frac{\rho_i\rho_j}{2\kappa}} \LF_\bbH^{(\beta_j,x_j')_{2\le j\le n+1},(\beta_\infty,\infty)}(dX')\times \IG_{\underline{x}',\underline{\lambda}',\underline{\theta}' }^\#(d\eta_2'...d\eta_{n+1}')  dx_3'...dx_{n}'
    \end{equation}
    where $\underline{x}' = (0,x_3',...,x_n',1)$, $\underline{\lambda}' = (\lambda_1-\lambda W_1^1,...,\lambda_{n+1}-\lambda W_1^1)$ and $\underline{\theta}' = (0,-\frac{\lambda W_2^1}{\chi}, ..., -\frac{\lambda(W_2^1+...+W_{n+1}^1)}{\chi})$, with the convention $x_2'=0,x_{n+1}'=1$. Then as we weight the law of $(X', x_2', ..., x_{n+1}', \eta_2', ..., \eta_{n+1}')$ by $\nu_{X'}((-\infty,0))$ and uniformly sample a point $x'$ from the probability measure proportional to $\nu_{X'}|_{(-\infty,0)}$, then by Lemma~\ref{lm:gamma-insertion}, the law of $(X', x',x_2', ..., x_{n+1}', \eta_2', ..., \eta_{n+1}')$ is 
    \begin{equation}\label{eq:thmpf-IG-1}
        c\mathds{1}_{x'<0<x_3'<...<x_n'<1}\prod_{2\le i<j\le n+1 }(x_j'-x_i')^{\frac{\rho_i\rho_j}{2\kappa}} \LF_\bbH^{(\gamma,y),(\beta_j,x_j')_{2\le j\le n+1},(\beta_\infty,\infty)}(dX')\times \IG_{\underline{x}',\underline{\lambda}',\underline{\theta}' }^\#(d\eta_2'...d\eta_{n+1}')  dx_3'...dx_{n}'.
    \end{equation}
    Consider the conformal map $f_{x'}(z):=\frac{z-x'}{1-x'}$. Let $x_j = f_{x'}(x_j')$, $X = X'\circ f_{x'}^{-1}+Q\log|(f_{x'}^{-1})'|$, $\eta_j = f_{x'}\circ\eta_j'$ for $2\le j\le n$. Then $\frac{\partial (x',x_3', ..., x_n') }{\partial (x_2, ..., x_n)} = (1-x_2)^{-n}$, and by Lemma~\ref{lm:lcft-H-conf}, the law of $(X, x_2, ..., x_n,\eta_2,...,\eta_n)$ is 
    \begin{equation}\label{eq:thmpf-IG-2}
       c \mathds{1}_{0<x_2<x_3<...<x_n<1}\prod_{2\le i<j\le n+1 }(x_j-x_i)^{\frac{\rho_i\rho_j}{2\kappa}} \LF_\bbH^{(\gamma,0),(\beta_j,x_j)_{2\le j\le n+1},(\beta_\infty,\infty)}(dX)\times \IG_{\underline{x}'',\underline{\lambda}',\underline{\theta}' }^\#(d\eta_2...d\eta_{n+1})  dx_2...dx_{n}
    \end{equation}
    with $\underline{x}''=(x_2,...,x_n,1)$. Here we used the fact $\sum_{j=2}^{n+1}\Delta_{\beta_j}-\Delta_{\beta_\infty} - \sum_{2\le i<j\le n+1}\frac{\rho_i\rho_j}{2\kappa} - n=0$. Given $(\eta_2,...,\eta_{n+1})$, if we draw an $\SLE_\kappa(W_0-2;W_1^2-2,W_1^3-2)$ curve $\eta_1$ with force points $0^-;0^+,x_2$ in the left component of $\bbH\backslash\eta_2$, then the law $(\eta_1,...,\eta_{n+1})$ agrees with the flow line picture $\IG_{\underline x, \underline \lambda'', \underline\theta''}^\#$ where $\lambda''_i = \lambda_i-\lambda W_1^1$ and $\theta_i'' = \theta_i+\frac{\lambda W_1^1}{\chi}$ for $0\le i\le n+1$, which further agrees with $\IG_{\underline x, \underline \lambda, \underline\theta}^\#$ by definition. This implies that as we conformally weld  samples from $\Md_{2,\bullet}(W_1^3),\QT(W_2^1,W_2^2,W_2^3),...,\QT(W_n^1,W_n^2,W_n^3),\Md_2(W_{n+1})$ in that order and draw the interface $\eta_1$ on the weight $W_1^3$ triply marked quantum disk, the law of the curve-decorated quantum surface has embedding $(X, x_1, ..., x_{n+1}, \eta_1, ..., \eta_{n+1})$ whose law is a constant times the right hand side of~\eqref{eq:thm-IG} for $n+1$. On the other hand, recall from Lemma~\ref{lem:QT(W,W,2)} that $\Md_{2,\bullet}(W) = C\QT(W,W,2)$, and therefore by Theorem~\ref{thm:disk+QT}, we may first weld the weight $W_0$ quantum disk with the weight $(W_1^1,W_1^2,W_1^3)$ quantum triangle together, which gives a weight $(W_1^3,2,W_1^3)$ quantum triangle decorated with an $\SLE_\kappa(W_0-2;W_1^2-2,W_1^3-2)$ curve. This concludes the proof for $W_0+W_1^2 = 2$.

    \emph{Case 2. $W_0+W_1^2>2$.} First assume $W_1^2=2$. Consider the conformal welding of samples from $\Md_{2,\bullet}(W_1^3),\QT(W_2^1,W_2^2,W_2^3),...,\QT(W_n^1,W_n^2,W_n^3),\Md_2(W_{n+1})$ as in the previous step. Then  {as argued in the previous case}, the output surface has embedding specified by~\eqref{eq:thmpf-IG-2}. Since $\rho_2, ..., \rho_{n+1}>\frac{\kappa}{2}-2$ and $\sum_{i=2}^k \rho_i>\frac{\kappa}{2}-2$ for each $2\le k\le n+1$,  by Lemma~\ref{lm:weld-IG}, if we weld the weight $W_0$ quantum disk to the left, then the output surface has embedding $(\bbH,Y,y_1,..., y_{n+1},\eta_1,...,\eta_{n+1})$ with $y_1=0$ and $y_{n+1}=1$, where the law of $(Y,y_1,..., y_{n+1},\eta_1,...,\eta_{n+1})$ is
    \begin{equation}
        c\mathds{1}_{0<y_2<...<y_n<1}\prod_{1\le i<j\le n+1}(y_j-y_i)^{\frac{\rho_i\rho_j}{2\kappa} }\LF_\bbH^{(\beta_i,y_i)_{1\le i\le n+1}, (\beta_\infty,\infty)}(dY)\times \mathcal{L}_{\underline y}(d\eta_1...d\eta_{n+1})dy_2...dy_{n}
    \end{equation}
    where a sample from $\mathcal{L}_{\underline y}$ is obtained by (i) first sample an $\SLE_\kappa(W_0-2;\rho_2,...,\rho_{n+1})$ process $\eta_1$ with force points $0^-;y_2,...,y_{n+1}$ and then (ii) sample $(\eta_2,...,\eta_{n+1})$ in the right component of $\bbH\backslash\eta_1$ from $ \IG_{\underline{x}',\underline{\lambda}',\underline{\theta}'}^\#$ with $\underline{\lambda}',\underline{\theta}'$ given in Step 1 and $\underline {x}' = \psi_{\eta_1}(\underline y)$ (where $\psi_{\eta_1}$ is the conformal map from the right component of $\bbH\backslash\eta_1$ fixing 0,1,$\infty$). By Theorem~\ref{thm:IGflowline}, this implies that for fixed $\underline y$, we have $\mathcal{L}_{\underline y} = \IG_{\underline y, \underline \lambda, \underline\theta}^\#$ and thus finishes the $W_1^2=2$ case. 

    \begin{figure}
    \centering
    \begin{tabular}{ccc} 
		\includegraphics[scale=0.45]{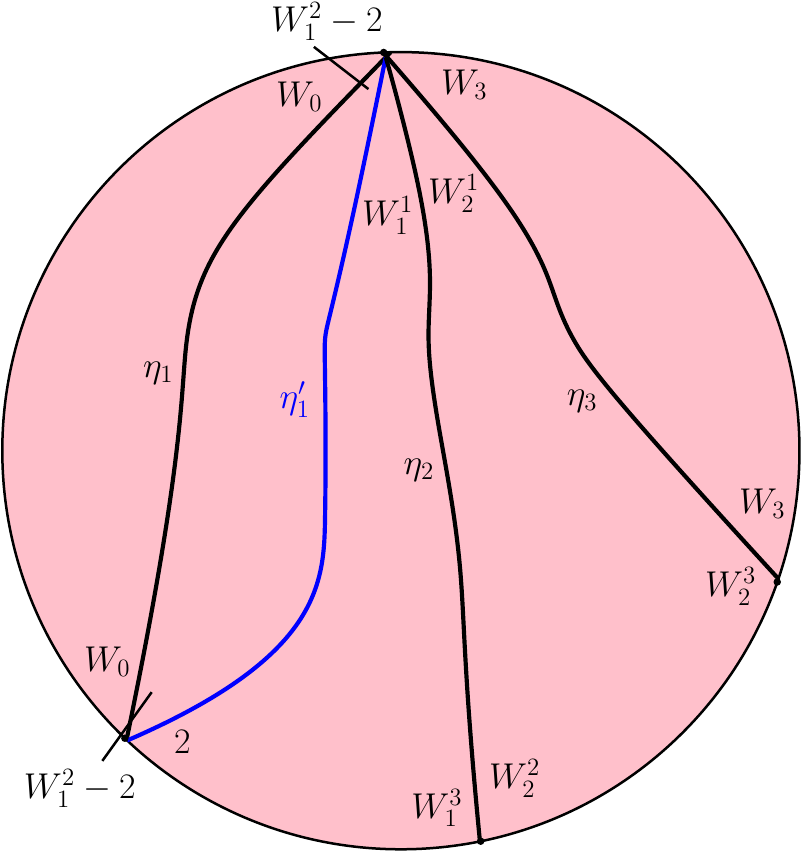}
		& \qquad \qquad \qquad &
		\includegraphics[scale=0.45]{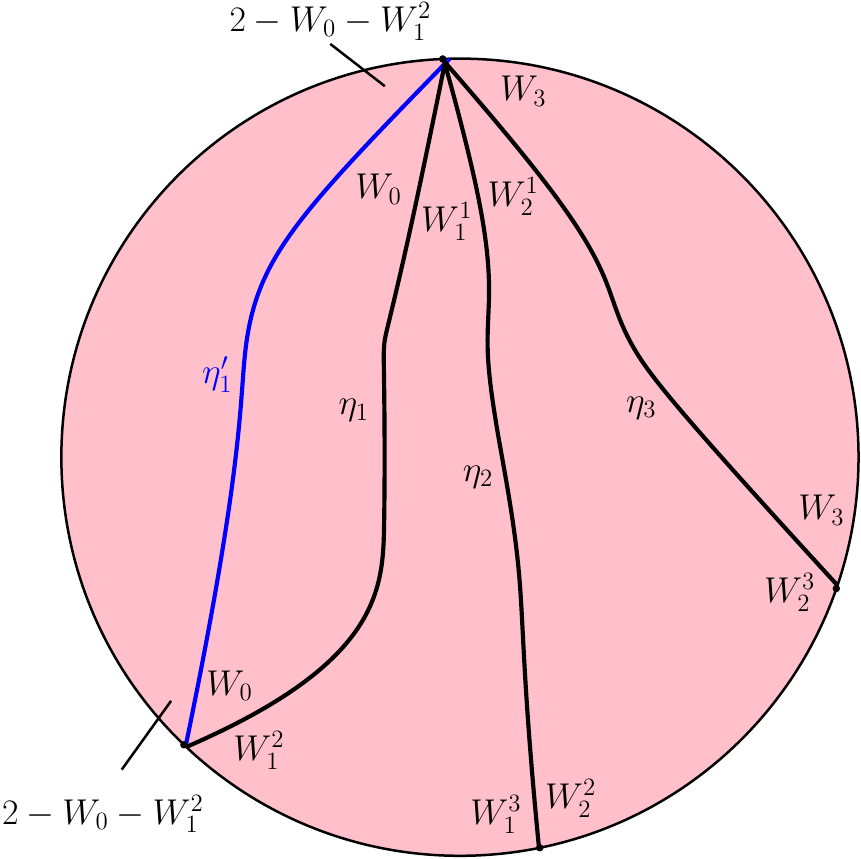}
	\end{tabular}
 \caption{Proof of Theorem~\ref{thm:IG}  for the induction step with $n+1=3$. \textbf{Left:} Case 2 when $W_0+W_1^2>2$ and $W_1^2>2$. \textbf{Right:} Case 3 when $W_0+W_1^2<2$.}\label{fig:thmpf-ig}
 \end{figure}
    
    If $W_1^2>2$, then consider the conformal welding of samples from $\Md_{2}(W_0),\Md_2(W_1^2-2), \QT(W_1^3,2,\\W_1^3),\QT(W_2^1,W_2^2,W_2^3),...,\QT(W_n^1,W_n^2,W_n^3),\Md_2(W_{n+1})$ (note $W_1^1-(W_1^2-2)=W_1^3$). See the left panel of Figure~\ref{fig:thmpf-ig} for an example when $n=3$. By  Theorem~\hyperref[thm:disk-welding]{A}, we may first weld the left two quantum disks together, where we obtain a weight $W_0+W_1^2-2$ quantum disk decorated with an $\SLE_\kappa(W_0-2;W_1^2-4)$ interface $\eta_1$. Then applying the result for $W_1^2=2$ case, it follows that the output surface we get has law 
      \begin{equation}\label{eq:thmpf-IG-3}
        c\mathds{1}_{0<y_2<...<y_n<1}\prod_{1\le i<j\le n+1}(y_j-y_i)^{\frac{\rho_i\rho_j}{2\kappa} }\LF_\bbH^{(\beta_i,y_i)_{1\le i\le n+1}, (\beta_\infty,\infty)}\times \mathcal{L}_{\underline y}'(d\eta_1d\eta_1'...d\eta_{n+1})dy_2...dy_{n}
    \end{equation}
    where a sample $(\eta_1,\eta_1',\eta_2,...,\eta_{n+1})$ from $\mathcal{L}_{\underline y}'$ is constructed by (i) sample $(\eta_1',\eta_2,...,\eta_{n+1})$ from $\IG_{\underline y, \underline\lambda''',\underline\theta'''}^\#$ with $\lambda'''_i = \lambda_i+\lambda(2-W_1^2)$ and $\theta'''_i = \theta_i-\frac{\lambda(2-W_1^2)}{\chi}$ and (ii) sample $\eta_1$ as an  $\SLE_\kappa(W_0-2;W_1^2-4)$ process in the left component of $\bbH\backslash\eta_1'$. On the other hand, we may apply Theorem~\ref{thm:disk+QT} to weld the weight $W_1^2-2$ quantum disk together with the weight $(W_1^3,2,W_1^3)$ quantum triangle, which gives the desired welding picture in the theorem statement once we forget the interface $\eta_1'$.  Therefore by Theorem~\ref{thm:IGflowline}, the jonit law of $(\eta_1,...,\eta_{n+1})$ is $\IG_{\underline y, \underline\lambda,\underline\theta}^\#$ and the claim  follows for $W_1^2>2$. The $W_1^2<2$ case is treated analogously, where we consider the conformal welding of samples from $\Md_{2}(W_0+W_1^2-2),\Md_2(2-W_1^2), \QT(W_1^1,W_1^2,W_1^3),\QT(W_2^1,W_2^2,W_2^3),...,\QT(W_n^1,W_n^2,W_n^3),\Md_2(W_{n+1})$ and first weld the weight $2-W_1^2$ quantum disk with the weight  $W_1^1,W_1^2,W_1^3$ quantum triangle and then apply the result for $W_1^2=2$ case. We omit the details.

    \emph{Case 3. $W_0+W_1^2<2$.} Consider the conformal welding of samples from $\Md_2(2-W_0-W_1^2), \Md_2(W_0), \\\QT(W_1^1,W_1^2,W_1^3),...,\QT(W_{n}^1,W_{n}^2,W_n^3),\Md_2(W_{n+1})$. See the right panel of Figure~\ref{fig:thmpf-ig} for an example when $n=3$. Then by Theorem~\hyperref[thm:disk-welding]{A} and Case 1, the output surface has embedding $(\bbH,X,y_1,..., y_{n+1},\eta_1',\eta_1,...,\\ \eta_{n+1})$ with $y_1=0$ and $y_{n+1}=1$, where the law of $(X,y_1,..., y_{n+1},\eta_1,..., \eta_{n+1})$ is
    \begin{equation}\label{eq:thmpf-IG-4}
        c\mathds{1}_{0<y_2<...<y_n<1}\prod_{1\le i<j\le n+1}(y_j-y_i)^{\frac{\rho_i'\rho_j'}{2\kappa} }\LF_\bbH^{(\beta_i',y_i)_{1\le i\le n+1}, (\beta_\infty',\infty)}\times \mathcal{L}_{\underline y}''(d\eta_1'd\eta_1...d\eta_{n+1})dy_2...dy_{n}
    \end{equation}
    where $\rho_1'=0$, $\beta_1'=\gamma$, $\beta_\infty' = \beta_\infty-\frac{2-W_0-W_1^2}{\gamma}$, $\rho_i'=\rho_i$ and $\beta_i'=\beta_i$ for $2\le i\le n+1$, while a sample from  $ \mathcal{L}_{\underline y}''$ is produced by (i) sample an $\SLE_\kappa(-W_0-W_1^2;\rho_1,\rho_2,...,\rho_{n+1})$ process $\eta_1'$ with force points $0^-;0^+,y_2,...,y_{n+1}$ and (ii) sample $(\eta_1,...,\eta_{n+1})$ from $\psi_{\eta_1'}^{-1}\circ\IG_{\underline x, \underline \lambda, \underline\theta}^\#$ where $\underline {x} = \psi_{\eta_1'}(\underline y)$. On the other hand, by Lemma~\ref{lm:weld-IG}, if we conformally weld a weight $2-W_0-W_1^2$ quantum disk to the $(-\infty,0)$ edge of the right hand side of~\eqref{eq:thm-IG} (for $n+1$), then the output curve-decorated quantum surface equals in law with~\eqref{eq:thmpf-IG-4}. This implies that the conformal welding of samples from $\Md_2(W_0), \QT(W_1^1,W_1^2,W_1^3),...,\QT(W_{n}^1,W_{n}^2,\\W_n^3),\Md_2(W_{n+1})$ equals in law with the right hand side of~\eqref{eq:thm-IG} (for $n+1$) and thus concludes the induction step and the proof.
 \end{proof}

As a consequence, we have the following result regarding the conformal welding of two quantum triangles.
\begin{corollary}\label{corollary:IG}
    Let $W_1^1,W_1^2,W_1^3,W_2^1,W_2^2,W_2^3>0$, such that $W_1^1+W_2^1, W_1^1+W_2^2, W_2^1+W_1^3, W_1^3+W_2^2>\frac{\gamma^2}{2}$, and $W_1^3+W_1^2-W_1^1 = W_2^3+W_2^2-W_2^1=2$. Further assume that $W_1^2, W_2^3\neq\frac{\gamma^2}{2}$. Let $\beta_1=\min\{\gamma+\frac{2-W_1^2}{\gamma},\frac{2+W_1^2}{\gamma}\}, \beta_2 = \gamma+\frac{2-W_1^3-W_2^2}{\gamma}, \beta_3 = \min\{\gamma+\frac{2-W_2^3}{\gamma},\frac{2+W_2^3}{\gamma}\}, \beta_\infty = \gamma+\frac{2-W_1^1-W_2^1}{\gamma}$. Consider the conformal welding of a weight $(W_1^1,W_1^2,W_1^3)$ quantum triangle and a weight $(W_2^1,W_2^2, W_2^3)$ quantum triangle as in Figure~\ref{fig:weld-polygon}. Let $\rho_i^j = W_i^j - 2$ for $i=1,2$ and $j=2,3$. If $W_1^2, W_2^3>\frac{\gamma^2}{2}$, then the output curve-decorated quantum surface $(S,\eta)$ can be embedded as $(\bbH,\phi,\eta,0,x,1,\infty) $ whose law can be described by 
    \begin{equation}\label{eq:weld-polygon-thin}
        c\mathds{1}_{x\in(0,1)} \LF_\bbH^{(\beta_1,0),(\beta_2,x),(\beta_3,1),(\beta_\infty,\infty)}(d\phi)\times x^{\frac{\rho_1^2(2+\rho_1^3+\rho_2^2)}{2\kappa}}(1-x)^{\frac{\rho_2^3(2+\rho_1^3+\rho_2^2)}{2\kappa}}\SLE_\kappa(\rho_1^3,\rho_1^2;\rho_2^2,\rho_2^3)(d\eta)\, dx
    \end{equation}
    where $\eta$ is from $x$ to $\infty$ with the force points  located at $x^-,0;x^+,1$.  {If $W_1^2,W_2^3<\frac{\gamma^2}{2}$, let $\mu_0$ be the law~\eqref{eq:weld-polygon-thin} and $((\phi_0,\eta_0,x),S_1,S_2)\sim \mu_0\times \Md_2(W_1^2)\times\Md_2(W_2^3)$. Then the output curve-decorated quantum surface $(S,\eta)$ is the concatenation of $((S_0,\eta_0),S_1,S_2)$ with $(S_0,\eta_0) = (\bbH,\phi_0,\eta_0,0,x,1,\infty)/\sim_\gamma$ at the points $0$ and $1$ correspondingly.   Similar description holds   when $W_1^2>\frac{\gamma^2}{2}$,  $W_2^3<\frac{\gamma^2}{2}$ or $W_1^2<\frac{\gamma^2}{2}$,  $W_2^3>\frac{\gamma^2}{2}$ (via concatenation of  thin quantum disks) as well.}
\end{corollary}
\begin{proof}
   Pick $W_0,W_3>\frac{\gamma^2}{2}$, and consider the conformal welding of samples from 
$\Md_2(W_0), \QT(W^1_1,W_1^2,W_1^3), \\ \QT(W_{2}^1,W_{2}^2,W_{2}^3),\Md_2(W_3)$  as in Figure~\ref{fig:thm-IG}. Then the conclusion follows by first applying Theorem~\ref{thm:IG}, and then applying Lemma~\ref{lm:weld-IG} or Lemma~\ref{lm:weld-IG-A} twice to cut off the weight $W_0$ and the weight $W_3$ quantum disk  as in Case 3 in the proof of Theorem~\ref{thm:IG}. We omit the details.
\end{proof}

\section{Boundary Green's function via conformal welding}\label{sec:green}
In this section we prove Theorem~\ref{thm:SLE-Green} and Theorem~\ref{thm:green-rho} which  gives the boundary  Green's function of $\SLE$ from conformal welding of LQG disks. We begin with Theorem~\ref{thm:green-rho}, and Theorem~\ref{thm:SLE-Green} follows analogously.

\begin{figure}
    \centering
    \begin{tabular}{cccc} 
		\includegraphics[scale=0.36]{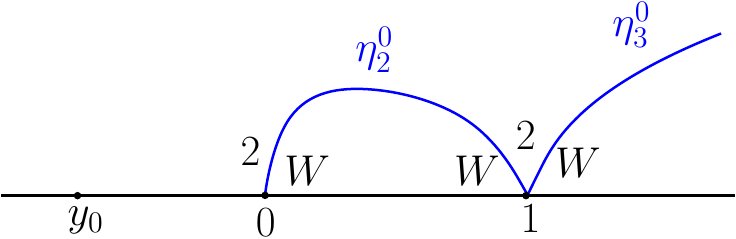}
		& \qquad &
		\includegraphics[scale=0.36]{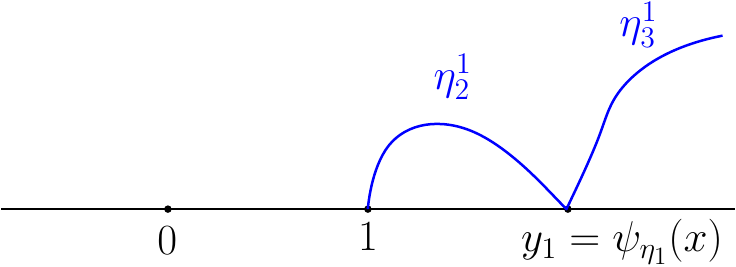}
  \qquad &
		\includegraphics[scale=0.36]{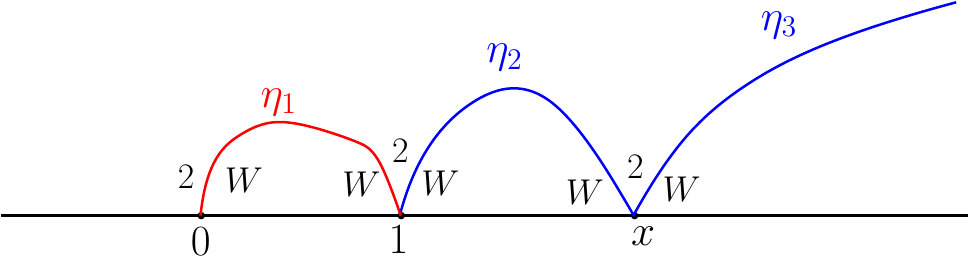}
	\end{tabular}
 \caption{Proof of Theorem~\ref{thm:green-rho}. \textbf{Left:} Conformal welding of two samples from $\Md_2(W)$ and a sample from $\QD_{3}$. We add a marked point $y_0$ from the LQG length measure on $(-\infty,0)$. \textbf{Middle:} We take the conformal map $f_{y_0}(z) = \frac{y_0-z}{y_0}$.  \textbf{Right:} We weld a weight $W$ quantum disk along the $(0,1)$ edge of the middle picture.}\label{fig:pfgreen-rho}
 \end{figure}

\begin{proof}[Proof of Theorem~\ref{thm:green-rho}]
    We begin with the conformal welding of two samples from $\Md_2(W)$ and a sample from $\QD_{3}$, embedded as $(\bbH, Y_0,\eta_2^0,\eta_3^0, 0,1,\infty)$ as in the left panel of Figure~\ref{fig:pfgreen-rho}. By Theorem~\ref{thm:disk+QT}, $\eta_2^0$ is an $\SLE_\kappa(2+\rho;\rho)$ process from $0$ to 1 with force points located at $\infty;0^+$, while given $\eta_2^0$, $\eta_3^0$ is an $\SLE_\kappa(\rho)$ curve from 1 to $\infty$ in the unbounded component of $\bbH\backslash\eta_2^0$. Note that by~\cite[Theorem 3]{SW05}, the law of $\eta_2^0$ is also the same as an $\SLE_\kappa(\rho,\kappa-8-2\rho)$ from 0 to $\infty$ with force points at $0^+,1$. Then we weight the law of $(Y_0,\eta_2^0,\eta_3^0)$ by $\nu_{Y_0}|_{(-\infty,0)}$ and sample a marked point $y_0$ on $(-\infty,0)$ according to $\nu_{Y_0}|_{(-\infty,0)}^\#$. By Lemma~\ref{lm:gamma-insertion}, the joint law of $(Y_0,y_0)$ is a constant times
    \begin{equation}
        \int_{-\infty}^0 \LF_\bbH^{(\gamma,y_0),(\beta_\rho, 0), (\beta_{2,\rho}, 1), (\beta_\rho,\infty)}(dY_0) dy_0.
    \end{equation}
Consider the conformal map $f_{y_0}(z) = \frac{y_0-z}{y_0}$. Let $Y_1 =  f_{y_0}\bullet_\gamma Y_0$ and $y_1 = f_{y_0}(1) = \frac{y_0-1}{y_0}$, $\eta_2^1 = f_{y_0}\circ\eta_2^0$ and $\eta_3^1 = f_{y_0}\circ \eta_3^0$. By Lemma~\ref{lm:lcft-H-conf}, the joint law of $(Y_1,y_1)$ is a constant multiple of
\begin{equation}\label{eq:pf-thm-green-rho-1}
       \mathds{1}_{y_1\in(1,\infty)}  (y_1-1)^{\Delta_{\beta_{2,\rho}}-1}\LF_\bbH^{(\gamma,0),(\beta_\rho, 1), (\beta_{2,\rho}, y_1), (\beta_\rho,\infty)}(dY_1) dy_1.
    \end{equation}

Now consider the conformal welding of a weight $W$ quantum disk and a weight $(2+W,2+W,2)$ quantum triangle embedded as $(\bbH, X,\eta_1, 0,1, \infty)$ with the interface $\eta_1$ running from $0$ to $1$. By Theorem~\ref{thm:disk+QT} and~\cite[Theorem 3]{SW05}, the law of $\eta_1$ is the $\SLE_\kappa(\rho,\kappa-8-2\rho)$ from 0 to $\infty$ with force points at $0^+,1$. Let $D_{\eta_1}$ be the unbounded component of $\bbH\backslash\eta_1$ and $\psi_{\eta_1}:D_{\eta_1}\to\bbH$ be the conformal map fixing 0,1,$\infty$. Let $D_{\eta_1}^L=\bbH\backslash D_{\eta_1}$, and $Z = (D_{\eta_1}^L,X,0,1)$. Also let $Y = \psi_{\eta_1}\bullet_\gamma X$. By  Lemma~\ref{lm:gamma-insertion}, as we sample a point $y_1'$ from the measure $(y_1'-1)^{\Delta_{\beta_{2,\rho}}-1}\nu_Y|_{(1,\infty)}(dy_1')$  (which induces a weighting over the law of $(X,\eta_1)$), the joint law of $(Z,Y,y_1')$ is given by
\begin{equation}
    \mathds{1}_{y_1'\in(1,\infty)} \int_0^\infty \Md_2(W;\ell)(dZ) (y_1'-1)^{\Delta_{\beta_{2,\rho}}-1}\LF_{\bbH,\ell}^{(\gamma,0),(\beta_\rho, 1), (\gamma, y_1'), (\beta_\rho,\infty)}(dY)\,d\ell  d{y_1'}
\end{equation}
where $\ell$ indicate the quantum length $\nu_Y(0,1)$. Then if we set $x_1 = \psi_{\eta_1}^{-1}(y_1')$,  the joint law of $(X,\eta_1,x_1)$ is 
\begin{equation}\label{eq:pf-thm-green-rho-2}
        \mathds{1}_{x_1\in(1,\infty)}    (\psi_{\eta_1}(x_1)-1)^{\Delta_{\beta_{2,\rho}}-1}\LF_\bbH^{(\beta_\rho,0),(\beta_\rho, 1), (\gamma, x_1), (\beta_\rho,\infty)}(dX) \times\SLE_\kappa(\rho,\kappa-8-2\rho)(d\eta_1)d{x_1}.
    \end{equation}
We weight the law of    $(X,\eta_1,x_1)$ by $\e^{\frac{\beta_{2,\rho}^2-\gamma^2}{4}}e^{\frac{\beta_{2,\rho}-\gamma}{2}Y_\e(y_1')}$. Following Lemma~\ref{lm:lf-change-weight} and the same argument as in the proof of Proposition~\ref{prop:disk+4-pt-disk}, as $\e\to0$ the joint law of $(X,\eta_1,x_1)$ converges vaguely to  
\begin{equation}\label{eq:pf-thm-green-rho-3}
        \mathds{1}_{x_1\in(1,\infty)} \psi_{\eta_1}'(x_1)^{1-\Delta_{\beta_{2,\rho}}}(\psi_{\eta_1}(x_1)-1)^{\Delta_{\beta_{2,\rho}}-1}\LF_\bbH^{(\beta_\rho,0),(\beta_\rho, 1), (\gamma, x_1), (\beta_\rho,\infty)}(dX) \times\SLE_\kappa(\rho,\kappa-8-2\rho)(d\eta_1)dx_1,
    \end{equation}
    while the joint law of $(Z,Y,y_1')$ converges vaguely to 
    \begin{equation}
    \mathds{1}_{y_1'\in(1,\infty)} \int_1^\infty\int_0^\infty \Md_2(W;\ell)(dZ) (y_1'-1)^{\Delta_{\beta_{2,\rho}}-1}\LF_{\bbH,\ell}^{(\gamma,0),(\beta_\rho, 1), (\beta_{2,\rho}, y_1'), (\beta_\rho,\infty)}(dY)d\ell dy_1'.
\end{equation}
 Note that $\psi_{\eta_1}-1$ is a constant multiple of $f_{\tau_1}$ (the centered Loewner map of $\eta_1$ at time $\tau_1$ when hitting 1), it follows that $\psi_{\eta_1}'(x_1)^{1-\Delta_{\beta_{2,\rho}}}(\psi_{\eta_1}(x_1)-1)^{\Delta_{\beta_{2,\rho}}-1} = f_\tau'(x_1)^{1-\Delta_{\beta_{2,\rho}}}(f_\tau(x_1)-1)^{\Delta_{\beta_{2,\rho}}-1}$. Also observe that $1-\Delta_{\beta_{2,\rho}} = b_\rho$. Since the law of $(\eta_2^1,\eta_3^1)$ agrees with the conditional law of $(\eta_2,\eta_3)$ given $\eta_1$ under $M(\rho;x)$ (as probability measure), it follows that the welding picture induced by the right panel of Figure~\ref{fig:pfgreen-rho} agrees with~\eqref{eq:thm-green-rho} and thus we conclude the proof.
\end{proof}

Before proving Theorem~\ref{thm:SLE-Green}, we observe the following dilation rule for $M_\alpha(x_0,...,x_N)$. 
\begin{lemma}\label{lm:green-dilation}
    Let $N\ge2,\alpha\in\mathrm{CLP}_N$ and $x_0<...<x_{N-1}$. Let $\varphi(z)=az+b$ where $a>0,b\in\bbR$. Then 
    \begin{equation}\label{eq:green-dilation}
        \varphi\circ M_\alpha(x_0,...,x_{N-1}) = a^{(N-1)b_2} M_\alpha(\varphi(x_0),...,\varphi(x_{N-1})).
    \end{equation}
\end{lemma}
\begin{proof}
    Note that if $\eta$ is an $\SLE_\kappa(\kappa-8)$ curve from $x$ to $\infty$ with force point at $y$, then $\varphi\circ\eta$ is an $\SLE_\kappa(\kappa-8)$ curve from $\varphi(x)$ to $\infty$ with force point at $\varphi(y)$. Moreover, if $(f_t)_{t>0}$ is  the centered Loewner map for $\eta$, then the centered Loewner map for $\varphi\circ\eta$ is $(af_{t/a}(\cdot/a))_{t>0}$. The claim then follows by induction.
\end{proof}

\begin{proof}[Proof of Theorem~\ref{thm:SLE-Green}]
    Similar to the proof of Theorem~\ref{thm:green-rho} for $W=2$, the $N=2$ case is a direct consequence of Theorem~\ref{thm:disk+QT} and~\cite[Theorem 3]{SW05}. Now assume the theorem has been proved for $N$ and we are working on $N+1$ regime. By symmetry we only work on the case $i_1>i_0$.

    First assume $i_2>i_1$. Let $\alpha^0$ be the curve link pattern induced by removing the first segment of $\alpha$. Consider the picture induced by $\Wd_{\alpha^0}(\QD^{N+1})$. By our induction hypothesis, the entire surface and interface can be embedded as $(\bbH,Y_0,\eta_2^0,...,\eta_{N+1}^0,y_{i_0}',...,y_1',0,1,x_2',...,x'_{N-i_0-1})$ whose law is
    \begin{equation}\label{eq:pf-slegreen}
        \begin{split}
    c\cdot&\mathds{1}_{y'_{i_0}<...<y'_1<0<1<x'_2<...<x'_{N-i_0-1}}\bigg[\LF_\bbH^{(\beta,0),(\beta,\infty),(\beta_2,1),(\beta_2,x'_2),...,(\beta_2,x'_{N-i_0-1}),(\beta_2,y'_1),...,(\beta,y'_{i_0})}(dY_0)\times \\ &M_{\alpha^0}(y'_{i_0},...,y'_1,0,1,x'_2,...,x'_{N-i_0-1})(d\eta_2^0...d\eta_{N+1}^0) \bigg]dy'_1...dy'_{i_0}dx'_2...dx'_{N-i_0-1}.
    \end{split}
    \end{equation}

    Now we weight the law~\eqref{eq:pf-slegreen} by $\nu_{Y_0}({(y_1',0)})$ and   sample a point $\xi$ along the $(y_1',0)$ edge according to the probability measure $\nu_{Y_0}|_{(y_1',0)}^\#$ (take $y_1'=-\infty$ if $i_0=0$). Then the joint law of $(Y_0,\eta_1^0,...,\eta_{N}^0,y_{i_0}',...,y_1',\xi,\\x_2',...,x'_{N-i_0-1})$ is given by
    \begin{equation}\label{eq:pf-slegreen-1}
        \begin{split}
    c\cdot&\mathds{1}_{y'_{i_0}<...<y'_1<\xi<0<1<x'_2<...<x'_{N-i_0-1}}\bigg[\LF_\bbH^{(\beta,0),(\beta,\infty),(\beta_2,1),(\beta_2,x'_2),...,(\beta_2,x'_{N-i_0-1}),(\beta_2,y'_1),...,(\beta,y'_{i_0}), (\gamma,\xi)}(dY_0)\times \\ &M_{\alpha^0}(y'_{i_0},...,y'_1,0,1,x'_2,...,x'_{N-i_0-1})(d\eta_2^0...d\eta_{N+1}^0) \bigg]dy'_1...dy'_{i_0}dx'_2...dx'_{N-i_0-1}d\xi.
    \end{split}
    \end{equation}
    Consider the conformal map $f_\xi(z) = \frac{\xi-z}{\xi}$. Let $Y_1 = f_\xi\bullet_\gamma Y_0$, $\eta_j^1 = f_\xi\circ\eta_j^0$ for $j=2,...,N+1$.  Let $y_j'' = f_\xi(y_j')$ for $j=1,...,i_0$, $x_j'' = f_\xi(x_j')$ for $j = 2,...,N-i_0-1$, and $x_1'' = f_\xi(1)$. Then $(\bbH,Y_0,\eta_2^0,...,\eta_{N+1}^0,y_{i_0}',...,y_1',\xi,0,1,x_2',...,x'_{N-i_0-1})$ and $(\bbH,Y_1,\eta_2^1,...,\eta_{N+1}^1,y_{i_0}'',...,y_1'',0,1,x_1'',x_2'',...,x''_{N-i_0-1})$ represent the same quantum surface, while by Lemma~\ref{lm:lcft-H-conf}, Lemma~\ref{lm:green-dilation} and a change of variables, the latter tuple has law
    \begin{equation}\label{eq:pf-slegreen-2}
        \begin{split}
    c\cdot&\mathds{1}_{y''_{i_0}<...<y''_1<0<1<x_1''<x''_2<...<x''_{N-i_0-1}}\bigg[\LF_\bbH^{(\beta,1),(\beta,\infty),(\beta_2,x''_1),...,(\beta_2,x''_{N-i_0-1}),(\beta_2,y''_1),...,(\beta,y''_{i_0}), (\gamma,0)}(dY_1)\times \\ &M_{\alpha^0}(y''_{i_0},...,y''_1,1,x_1'',...,x''_{N-i_0-1})(d\eta_2^1...d\eta_{N+1}^1) \bigg]dy''_1...dy''_{i_0}dx''_1...dx''_{N-i_0-1}.
    \end{split}
    \end{equation}
    Here we have used the fact $\Delta_{\beta_2}+b_2=1$. By our definition, once we weld a quantum disk from $\QD_{2}$ along the $(0,1)$ edge, we obtain the picture $\Wd_\alpha(\QD^{N+2})$.

    On the other hand, consider the conformal welding of $\QD_{2}$ with a quantum triangle of weight $(4,2,4)$ embedded as $(\bbH,X,\eta_1,0,1,\infty)$ with the interface running from 0 to 1.  By Theorem~\ref{thm:disk+QT} and~\cite[Theorem 3]{SW05}, $\eta_1$ is an $\SLE_\kappa(\kappa-8)$ process from 0 to $\infty$ with force point at 1. Let $D_{\eta_1}$ be the unbounded component of $\bbH\backslash\eta_1$ and $\psi_{\eta_1}:D_{\eta_1}\to\bbH$ be the conformal map fixing 0,1,$\infty$. Also let $Y = \psi_{\eta_1}\bullet_\gamma X$. Let $D_{\eta_1}^L=\bbH\backslash D_{\eta_1}$ and $Z = (D_{\eta_1}^L,Y,0,1)$. Now we sample $\tilde x_1,...,\tilde x_{N-i_0-1},\tilde y_1,...,\tilde y_{i_0}$ from the measure
    $$1_{\tilde y_{i_0}<...<\tilde y_1<0<1<\tilde x_1<...<\tilde x_{N-i_0-1}}G_{\alpha^0}(\tilde y_{i_0},...,\tilde y_1,1,\tilde x_1,...,\tilde x_{N-i_0-1})\nu_Y(d\tilde y_{i_0})...\nu_Y(d\tilde y_1) \nu_Y(d\tilde x_1)...\nu_Y(d\tilde x_{N-i_0-1})$$
    and draw $\tilde \eta_2,...,\tilde\eta_{N+1}$ from the probability measure $M_{\alpha^0}^\#(\tilde y_{i_0},...,\tilde y_1,1,\tilde x_1,...,\tilde x_{N-i_0-1}).$ Then the joint law of $(Y,Z,\tilde y_{i_0},...,\tilde y_1,\tilde x_1,...,\tilde x_{N-i_0-1},\tilde \eta_2,...,\tilde\eta_{N+1})$ is 
    \begin{equation}\label{eq:pf-slegreen-3}
        \begin{split}
    c\cdot&\mathds{1}_{\tilde y_{i_0}<...<\tilde y_1<0<1<\tilde x_1<...<\tilde x_{N-i_0-1}}\int_0^\infty\bigg[\LF_{\bbH,\ell}^{(\gamma,0),(\beta,1),(\beta,\infty),(\gamma,\tilde x_1),...,(\gamma,\tilde x_{N-i_0-1}),(\gamma,\tilde y_1),...,(\gamma,\tilde y_{i_0})}(dY)\times\\& \Md_2(2;\ell)(dZ) \times M_{\alpha^0}(\tilde y_{i_0},...,\tilde y_1,1,\tilde x_1,...,\tilde x_{N-i_0-1})(d\tilde\eta_2^1...d\tilde\eta_{N+1}^1) \bigg]d\ell d\tilde y_1...d\tilde y_{i_0}d\tilde x_1...d\tilde x_{N-i_0-1},
    \end{split}
    \end{equation}
   where $\ell$ represents the quantum length $\nu_Y((0,1))$.  Let $x_j = \psi_{\eta_1}^{-1}(\tilde x_j)$ and $y_j = \psi_{\eta_1}^{-1}(\tilde y_j)$. Then the joint law of $(X,\eta_1,y_{i_0},...,y_1,x_1,...,x_{N-i_0-1})$  is 
    \begin{equation}\label{eq:pf-slegreen-4}
        \begin{split}
            c\cdot&\mathds{1}_{y_{i_0}<...<y_1<0<1<x_1<...<x_{N-i_0-1}}\bigg[\LF_\bbH^{(\beta,0),(\beta_2,1),(\beta,\infty),(\gamma,  x_1),...,(\gamma, x_{N-i_0-1}),(\gamma, y_1),...,(\gamma,  y_{i_0})}(dX)\times\\& G_{\alpha^0}\big(\psi_{\eta_1}(y_{i_0}),..., \psi_{\eta_1}(y_{1}),1,\psi_{\eta_1}(x_{1}),...,\psi_{\eta_1}(x_{N-i_0-1})\big)\SLE_\kappa(\kappa-8)(d\eta_1) \bigg] dy_{i_0}...dy_1dx_1..dx_{N-i_0-1}.
        \end{split}
    \end{equation}
    Now we weight the law of $(X,\eta_1,y_{i_0},...,y_1,x_1,...,x_{N-i_0-1})$ by $$\bigg(\prod_{j=1}^{i_0} \e^{\frac {\beta_2^2-\gamma^2}{4}}e^{\frac{(\beta_2-\gamma)X_\e(\tilde y_j) }{2} } \bigg)\cdot \bigg(\prod_{j=1}^{N-i_0-1} \e^{\frac {\beta_2^2-\gamma^2}{4}}e^{\frac{(\beta_2-\gamma)X_\e(\tilde x_j) }{2} } \bigg).$$
    Following Lemma~\ref{lm:lf-change-weight}, as we send $\e\to0$, the joint law of ~\eqref{eq:pf-slegreen-3} converges vaguely to 
    \begin{equation}\label{eq:pf-slegreen-5}
        \begin{split}
    c\cdot&\mathds{1}_{\tilde y_{i_0}<...<\tilde y_1<0<1<\tilde x_1<...<\tilde x_{N-i_0-1}}\int_0^\infty\bigg[\LF_{\bbH,\ell}^{(\gamma,0),(\beta,1),(\beta,\infty),(\beta_2,\tilde x_1),...,(\beta_2,\tilde x_{N-i_0-1}),(\beta_2,\tilde y_1),...,(\beta_2,\tilde y_{i_0})}(dY)\times\\& \Md_2(2;\ell)(dZ) \times M_{\alpha^0}(\tilde y_{i_0},...,\tilde y_1,1,\tilde x_1,...,\tilde x_{N-i_0-1})(d\tilde\eta_2...d\tilde\eta_{N+1}) \bigg]d\ell d\tilde y_1...d\tilde y_{i_0}d\tilde x_1...d\tilde x_{N-i_0-1},
    \end{split}
    \end{equation}
    which is the conformal welding of a curve-decorated surface from~\eqref{eq:pf-slegreen-2} with a sample from $\QD_{2}$ along the $(0,1)$ edge. Meanwhile, the joint law~\eqref{eq:pf-slegreen-4} converges vaguely to 
     \begin{equation}\label{eq:pf-slegreen-6}
        \begin{split}
            c\cdot&\mathds{1}_{y_{i_0}<...<y_1<0<1<x_1<...<x_{N-i_0-1}}\bigg[ \LF_\bbH^{(\beta,0),(\beta_2,1),(\beta,\infty),(\gamma,  x_1),...,(\gamma, x_{N-i_0-1}),(\gamma, y_1),...,(\gamma,  y_{i_0})}(dX)\times\\&\big( \prod_{j=1}^{i_0}\psi_{\eta_1}'(y_j)^{1-\Delta_{\beta_2}}\cdot \prod_{j=1}^{N-i_0-1}\psi_{\eta_1}'(x_j)^{1-\Delta_{\beta_2}}\big)G_{\alpha^0}(\psi_{\eta_1}(y_{i_0}),..., \psi_{\eta_1}(y_{1}),1,\psi_{\eta_1}(x_{1}),...,\psi_{\eta_1}(x_{N-i_0-1}))\\&\SLE_\kappa(\kappa-8)(d\eta_1) \bigg] dy_{i_0}...dy_1dx_1..dx_{N-i_0-1}.
        \end{split}
    \end{equation}
    Since $\psi_{\eta_1}-1$ is a constant multiple of $f_{\tau_1}$ (the centered Loewner map of $\eta_1$ at time $\tau_1$ when hitting 1), it follows from Lemma~\ref{lm:green-dilation} that if we let $\eta_j = \psi_{\eta_1}^{-1}\circ\tilde\eta_j$, the joint law of $(Y,y_{i_0},...,y_1,x_1,...,x_{N-i_0-1},\eta_1,...,\eta_{N+1})$ agrees with the left hand side of~\eqref{eq:thm-Green} (with $N$ replaced by $N+1$). Therefore we finish the proof for $i_2>i_1$ case.

    The $i_2<i_1$ case follows similarly. By our induction hypothesis, the surface and interface from $\Wd_{\alpha^0}(\QD^{N+1})$ can be embedded as $(\bbH,Y_0,\eta_2^0,...,\eta_{N+1}^0,y_{i_0}',...,y_2',0,1,x_1',...,x'_{N-i_0-1})$ whose law is
    \begin{equation}\label{eq:pf-slegreen-7}
\begin{split}
    c\cdot&\mathds{1}_{y_{i_0}'<...<y_2'<-1<0<x_1'<...<x_{N-i_0-1}'}\bigg[\LF_\bbH^{(\beta,0),(\beta,\infty),(\beta_2,1),(\beta_2,x_1'),...,(\beta_2,x_{N-i_0-1}'),(\beta_2,y_1'),...,(\beta,y_{i_0-1}')}(dY_0)\times \\ &M_{\alpha^0}(y_{i_0}',...,y_2',-1,0,x_1',...,x_{N-i_0-1}')(d\eta_2^0...d\eta_{N+1}^0)\bigg]dy_2'...dy_{i_0}'dx_1'...dx_{N-i_0-1}'.
    \end{split}
\end{equation}
We weight the law~\eqref{eq:pf-slegreen-7} by $\nu_{Y_0}((-1,0))$ and sample a point $\xi$ from $\nu_{Y_0}|_{(-1,0)}^\#$. Consider the conformal map $f_\xi(z) = \frac{\xi-z}{\xi}$, and let $y_j'' = f_\xi(y_j)$, $y_1'' = f_\xi(-1)$ and $x_j'' = f_\xi(x_j)$. Set $\eta_j^1 = f_\xi\circ\eta_j^0$ and $Y_1 = f_\xi\bullet_\gamma Y_0$. Then $(\bbH,Y_0,\eta_2^0,...,\eta_{N+1}^0,y_{i_0}',...,y_2',0,1,x_1',...,x'_{N-i_0-1})$ and $(\bbH,Y_1,\eta_2^1,...,\eta_{N+1}^1,y_{i_0}'',...,y_1'',0,1,x_1'',x_2'',...,x''_{N-i_0-1})$ represent the same quantum surface. Moreover by Lemma~\ref{lm:lcft-H-conf}, Lemma~\ref{lm:green-dilation} and a change of variables, the latter tuple has the same law as~\eqref{eq:pf-slegreen-2}.  The rest of the proof is the same as the previous case.

\end{proof}

\bibliographystyle{alpha}
\bibliography{theta}

\newcommand{\etalchar}[1]{$^{#1}$}
\begin{thebibliography}{CDCH{\etalchar{+}}14}

\bibitem[AB24]{aru2024sle}
Juhan Aru and Phil{\'e}mon Bordereau.
\newblock Sle and its partition function in multiply connected domains via the
  gaussian free field and restriction measures.
\newblock {\em arXiv preprint arXiv:2405.20148}, 2024.

\bibitem[ABK24]{alberts2024conformal}
Tom Alberts, Sung-Soo Byun, and Nam-Gyu Kang.
\newblock Conformal field theory of gaussian free fields in a multiply
  connected domain.
\newblock {\em arXiv preprint arXiv:2407.08220}, 2024.

\bibitem[AHS17]{AHS17}
Juhan Aru, Yichao Huang, and Xin Sun.
\newblock Two perspectives of the 2{D} unit area quantum sphere and their
  equivalence.
\newblock {\em Comm. Math. Phys.}, 356(1):261--283, 2017.

\bibitem[AHS23]{AHS20}
Morris Ang, Nina Holden, and Xin Sun.
\newblock Conformal welding of quantum disks.
\newblock {\em Electronic Journal of Probability}, 28:1--50, 2023.

\bibitem[AHS24]{AHS21}
Morris Ang, Nina Holden, and Xin Sun.
\newblock Integrability of {SLE} via conformal welding of random surfaces.
\newblock {\em Comm. Pure Appl. Math.}, 77(5):2651--2707, 2024.

\bibitem[AHSY23]{AHSY23}
Morris Ang, Nina Holden, Xin Sun, and Pu~Yu.
\newblock Conformal welding of quantum disks and multiple sle: the non-simple
  case.
\newblock {\em arXiv preprint arXiv:2310.20583}, 2023.

\bibitem[ARS22]{ARS22}
Morris Ang, Guillaume Remy, and Xin Sun.
\newblock The moduli of annuli in random conformal geometry.
\newblock {\em arXiv preprint arXiv:2203.12398}, 2022.

\bibitem[ARS23]{ARS21}
Morris Ang, Guillaume Remy, and Xin Sun.
\newblock {FZZ formula of boundary {Liouville CFT} via conformal welding}.
\newblock {\em J. Eur. Math. Soc.}, 2023.

\bibitem[ASY22]{ASY22}
Morris Ang, Xin Sun, and Pu~Yu.
\newblock Quantum triangles and imaginary geometry flow lines.
\newblock {\em arXiv preprint arXiv:2211.04580}, 2022.

\bibitem[AZ18]{zhan17g}
Mohammad A.Rezaei and Dapeng Zhan.
\newblock Green's functions for chordal {SLE} curves.
\newblock {\em Probability Theory and Related Fields}, 171:1093--1155, 2018.

\bibitem[BB04]{bauer2004conformal}
Michel Bauer and Denis Bernard.
\newblock Conformal transformations and the sle partition function martingale.
\newblock In {\em Annales Henri Poincare}, volume~5, pages 289--326. Springer,
  2004.

\bibitem[BBK05]{bauer2005multiple}
Michel Bauer, Denis Bernard, and Kalle Kyt{\"o}l{\"a}.
\newblock Multiple {S}chramm--{L}oewner evolutions and statistical mechanics
  martingales.
\newblock {\em Journal of statistical physics}, 120:1125--1163, 2005.

\bibitem[BM17]{BM17}
J\'{e}r\'{e}mie Bettinelli and Gr\'{e}gory Miermont.
\newblock Compact {B}rownian surfaces {I}: {B}rownian disks.
\newblock {\em Probab. Theory Related Fields}, 167(3-4):555--614, 2017.

\bibitem[BP21]{BP21}
Nathana{\"e}l Berestycki and Ellen Powell.
\newblock Gaussian free field, {L}iouville quantum gravity and {G}aussian
  multiplicative chaos.
\newblock {\em Lecture notes}, 2021.

\bibitem[BPW21]{beffara2021uniqueness}
Vincent Beffara, Eveliina Peltola, and Hao Wu.
\newblock On the uniqueness of global multiple {SLEs}.
\newblock {\em Annals of probability}, 49(1):400--434, 2021.

\bibitem[CDCH{\etalchar{+}}14]{chelkak2014convergence}
Dmitry Chelkak, Hugo Duminil-Copin, Cl{\'e}ment Hongler, Antti Kemppainen, and
  Stanislav Smirnov.
\newblock Convergence of {I}sing interfaces to {S}chramm's {SLE} curves.
\newblock {\em Comptes Rendus Mathematique}, 352(2):157--161, 2014.

\bibitem[Cer21]{cercle2021unit}
Baptiste Cercl{\'e}.
\newblock Unit boundary length quantum disk: a study of two different
  perspectives and their equivalence.
\newblock {\em ESAIM: Probability and Statistics}, 25:433--459, 2021.

\bibitem[DKRV16]{DKRV16}
Fran{\c{c}}ois David, Antti Kupiainen, R{\'e}mi Rhodes, and Vincent Vargas.
\newblock Liouville quantum gravity on the {R}iemann sphere.
\newblock {\em Communications in Mathematical Physics}, 342(3):869--907, 2016.

\bibitem[DMS21]{DMS14}
Bertrand Duplantier, Jason Miller, and Scott Sheffield.
\newblock Liouville quantum gravity as a mating of trees.
\newblock {\em Ast\'{e}risque}, 427, 2021.

\bibitem[DS11]{DS11}
Bertrand Duplantier and Scott Sheffield.
\newblock Liouville quantum gravity and {KPZ}.
\newblock {\em Inventiones mathematicae}, 185(2):333--393, 2011.

\bibitem[Dub05]{Dub05}
Julien Dub{\'e}dat.
\newblock {SLE} ($\kappa$, $\rho$) martingales and duality.
\newblock {\em The Annals of Probability}, 33(1):223--243, 2005.

\bibitem[Dub07]{dubedat2007commutation}
Julien Dub{\'e}dat.
\newblock Commutation relations for {S}chramm-{L}oewner evolutions.
\newblock {\em Communications on Pure and Applied Mathematics: A Journal Issued
  by the Courant Institute of Mathematical Sciences}, 60(12):1792--1847, 2007.

\bibitem[Dub09a]{Dub09}
J.~Dub\'{e}dat.
\newblock Duality of {S}chramm-{L}oewner {E}volutions.
\newblock {\em Ann. Sci. \'{E}c. Norm. Sup\'{e}r}, 42(5), 2009.

\bibitem[Dub09b]{dubedat09partition}
Julien Dub{\'e}dat.
\newblock {SLE} and the free field: partition functions and couplings.
\newblock {\em Journal of the American Mathematical Society}, 22(4):995--1054,
  2009.

\bibitem[FZ23]{Zhan22b}
Rami Fakhry and Dapeng Zhan.
\newblock Existence of multi-point boundary {G}reen's function for chordal
  {S}chramm-{L}oewner evolution ({SLE}).
\newblock {\em Electron. J. Probab.}, 28:Paper No. 46, 29, 2023.

\bibitem[GHS23]{GHS19}
Ewain Gwynne, Nina Holden, and Xin Sun.
\newblock Mating of trees for random planar maps and {L}iouville quantum
  gravity: a survey.
\newblock 59:41--120, 2023.

\bibitem[GM21]{gwynne2021convergence}
Ewain Gwynne and Jason Miller.
\newblock Convergence of percolation on uniform quadrangulations with boundary
  to {SLE}$_6$ on $\sqrt{8/3}$-liouville quantum gravity.
\newblock {\em Ast{\'e}risque}, 429:1--127, 2021.

\bibitem[Gra07]{graham2007multiple}
Kevin Graham.
\newblock On multiple {S}chramm--{L}oewner evolutions.
\newblock {\em Journal of Statistical Mechanics: Theory and Experiment},
  2007(03):P03008, 2007.

\bibitem[Gwy20]{gwynne2020random}
Ewain Gwynne.
\newblock Random surfaces and {L}iouville quantum gravity.
\newblock {\em Notices of the American Mathematical Society}, 67(4):484--491,
  2020.

\bibitem[HL21]{healey2021n}
Vivian~Olsiewski Healey and Gregory~F Lawler.
\newblock N-sided radial {S}chramm--{L}oewner evolution.
\newblock {\em Probability Theory and Related Fields}, 181(1-3):451--488, 2021.

\bibitem[HRV18]{HRV-disk}
Yichao Huang, R\'{e}mi Rhodes, and Vincent Vargas.
\newblock Liouville quantum gravity on the unit disk.
\newblock {\em Ann. Inst. Henri Poincar\'{e} Probab. Stat.}, 54(3):1694--1730,
  2018.

\bibitem[HS23]{HS19}
Nina Holden and Xin Sun.
\newblock Convergence of uniform triangulations under the {C}ardy embedding.
\newblock {\em Acta Math.}, 230(1):93--203, 2023.

\bibitem[JL18]{jahangoshahi2018multiple}
Mohammad Jahangoshahi and Gregory~F Lawler.
\newblock Multiple-paths $ {SLE}_\kappa $ in multiply connected domains.
\newblock {\em arXiv preprint arXiv:1811.05066}, 2018.

\bibitem[KL06]{kozdron2006configurational}
Michael~J Kozdron and Gregory~F Lawler.
\newblock The configurational measure on mutually avoiding {SLE} paths.
\newblock {\em arXiv preprint math/0605159}, 2006.

\bibitem[KP16]{kytola2016pure}
Kalle Kyt{\"o}l{\"a} and Eveliina Peltola.
\newblock Pure partition functions of multiple {SLE}s.
\newblock {\em Communications in Mathematical Physics}, 346:237--292, 2016.

\bibitem[Law08]{Law08}
Gregory~F Lawler.
\newblock {\em Conformally invariant processes in the plane}.
\newblock Number 114. American Mathematical Soc., 2008.

\bibitem[Law09]{lawler2009partition}
Gregory~F Lawler.
\newblock Partition functions, loop measure, and versions of {SLE}.
\newblock {\em Journal of Statistical Physics}, 134(5-6):813--837, 2009.

\bibitem[Law11]{lawler2011defining}
Gregory~F Lawler.
\newblock Defining {SLE} in multiply connected domains with the brownian loop
  measure.
\newblock {\em arXiv preprint arXiv:1108.4364}, 2011.

\bibitem[Law15]{Lawler2015g}
Gregory~F. Lawler.
\newblock Minkowski content of the intersection of a {S}chramm-{L}oewner
  evolution ({SLE}) curve with the real line.
\newblock {\em Journal of the Mathematical Society of Japan}, 67(4):1631--1669,
  2015.

\bibitem[LG13]{LeGall13}
Jean-Fran\c{c}ois Le~Gall.
\newblock Uniqueness and universality of the {B}rownian map.
\newblock {\em Ann. Probab.}, 41(4):2880--2960, 2013.

\bibitem[LR15]{LR15Green}
Gregory~F. Lawler and Mohammad~A. Rezaei.
\newblock Minkowski content and natural parameterization for the
  {S}chramm-{L}oewner evolution.
\newblock {\em Ann. Probab.}, 43(3):1082--1120, 2015.

\bibitem[LSW03]{LSW03}
Gregory Lawler, Oded Schramm, and Wendelin Werner.
\newblock Conformal restriction: the chordal case.
\newblock {\em Journal of the American Mathematical Society}, 16(4):917--955,
  2003.

\bibitem[LSW11]{lawler2011conformal}
Gregory~F Lawler, Oded Schramm, and Wendelin Werner.
\newblock Conformal invariance of planar loop-erased random walks and uniform
  spanning trees.
\newblock In {\em Selected Works of Oded Schramm}, pages 931--987. Springer,
  2011.

\bibitem[LW21]{liu2020scaling}
Mingchang Liu and Hao Wu.
\newblock Scaling limits of crossing probabilities in metric graph {GFF}.
\newblock {\em Electron. J. Probab.}, 26:Paper No. 37, 46, 2021.

\bibitem[MS16a]{MS16a}
J.~Miller and S.~Sheffield.
\newblock Imaginary {G}eometry {I}: Interacting {SLE}s.
\newblock {\em Probability Theory and Related Fields}, 164(3-4):553--705, 2016.

\bibitem[MS16b]{MS16b}
Jason Miller and Scott Sheffield.
\newblock {Imaginary geometry II: Reversibility of
  {SLE}$_\kappa(\rho_1;\rho_2)$ for $\kappa\in (0, 4) $}.
\newblock {\em The Annals of Probability}, 44(3):1647--1722, 2016.

\bibitem[MS17]{ig4}
Jason Miller and Scott Sheffield.
\newblock Imaginary geometry {IV}: interior rays, whole-plane reversibility,
  and space-filling trees.
\newblock {\em Probab. Theory Related Fields}, 169(3-4):729--869, 2017.

\bibitem[MSW19]{MSW19}
Jason Miller, Scott Sheffield, and Wendelin Werner.
\newblock Non-simple {SLE} curves are not determined by their range.
\newblock {\em Journal of the European Mathematical Society}, 22(3):669--716,
  2019.

\bibitem[MT09]{Meyn-Tweedie}
Sean Meyn and Richard~L. Tweedie.
\newblock {\em Markov chains and stochastic stability}.
\newblock Cambridge University Press, Cambridge, second edition, 2009.
\newblock With a prologue by Peter W. Glynn.

\bibitem[MW18]{miller2018connection}
Jason Miller and Wendelin Werner.
\newblock Connection probabilities for conformal loop ensembles.
\newblock {\em Communications in Mathematical Physics}, 362:415--453, 2018.

\bibitem[Pel19]{peltola2019toward}
Eveliina Peltola.
\newblock Toward a conformal field theory for {S}chramm-{L}oewner evolutions.
\newblock {\em Journal of Mathematical Physics}, 60(10):103305, 2019.

\bibitem[Pol81]{polyakov1981quantum}
Alexander~M Polyakov.
\newblock Quantum geometry of bosonic strings.
\newblock {\em Physics Letters B}, 103(3):207--210, 1981.

\bibitem[PW19]{peltola2019global}
Eveliina Peltola and Hao Wu.
\newblock Global and local multiple {SLE}s for $\kappa \le4$ and connection
  probabilities for level lines of {GFF}.
\newblock {\em Communications in Mathematical Physics}, 366(2):469--536, 2019.

\bibitem[Qia18]{Qian-hSLE}
Wei Qian.
\newblock Conformal restriction: the trichordal case.
\newblock {\em Probab. Theory Related Fields}, 171(3-4):709--774, 2018.

\bibitem[Qia21]{Qian-hSLE2}
Wei Qian.
\newblock Generalized disconnection exponents.
\newblock {\em Probab. Theory Related Fields}, 179(1-2):117--164, 2021.

\bibitem[Sch00]{Sch00}
Oded Schramm.
\newblock Scaling limits of loop-erased random walks and uniform spanning
  trees.
\newblock {\em Israel Journal of Mathematics}, 118(1):221--288, 2000.

\bibitem[She07]{She07}
S.~Sheffield.
\newblock Gaussian free fields for mathematicians.
\newblock {\em Probab. Theory Related Fields}, 139, 2007.

\bibitem[She16]{She16a}
Scott Sheffield.
\newblock Conformal weldings of random surfaces: {SLE} and the quantum gravity
  zipper.
\newblock {\em The Annals of Probability}, 44(5):3474--3545, 2016.

\bibitem[She22]{She22}
Scott Sheffield.
\newblock What is a random surface?
\newblock {\em arXiv preprint arXiv:2203.02470}, 2022.

\bibitem[Smi01]{smirnov2001critical}
Stanislav Smirnov.
\newblock Critical percolation in the plane: conformal invariance, {C}ardy's
  formula, scaling limits.
\newblock {\em Comptes Rendus de l'Acad{\'e}mie des Sciences-Series
  I-Mathematics}, 333(3):239--244, 2001.

\bibitem[SS09]{schramm2009contour}
Oded Schramm and Scott Sheffield.
\newblock Contour lines of the two-dimensional discrete {G}aussian free field.
\newblock {\em Acta Math}, 202:21--137, 2009.

\bibitem[SS13]{SS13}
O.~Schrmm and S.~Sheffield.
\newblock A contour line of the continuum {G}aussian free field.
\newblock {\em Probab. Theory Related Fields}, 157, 2013.

\bibitem[SW05]{SW05}
Oded Schramm and David~B. Wilson.
\newblock S{LE} coordinate changes.
\newblock {\em New York J. Math.}, 11:659--669, 2005.

\bibitem[SW16]{SW16}
Scott Sheffield and Menglu Wang.
\newblock {Field-measure correspondence in Liouville quantum gravity almost
  surely commutes with all conformal maps simultaneously}.
\newblock {\em arXiv preprint arXiv:1605.06171}, 2016.

\bibitem[Var17]{vargas-dozz}
Vincent Vargas.
\newblock Lecture notes on {L}iouville theory and the {DOZZ} formula.
\newblock {\em arXiv preprint arXiv:1712.00829}, 2017.

\bibitem[Wu20]{wu2020hypergeometric}
Hao Wu.
\newblock Hypergeometric {SLE}: conformal {M}arkov characterization and
  applications.
\newblock {\em Communications in Mathematical Physics}, 374(2):433--484, 2020.

\bibitem[Yu23]{Yu22}
Pu~Yu.
\newblock Time-reversal of multiple-force-point chordal {${\rm
  SLE}_\kappa(\underline\rho)$}.
\newblock {\em Electron. J. Probab.}, 28:Paper No. 140, 19, 2023.

\bibitem[Zha10]{Zhan-hSLE}
Dapeng Zhan.
\newblock Reversibility of some chordal {${\rm SLE}(\kappa;\rho)$} traces.
\newblock {\em J. Stat. Phys.}, 139(6):1013--1032, 2010.

\bibitem[Zha20]{zhan182green}
Dapeng Zhan.
\newblock Two-curve {G}reen's function for 2-{SLE}: the interior case.
\newblock {\em Comm. Math. Phys.}, 375(1):1--40, 2020.

\bibitem[Zha22a]{Zhan21}
Dapeng Zhan.
\newblock Boundary {G}reen's functions and {M}inkowski content measure of
  multi-force-point {${\rm SLE}_\kappa(\underline\rho)$}.
\newblock {\em Stochastic Process. Appl.}, 151:265--306, 2022.

\bibitem[Zha22b]{Zhan19a}
Dapeng Zhan.
\newblock Time-reversal of multiple-force-point {${\rm
  SLE}_\kappa(\underline\rho)$} with all force points lying on the same side.
\newblock {\em Ann. Inst. Henri Poincar\'{e} Probab. Stat.}, 58(1):489--523,
  2022.

\bibitem[Zha24]{Zhan23}
Dapeng Zhan.
\newblock Existence and uniqueness of nonsimple multiple {SLE}.
\newblock {\em J. Stat. Phys.}, 191(8):Paper No. 101, 15, 2024.

\end{thebibliography}

\end{document}